\documentclass[10pt]{article}

\usepackage{amssymb,amsmath,amsthm,amsfonts}
\usepackage{graphicx}
\usepackage{algorithm,algorithmic}
\graphicspath{{./pics/},{./figs/}}
\usepackage{epstopdf}
\usepackage{color}
\usepackage{indentfirst}
\usepackage[active]{srcltx}

\numberwithin{equation}{section}
\newtheorem{theorem}{Theorem}[section]
\newtheorem{remark}{Remark}[section]

\newtheorem{lemma}{Lemma}[section]

\newtheorem{example}{Example}[section]

\def\cT {\mathcal T}

\begin{document}
\title{A Convergent Adaptive Finite Element Method \\
for Electrical Impedance Tomography}
\author{
\and Bangti Jin\footnote{Department of Computer Science, University College London, Gower Street, London WC1E 6BT, UK (b.jin@ucl.ac.uk, bangti.jin@gmail.com)}
\and Yifeng Xu\footnote{Department of Mathematics, Scientific Computing Key Laboratory of Shanghai Universities and E-Institute for Computational Science of Shanghai Universities, Shanghai Normal University, Shanghai 200234, China. (yfxuma@aliyun.com)}
\and Jun Zou\footnote{Department of Mathematics, The Chinese University of Hong Kong, Shatin, New Territories, Hong Kong. (zou@math.cuhk.edu.hk)}
}

\date{}

\maketitle
\begin{abstract}
{In this work we develop and analyze an adaptive finite element method for efficiently solving
electrical impedance tomography -- a severely ill-posed nonlinear inverse problem
for recovering the conductivity from boundary voltage measurements. The reconstruction technique is
based on Tikhonov regularization with a Sobolev smoothness penalty and discretizing the forward model using
continuous piecewise linear finite elements. We derive an adaptive finite element algorithm with an a
posteriori error estimator involving the concerned state and adjoint variables and the recovered conductivity.
The convergence of the algorithm is established, in the sense that the sequence of discrete
solutions contains a convergent subsequence to a solution of the optimality system for the continuous
formulation. Numerical results are presented to verify the convergence and efficiency of the
algorithm.} \smallskip

\end{abstract}

\section{Introduction}\label{sect:intro}
Electrical impedance tomography (EIT) is a diffusive imaging modality for probing internal structures of the
concerned object, by recovering its electrical conductivity/permittivity distribution from voltage measurements on the
boundary. One typical experimental setup is as follows. One first attaches a set of metallic electrodes to the surface of
the object, then injects an input current into the object through these electrodes, which induces an electromagnetic
field inside the object. Last, one measures the induced electric voltages on the electrodes. The procedure is usually repeated
several times with different input currents in order to yield sufficient information about the sought-for conductivity distribution.
In many applications, the physical process can be most accurately described by the complete electrode model (CEM)
\cite{ChengIsaacsonNewellGisser:1989,SomersaloCheneyIsaacson:1992}. The imaging modality has attracted considerable interest in  medical imaging,
geophysical prospecting, nondestructive evaluation and pneumatic oil pipeline conveying etc.

A number of reconstruction algorithms have been proposed for the EIT inverse problem;
see, e.g., \cite{LechleiterHyvonen:2008,AdGaLi:2011,JinKhanMaass:2012,JinMaass:2010,LechleiterRieder:2006,WinklerRieder:2014,
KnudsenLassasMueller:2009,HarrachUllrich:2013,GehreJin:2014,ChowItoZou:2014,DunlopStuart:2015} for an incomplete list. One prominent idea underlying
 existing imaging algorithms is regularization, especially Tikhonov regularization with a smoothness or sparsity type
penalty, and they have demonstrated encouraging results with experimental data. In practice,
they are customarily implemented using the continuous piecewise linear finite element method (FEM),
due to its flexibility in handling variable coefficients and general geometry. Despite its popularity, it was only
rigorously justified recently in \cite{GehreJinLu:2014} for the CEM on either polygonal or smooth convex domains.

The accuracy of the CEM relies crucially on the use of nonstandard boundary conditions for capturing important
characteristics of the physical experiment, notably contact impedance effect. As a consequence, around the boundary
of the electrodes, the boundary condition changes from the Neumann to Robin type, which induces weak
singularity of the forward solution around the interface \cite{Grisvard:1985}. The low-regularity
of the sought-for conductivity field, as enforced by Sobolev smoothness penalty, will possibly also induce weak solution singularities.
With a quasi-uniform triangulation of the domain, the solution singularities are not effectively resolved and the errors around
electrode edges and discontinuity interfaces are dominant, which can potentially compromise the reconstruction accuracy greatly, if done
inadvertently. This naturally motivates the use of an adaptive strategy to achieve the desired accuracy with
reduced computational complexity. In this work, we shall develop a novel adaptive finite
element method (AFEM) for the EIT inverse problem and analyze its convergence.

Generally the AFEM generates a sequence of nested triangulations and discrete solutions by the following successive loop:
\begin{equation}\label{afem_loop}
  \mbox{SOLVE}\rightarrow\mbox{ESTIMATE}\rightarrow\mbox{MARK}\rightarrow\mbox{REFINE}.
\end{equation}
The key ingredient in the procedure is the module \texttt{ESTIMATE}, which consists of computing a posteriori error estimators, i.e., computable
quantities from the discrete solution, the local mesh size and other given data. This has been thoroughly studied for
forward problems; see, e.g., \cite{AinsworthOden:2000,Ver:1996}. Over the past few decades, there are also many important works on
the a posteriori error analysis of PDE-constrained optimal control problems; see \cite{HinterHoppe:2010,HinterHoppeIliashKieweg:2008,LiLiuMaTang:2002,LiuYan:2001,BeckerMao:2011} for a very incomplete list. 
In particular, Becker and Mao \cite{BeckerMao:2011} showed the
quasi-optimality of the AFEM for an optimal control problem with control constraints. However, the behavior of inverse problems such as EIT is quite different
from that of optimal control problems due to the ill-posed nature, the {presence} of the data noise and high-degree nonlinearity.

The adaptive idea, including the AFEM, has started to attract some attention in
the context of inverse problems in recent years. In \cite{BeckerVexler:2004,BeilinaClason:2006,BeilinaJohnson:2005},
the AFEM using a dual weighted residual framework was studied for parameter identification problems, and high order terms in relevant Lagrangian
functionals were ignored. Feng et al \cite{FengYanLiu:2008} proposed a residual-based estimator for state, costate (adjoint) and parameter
by assuming convexity of the cost functional and high regularity on the parameter. Li et al \cite{LiXieZou:2011} derived rigorous a posteriori
error estimators for reconstructing the distributed flux under a practical regularity assumption, in the sense that like for forward problems,
the errors of the state variable, the adjoint variable and the flux are bounded from above and below by multiples of the estimators. In
a series of interesting works \cite{BeilinaKlibanov:2010a,BeilinaKlibanov:2010b,BeilinaKlibanovKokurin:2010}, Beilina et al adopted the
AFEM for hyperbolic coefficient inverse problems.
Kaltenbacher et al \cite{GriesbaumKaltenbacherVexler:2008,KaltenbacherKirchnerVeljovic:2014} described and analyzed adaptive strategies for
choosing the regularization parameter in Tikhonov regularization and iterative regularization techniques, e.g., Gauss-Newton methods.
Unlike the AFEM for forward problems, for which the convergence and computational complexity have been systematically studied (see
the survey papers \cite{CFPP:2014,NSV:2009}), the theoretical analysis of the AFEM for inverse problems is still in its
infancy. Recently, Xu and Zou \cite{XuZou:2015a,XuZou:2013b} established the convergence of the AFEM for recovering the flux and the Robin
coefficient. We remark that the convergence rate and optimality of the AFEM in the context of nonlinear inverse
problems are completely open, due to inherent nonconvexity of the functional, and lack of precise regularity results of the
minimizers to the nonlinear optimization problem. Nonetheless, our convergence result in Theorem \ref{thm:conv_alg}
provides some theoretical justifications of the AFEM for the EIT inverse problem.

In this paper, we develop a novel AFEM for the EIT {based on Tikhonov regularization with a $H^1(\Omega)$ seminorm
penalty} and analyze its convergence. The AFEM is of the standard form \eqref{afem_loop}: it does not require the interior
node property in the module \texttt{REFINE}, and hence it is easy to implement. The derivation of
a posteriori error estimators is constructive: it lends itself to a route for convergence analysis. The analysis
relies on a limiting output least-squares problem defined on the closure of adaptively generated finite
element spaces, and it consists of the following two steps. First, the sequence of discrete minimizers is shown in Section
\ref{ssec:limitingproblem} to contain a subsequence converging to a solution of the limiting problem, and then
the limiting minimizer and related state and adjoint variables are proved in Section \ref{ssec:convergence}
to satisfy the necessary optimality system of the continuous Tikhonov functional.

This work is a continuation of our prior work \cite{GehreJinLu:2014} on the FEM analysis of EIT, but differs
from the latter considerably in several aspects. The major effort of \cite{GehreJinLu:2014} was to justify the convergence of
the \textit{quasiuniform} FEM approximation of the Tikhonov formulation of the EIT, and no a posteriori error estimator
and adaptive method were studied, which is the main goal of the present work. The convergence analysis in
\cite{GehreJinLu:2014} relies crucially on the $W^{1,q}(\Omega)$ ($q>2$) regularity of the forward solution and
the density of FE spaces $V_h$ in $H^1(\Omega)$. The density does not hold generally for adaptively generated
FE spaces. Hence, the analysis in \cite{GehreJinLu:2014} does not carry over to the AFEM directly. In this work, we shall
adopt a strategy developed in \cite{XuZou:2013b} for recovering the Robin coefficient from the Cauchy data to overcome these technical
difficulties. Nonetheless, there are major differences in the analysis due to higher degree of
nonlinearity of the EIT problem. In \cite{XuZou:2013b}, the continuity
of the parameter-to-state map from $L^2(\Gamma_i)$ to $L^2(\Gamma_c)$ plays a crucial role. For
the EIT, only the $H^1(\Omega)$ weak continuity of the forward map holds (cf. Lemma
\ref{lem:weakcontmedpoly}), and we shall exploit the pointwise convergence of discrete minimizers
and Lebesgue's dominated convergence theorem.
This allows us to establish the $H^1(\Omega)$ convergence of discrete state variables
(cf. Theorem \ref{thm:conver_medmin}), and thus enables  us to verify that
the limiting solution also satisfies the optimality system of the continuous
functional (Lemmas \ref{lem:vp_mc} and \ref{lem:gat_mc}).

The rest of this paper is organized as follows. In Section \ref{sect:prelim}, we describe the CEM,
regularized least-squares formulation and its necessary optimality system. The finite element discretization
is described, and an  adaptive FEM algorithm for the EIT is proposed in Section \ref{sect:polydom}, where a
heuristic yet constructive derivation is also provided. The convergence analysis of the adaptive algorithm is given in Section
\ref{sect:conv}. Some numerical results are given in Section \ref{sec:numer}
to illustrate its convergence and efficiency. We conclude the section with some notation. We shall use the
standard notation for Sobolev spaces, following \cite{Evans:1992a}. Further, we use $\langle\cdot,\cdot\rangle$ and
$(\cdot,\cdot)$ to denote the inner product on the Euclidean space and $(L^2(\Omega))^d$, respectively, by $\|\cdot\|$
the Euclidean norm,  and occasionally abuse $\langle\cdot,\cdot\rangle$ for the duality pairing between the space
$\mathbb{H}$ and its dual space. Throughout, the notation $c$ denotes
a generic constant, which may differ at each occurrence, but is always
independent of the mesh size and other quantities of interest.

\section{Preliminaries}\label{sect:prelim}
We shall recall in this section
the mathematical model for the EIT problem and describe the reconstruction technique based on
Tikhonov regularization and its necessary optimality system.

\subsection{Complete electrode model}\label{subsect:cem}
Let $\Omega$ be an open bounded domain in $\mathbb{R}^{d}$ $(d=2,3)$ with a polyhedral boundary $\Gamma$.
We denote the set of electrodes by $\{e_l\}_{l=1}^L$, which are line segments/planar surfaces on
$\Gamma$ and disjoint from each other, i.e., $\bar{e}_i\cap\bar{e}_k=\emptyset$
if $i\neq k$. The applied current on the $l$th electrode $e_l$ is denoted by $I_l$, and the current vector
$I=(I_1,\ldots,I_L)^\mathrm{t}$ satisfies $\sum_{l=1}^LI_l=0$ by the law of charge conservation.
Let the space $\mathbb{R}_\diamond^L$ be the subspace of the vector space $\mathbb{R}^L$ with zero mean.
Then we have $I\in\mathbb{R}_\diamond^L$. The electrode voltage $U=(U_1,\ldots,U_L)^\mathrm{t}$ is
also normalized such that $U\in\mathbb{R}_\diamond^L$. Then the
CEM reads: given the conductivity $\sigma$, positive contact
impedances $\{z_l\}_{l=1}^L$ and input current $I\in\mathbb{R}_\diamond^L$, find the
potential $u\in H^1(\Omega)$ and electrode voltage $U\in\mathbb{R}_\diamond^L$ such that
\begin{equation}\label{eqn:cem}
\left\{\begin{aligned}
\begin{array}{ll}
-\nabla\cdot(\sigma\nabla u)=0 & \mbox{ in }\Omega,\\[1ex]
u+z_l\sigma\frac{\partial u}{\partial n}=U_l& \mbox{ on } e_l, l=1,2,\ldots,L,\\[1ex]
\int_{e_l}\sigma\frac{\partial u}{\partial n}ds =I_l& \mbox{ for } l=1,2,\ldots, L,\\ [1ex]
\sigma\frac{\partial u}{\partial n}=0&\mbox{ on } \Gamma\backslash\cup_{l=1}^Le_l.
\end{array}
\end{aligned}\right.
\end{equation}

The physical motivation behind the model \eqref{eqn:cem} is as follows. The governing equation is
derived under a quasi-static assumption on the electromagnetic process. The second
line describes the contact impedance effect: When injecting electrical currents into the object, a highly
resistive thin layer forms at the electrode-electrolyte interface,
which causes potential drops across the electrode-electrolyte interface. The potential drop is described by Ohm's law,
with proportionality factors $\{z_l\}_{l=1}^L$.
It also takes into account the fact that metallic electrodes are perfect conductors, and hence the voltage $U_l$ is
constant on each electrode. The third line reflects the fact that the current $I_l$ injected through the
electrode $e_l$ is completely confined to $e_l$ itself. The nonstandard boundary conditions
is essential for the model \eqref{eqn:cem} to reproduce experimental data within the measurement precision
 \cite{ChengIsaacsonNewellGisser:1989,SomersaloCheneyIsaacson:1992}.

Due to physical constraint, the conductivity distribution $\sigma$ is naturally bounded both from
below and above by positive constants. Hence we introduce the following admissible set $\mathcal{A}$:
for some $\lambda\in (0,1)$, let
\begin{equation*}
 \mathcal{A}=\{{\lambda \in H^1(\Omega)}: \lambda\leq \sigma(x)\leq \lambda^{-1}\mbox{ a.e. } x\in\Omega\}.
\end{equation*}
The set $\mathcal{A}$ is endowed with the $H^1(\Omega)$-norm, in view of
the $H^1(\Omega)$-seminorm regularization, cf. \eqref{eqn:tikh} below. Further, we denote by $\mathbb{H}$ the product
space $H^1(\Omega)\otimes \mathbb{R}_\diamond^L$ with its norm defined by
\begin{equation*}
  \|(u,U)\|_{\mathbb{H}}^2 = \|u\|_{H^1(\Omega)}^2 + \|U\|^2.
\end{equation*}

A convenient equivalent norm on the space $\mathbb{H}$ is given below.
\begin{lemma}\label{lem:normequiv}
On the space $\mathbb{H}$, the norm $\|\cdot\|_\mathbb{H}$ is equivalent to the norm
$\|\cdot\|_{\mathbb{H},*}$ defined by
\begin{equation*}
  \|(u,U)\|_{\mathbb{H},*}^2 = \|\nabla u\|_{L^2(\Omega)}^2 + \sum_{l=1}^L\|u-U_l\|_{L^2(e_l)}^2.
\end{equation*}
\end{lemma}
\begin{proof}{
The lemma is a folklore result in the EIT community, and we provide a proof only for completeness.
It is easy to verify that $\|(u,U)\|_{\mathbb{H},*}$ indeed defines a proper norm.
It suffices to show the following two inequalities: there exist $c_1,c_2>0$ such that
\begin{equation*}
  c_1\|(u,U)\|_{\mathbb{H}} \leq \|(u,U)\|_{\mathbb{H},\ast}\leq c_2\|(u,U)\|_{\mathbb{H}}.
\end{equation*}
The second inequality follows from the Cauchy-Schwarz inequality and
trace theorem. We show the first inequality by contradiction. Assume the
contrary. Then there exists a sequence $\{(u^n,U^n)\}\subset \mathbb{H}$ such
that $\|(u^n,U^n)\|_{\mathbb{H}}=1$ and $\|(u^n,U^n)\|_{\mathbb{H},*}<n^{-1}$.
Then there exists a convergent subsequence, also denoted by $\{u^n\}$, to some $u\in H^1(\Omega)$
weakly in $H^1(\Omega)$. By the compact embedding of $H^1(\Omega)$ into $L^2(\Omega)$, the
sequence $\{u^n\}$ converges to $u$ in $L^2(\Omega)$. Further, by construction,
$\|\nabla u^n\|_{L^2(\Omega)}\leq n^{-1}$. Thus $\{u^n\}$ converges to $u$ in
$H^1(\Omega)$, and $u=c$ in the domain $\Omega$ for some $c\in\mathbb{R}$. By trace theorem
and Sobolev embedding theorem, $\{u^n\}$ converges to $u$ in $L^2(\Gamma)$. Since
$\|u^n-U_l^n\|_{L^2(e_l)}<n^{-1}$, $\{U_l^n\}$ converges to the trace of $u$ on $e_l$
for each $l=1,2,\ldots,L$, i.e., the limit $U=c(1,\ldots,1)^T$. Now the condition $U\in \mathbb{R}^L_\diamond$
implies $U=0$, $c=0$ and $u\equiv0$. Consequently, $u^n\to0$ in $H^1(\Omega)$ and $U^n\to0$
in $\mathbb{R}^L$, which contradicts the assumption $\|(u^n,U^n)\|_{\mathbb{H}}=1$.}
\end{proof}

The weak formulation of the model \eqref{eqn:cem} reads: find $(u,U)\in \mathbb{H}$ such that
\begin{equation}\label{eqn:cemweakform}
  a(\sigma,(u,U),(v,V)) = \langle I,V\rangle \quad \forall (v,V)\in \mathbb{H},
\end{equation}
where the trilinear form $a(\sigma,(u,U),(v,V))$ on $\mathcal{A}\times\mathbb{H}\times\mathbb{H}$ is defined by
\begin{equation*}
   a(\sigma,(u,U),(v,V)) = (\sigma \nabla u ,\nabla v ) +\sum_{l=1}^Lz_l^{-1}(u-U_l,v-V_l)_{L^2(e_l)},
\end{equation*}
where $(\cdot,\cdot)_{L^2(e_l)}$ denotes the $L^2(e_l)$ inner product. {By Lemma \ref{lem:normequiv},
for any $\sigma\in\mathcal{A}$, the bilinear form $a(\sigma,\cdot,\cdot)$ is continuous and coercive
on the space $\mathbb{H}$.} Hence, by Lax-Milgram theorem, for any fixed $\sigma\in\mathcal{A}$ and
given contact impedances $\{z_l\}_{l=1}^L$ and current $I\in \Sigma_\diamond^L$, there exists a unique
solution $(u,U)\equiv(u(\sigma),U(\sigma))\in\mathbb{H}$ to \eqref{eqn:cemweakform}, and it depends
continuously on the input current pattern $I$. {Since $\sigma\in\mathcal{A}$, one can
deduce that $u\in W^{1,q}(\Omega)$ for some $q>2$ \cite{JinMaass:2010}.} See also \cite{JinMaass:2010,GehreJinLu:2014,DunlopStuart:2015} for various continuity results of $(u,U)$ with respect to the conductivity $\sigma$.

{
\begin{remark}
Alternatively, one can formulate a proper variational formulation of the CEM \eqref{eqn:cem} on the
quotient space $\dot{\mathbb{H}}=(H^1(\Omega)\times \mathbb{R}^L)/\mathbb{R}$, with the norm defined by
\begin{equation*}
  \|(u,U)\|_{\dot{\mathbb{H}}}=\inf_{c\in\mathbb{R}}(\|u-c\|^2_{H^1(\Omega)}+\|U-c\|^2)^{1/2}.
\end{equation*}
Then the bilinear form $a(\sigma,\cdot,\cdot)$ is continuous and coercive on the space $\dot{\mathbb{H}}$;
see \cite{SomersaloCheneyIsaacson:1992} for details. It differs from the preceding one in the grounding
condition: in the choice $\mathbb{H}$, the grounding is enforced by the zero mean condition
$U\in\mathbb{R}_\diamond^L$.
\end{remark}
}
\subsection{Tikhonov regularization}\label{subsect:tikh}

The inverse problem is to reconstruct the conductivity $\sigma$ from noisy measurements $U^\delta$
of the exact electrode voltage $U(\sigma^\dagger)$, corresponding to one or multiple input currents, with
a noise level $\delta$:
\begin{equation*}
  \|U^\delta-U(\sigma^\dagger)\|\leq \delta.
\end{equation*}
It is severely ill-posed in the sense that small errors
in the data can lead to very large deviations in the reconstructions. Hence, some sort of regularization
is beneficial, and it is incorporated into imaging algorithms, either implicitly or explicitly, in order to yield
physically meaningful images. One prominent idea behind many existing imaging algorithms is Tikhonov
regularization, which minimizes the following functional
\begin{equation}\label{eqn:tikh}
   \min_{\sigma\in\mathcal{A}}\left\{J(\sigma) = \tfrac{1}{2} \|U(\sigma)-U^\delta\|^2 + \tfrac{\alpha}{2}\|\nabla\sigma\|_{L^2(\Omega)}^2\right\},
\end{equation}
and then takes the minimizer as an approximation to the true conductivity $\sigma^\dagger$.
The first term in the functional $J$ integrates the information in the data $U^\delta$. For notational simplicity,
we consider only one dataset in the discussion, and the adaptation to multiple datasets is straightforward. The
second term imposes a priori regularity assumption (Sobolev smoothness) on the expected conductivity $\sigma$. The scalar
$\alpha>0$ is known as a regularization parameter, and controls the tradeoff between the two terms \cite{ItoJin:2014}.
Problem \eqref{eqn:tikh} has at least one minimizer, and it depends continuously on the data perturbation
\cite{JinMaass:2010}. {The convergence of the Tikhonov minimizer to $\sigma^\dag$ as the noise level $\delta$
tends to zero was shown in \cite{JinMaass:2010}, if the true conductivity $\sigma^\dag\in H^1(\Omega)$, and
also a convergence rate $O(\delta^{1/2})$ was given under suitable source condition as $\delta\to0$, both under a
proper choice of regularization parameter $\alpha$.}

{Following the standard adjoint technique (see, e.g., \cite{ItoKunisch:2008}),} we introduce the following
adjoint problem for \eqref{eqn:cemweakform}: find $(p,P)\equiv(p(\sigma),P(\sigma))
\in\mathbb{H}$ such that
\begin{equation}\label{eqn:cem-adj}
  a(\sigma,(p,P),(v,V)) = \langle U(\sigma)-U^\delta,V\rangle \quad \forall (v,V)\in \mathbb{H}.
\end{equation}
Then it can be verified that the G\^{a}teaux derivative of $J(\sigma)$ at $\sigma\in\mathcal{A}$
in the direction $\mu$ is given by
\begin{equation*}
  J'(\sigma)[\mu] = (\alpha\nabla\sigma,\nabla\mu)- (\mu \nabla u(\sigma),\nabla p(\sigma)).
\end{equation*}
Then the minimizer $\sigma^*$ to problem \eqref{eqn:tikh} and the respective forward solution
$(u^*,U^*)$ and the adjoint solution $(p^*,P^*)$ satisfies the following necessary optimality system:
\begin{equation}\label{eqn:cem-optsys}
   \begin{aligned}
     &a(\sigma^*,(u^*,U^*),(v,V)) = \langle I,V\rangle \quad \forall (v,V)\in \mathbb{H},\\
     &a(\sigma^*,(p^*,P^*),(v,V)) = \langle U^*-U^\delta,V\rangle       \quad\forall(v,V)\in\mathbb{H},\\
     &        \alpha (\nabla\sigma^{\ast},\nabla(\mu-\sigma^{\ast})) - ((\mu-\sigma^{\ast})\nabla u^{\ast},\nabla p^{\ast})\geq
            0\quad\forall \mu\in\mathcal{A},
    \end{aligned}
\end{equation}
where the variational inequality at the last line corresponds to the box constraint in the admissible set $\mathcal{A}$.

\section{Adaptive finite element method}\label{sect:polydom}

Now we describe the finite element method (FEM) for discretizing problem \eqref{eqn:tikh}, derive
the a posteriori error estimator and develop a novel adaptive algorithm, which uses a general
marking strategy and thus is easy to implement. The convergence analysis of the algorithm
will be presented in Section \ref{sect:conv}.

\subsection{Finite element discretization}
To discretize the problem, we first triangulate the domain $\Omega$. Let $\cT$ be a shape regular triangulation
of the polyhedral domain $\overline{\Omega}$ consisting of closed simplicial elements, with a local mesh size
$h_{T}:=|T|^{1/d}$ for each element $T\in\cT$, which is assumed to intersect at most one electrode surface $e_l$.
On the triangulation $\cT$, we define a continuous piecewise linear finite element space
\begin{equation*}
   V_\cT = \left\{v\in C(\overline{\Omega}): v|_T\in P_1(T)\ \forall T\in\cT\right\},
\end{equation*}
where the space $P_1(T)$ consists of all linear functions on the element $T$. The space $V_\cT$ is also used
for approximating the potential $u$ and the conductivity $\sigma$. The use of piecewise linear
finite elements is popular since the problem data, e.g., boundary conditions, have only limited regularity.

Now we can describe the FEM approximation. First, we approximate the
forward map $(u(\sigma),U(\sigma))\in \mathbb{H}$ by $(u_{\cT},U_{\cT})\equiv(u_{\cT}(\sigma_{\cT}),
U_{\cT}(\sigma_{\cT}))\in \mathbb{H}_{\cT}\equiv V_{\cT}\otimes\mathbb{R}_\diamond^L$ defined by
\begin{equation}\label{eqn:dispoly}
  a(\sigma_\cT,(u_\cT,U_\cT),(v_\cT,V)) = \langle I, V\rangle
  \quad (v_{\cT},V)\in \mathbb{H}_\cT,
\end{equation}
where the (discretized) conductivity $\sigma_{\cT}$ lies in the discrete admissible set
\begin{equation*}
  \mathcal{A}_{\cT}=\{\sigma_{\cT}\in V_{\cT}:\lambda\leq \sigma_{\cT}\leq \lambda^{-1}\ \mbox{ a.e. } \Omega\}=\mathcal{A}\cap V_{\cT}.
\end{equation*}
Then the discrete optimization problem reads
\begin{equation}\label{eqn:discopt}
  \min_{\sigma_{\cT}\in\mathcal{A}_{\cT}}\left\{J_{\cT}(\sigma_{\cT}) =\tfrac{1}{2}\|U_{\cT}(\sigma_{\cT})-U^\delta\|^2 + \tfrac{\alpha}{2}\|\nabla\sigma_{\cT}\|_{L^2(\Omega)}^2\right\}.
\end{equation}
Due to the compactness of the finite-dimensional space $\mathcal{A}_\cT$, it is easy to see that there exists at least one minimizer $\sigma_\cT^*$ to
problem \eqref{eqn:dispoly}-\eqref{eqn:discopt} (see, e.g., \cite{GehreJinLu:2014}). The minimizer $\sigma_{\cT}^{\ast}$
and the related forward solution $(u^*_\cT,U^*_\cT)\equiv(u^{\ast}_{\cT}(\sigma_{\cT}^{\ast}),U^{\ast}_{\cT}(\sigma_{\cT}^{\ast}))\in \mathbb{H}_\cT$
and adjoint solution $(p^*_\cT,P^*_\cT)\equiv(p^{\ast}_{\cT}(\sigma_{\cT}^{\ast}),P^{\ast}_{\cT}(\sigma_{\cT}^{\ast}))
\in\mathbb{H}_{\cT}$ satisfies the following necessary optimality system
\begin{equation}\label{eqn:cem-discoptsys}
    \begin{aligned}
      &a(\sigma_\cT^*,(u^*_\cT,U_\cT^*),(v_\cT,V)) = \langle I,V\rangle \quad\forall(v_{\cT},V)\in \mathbb{H}_{\cT},\\
      &a(\sigma_\cT^*,(p^*_\cT,P^*_\cT),(v_\cT,V)) = \langle U^*_\cT-U^\delta,V\rangle\quad \forall(v_\cT,V)\in\mathbb{H}_{\cT},\\
      &\alpha(\nabla\sigma_{\cT}^{\ast},\nabla(\mu_{\cT}-\sigma_{\cT}^{\ast})) -((\mu_{\cT}-\sigma_{\cT}^{\ast})\nabla u_{\cT}^{\ast},\nabla p_{\cT}^{\ast})\geq 0\quad\forall\mu_{\cT}\in\mathcal{A}_{\cT},
    \end{aligned}
\end{equation}
which is the discrete analogue of \eqref{eqn:cem-optsys}. Like in the continuous
case, it is straightforward to verify that
the discrete solutions $(u_{\cT}^\ast,U_{\cT}^\ast)$ and $(p_{\cT}^\ast,P_{\cT}^\ast)$ depend continuously on the input
current pattern $I$, i.e.,
\begin{equation}\label{stab-discpolyadj}
    \|(u_{\cT}^\ast,U_{\cT}^\ast)\|_{\mathbb{H},\ast}+\|(p_{\cT}^\ast,U_{\cT}^\ast)\|_{\mathbb{H},\ast}\leq c(\|I\|+\|U^\delta\|),
\end{equation}
{where the constant $c$ can be made independent of $\alpha$.}

\subsection{Adaptive algorithm}

Now we can present a novel AFEM for problem \eqref{eqn:cemweakform}-\eqref{eqn:tikh}. First we introduce some notation.
Let $\mathbb{T}$ be the set of all possible conforming triangulations of the domain $\overline{\Omega}$ obtained from
some shape-regular initial mesh $\cT_0$ by the successive use of bisection. We call $\cT'\in\mathbb{T}$ a refinement of $\cT\in
\mathbb{T} $ if $\cT'$ can be obtained from $\cT$ by a finite number of bisections. The collection of all faces (respectively
all interior faces) in $\cT\in\mathbb{T}$ is denoted by $\mathcal{F}_{\cT}$ (respectively $\mathcal{F}_{\cT}^i$) and
its restriction on  the electrode $\bar{e}_{l}$ and $\Gamma\backslash\cup_{l=1}^Le_l$ by $\mathcal{F}_{\cT}^{l}$ and $\mathcal{F
}_{\cT}^{c}$, respectively. The scalar $h_{F}:=|F|^{1/(d-1)}$ denotes the diameter of a face $F\in\mathcal{F}_{\cT}$, which is
associated with a fixed normal unit vector $\boldsymbol{n}_{F}$ in $\overline{\Omega}$ with
$\boldsymbol{n}_{F}=\boldsymbol{n}$ on the boundary $\Gamma$. Further, we denote by $D_{T}$ (respectively $D_{F}$)
the union of all elements in $\cT$ with non-empty intersection with an element $T\in\cT$ (respectively $F\in\mathcal{F}_{\cT}$).

\begin{remark}
{Our convergence analysis covers any bisection method that ensures} that the family $\mathbb{T}$ is uniformly
shape regular during the refinement process, i.e., shape regularity
of any $\cT\in\mathbb{T}$ is uniformly bounded by a constant depending only on the initial mesh $\cT_0$
\cite[Lemma 4.1]{NSV:2009}, and thus all constants only depend
on the initial mesh $\cT_0$ and given data but not on any subsequent mesh. {Such bisection methods include
in particular newest vertex bisection in two dimensions \cite{Mitchell:1989} and the bisection of \cite{Koss:1995} in three dimensions.
Note that no interior node property is enforced between two consecutive refinements by bisection in our AFEM.}
\end{remark}

For the solution $(\sigma^{\ast}_{\cT},u_{\cT}^{\ast},U_{\cT}^{\ast},
p_{\cT}^{\ast},P_{\cT}^{\ast})$ to problem \eqref{eqn:cem-discoptsys}, we define
two element residuals for each element $T\in\cT$ and two face residuals for each face
$F\in\mathcal{F}_{\cT}$ by
\begin{equation*}
 \begin{aligned}
    R_{T,1}(\sigma_{\cT}^{\ast},u^{\ast}_{\cT}) & = \nabla\cdot(\sigma_{\cT}^{\ast}\nabla u^{\ast}_{\cT}),\\
    R_{T,2}(u^{\ast}_{\cT},p^{\ast}_{\cT})& =\nabla u^{\ast}_{\cT}\cdot\nabla p^{\ast}_{\cT},
 \end{aligned}
\end{equation*}
\begin{equation*}
  \begin{aligned}
    J_{F,1}(\sigma_{\cT}^{\ast},u^{\ast}_{\cT},U_{\cT}^{\ast}) &=
    \left\{\begin{array}{lll}
                        [\sigma_{\cT}^{\ast}\nabla u_{\cT}^{\ast}\cdot\boldsymbol{n}_{F}]\quad&
                        \mbox{for} ~~F\in\mathcal{F}_{\cT}^i,\\ [1ex]
                        \sigma_{\cT}^{\ast}\nabla u_{\cT}^{\ast}\cdot\boldsymbol{n}+(u_{\cT}^{\ast}-U_{\cT,l}^{\ast})/z_{l}\quad&
                        \mbox{for} ~~ F\in\mathcal{F}_{\cT}^l,\\ [1ex]
                        \sigma_{\cT}^{\ast}\nabla u_{\cT}^{\ast}\cdot\boldsymbol{n}\quad&
                        \mbox{for} ~~F\in\mathcal{F}_{\cT}^{c},
    \end{array}\right.\\
    J_{F,2}(\sigma^{\ast}_{\cT}) &= \left\{\begin{array}{lll}
        [\alpha\nabla\sigma_{\cT}^{\ast}\cdot\boldsymbol{n}_{F}]\quad&
                        \mbox{for} ~~F\in\mathcal{F}_{\cT}^i,\\ [1ex]
                        \alpha\nabla\sigma_{\cT}^{\ast}\cdot\boldsymbol{n}\quad&
                        \mbox{for} ~~F\in \mathcal{F}_\cT^l \cup \mathcal{F}_\cT^c,
        \end{array}\right.
\end{aligned}
\end{equation*}
where $[\cdot]$ denotes the jumps across interior faces $F$. Then for any collection of elements
$\mathcal{M}_{\cT}\subseteq\cT$, we introduce the following error estimator
\begin{equation}
  \begin{aligned}
    \quad\eta_{\cT}^{2}(\sigma^{\ast}_{\cT},&u^{\ast}_{\cT},U^{\ast}_{\cT},p^{\ast}_{\cT},P^{\ast}_{\cT},
    \mathcal{M}_{\cT})
    :=\sum_{T\in\mathcal{M}_{\cT}}\eta_{\cT}^{2}(\sigma_{\cT}^{\ast},u_{\cT}^{\ast},U_{\cT}^{\ast},
    p_{\cT}^{\ast},P_{\cT}^{\ast},T)\\
    &:=\sum_{T\in\mathcal{M}_{\cT}}\eta_{\cT,1}^{2}(\sigma_{\cT}^{\ast},u_{\cT}^{\ast},U_{\cT}^{\ast},T)
    +\eta_{\cT,2}^{2}(\sigma_{\cT}^{\ast},p_{\cT}^{\ast},P_{\cT}^{\ast},T)+
    \eta_{\cT,3}^{2}(\sigma_{\cT}^{\ast},u_{\cT}^{\ast},p_{\cT}^{\ast},T),
  \end{aligned}\label{def:estimator}
\end{equation}
where the three components $\eta_{\cT,i}^2$, $i=1,2,3$, are defined by
\begin{equation*}
   \begin{aligned}
    \eta_{\cT,1}^{2}(\sigma_{\cT}^{\ast},u_{\cT}^{\ast},U_{\cT}^{\ast},T)
    & :=h_{T}^{2}\|R_{T,1}(\sigma^{\ast}_{\cT},u^{\ast}_{\cT})\|_{L^{2}(T)}^{2}
    +\sum_{F\subset\partial T}h_{F}\|J_{F,1}(\sigma^{\ast}_{\cT},u^{\ast}_{\cT},U^{\ast}_{\cT})\|_{L^{2}(F)}^{2},\\
    \eta_{\cT,2}^{2}(\sigma_{\cT}^{\ast},p_{\cT}^{\ast},P_{\cT}^{\ast},T)
    &:=h_{T}^{2}\|R_{T,1}(\sigma^{\ast}_{\cT},p^{\ast}_{\cT})\|_{L^{2}(T)}^{2}
    +\sum_{F\subset\partial T}h_{F}\|J_{F,1}(\sigma^{\ast}_{\cT},p^{\ast}_{\cT},P^{\ast}_{\cT})\|_{L^{2}(F)}^{2},\\
    \eta_{\cT,3}^{2}(\sigma_{\cT}^{\ast},u_{\cT}^{\ast},p_{\cT}^{\ast},T)
    &:=h_{T}^{4}\|R_{T,2}(u_{\cT}^{\ast},p_{\cT}^{\ast})\|^{2}_{L^{2}(T)}+\sum_{F\subset\partial T}h_{F}^{3}\|J_{F,2}(\sigma_{\cT}^{\ast})\|^{2}_{L^{2}(F)}.
  \end{aligned}
\end{equation*}
We defer the derivation of the a posteriori error estimator $\eta_\cT(\sigma^{\ast}_{\cT},u^{\ast}_{\cT},
U^{\ast}_{\cT},p^{\ast}_{\cT},P^{\ast}_{\cT}, \mathcal{M}_{\cT})$ to Section \ref{ssec:discu} below.
The notation $\mathcal{M}_{\cT}$ will be omitted whenever $\mathcal{M}_{\cT}=\cT$. Note that the estimator
$\eta_{\cT}$ depends only on the discrete solutions $(\sigma^*_\cT,u_\cT^*,U^*_\cT,p^*_\cT,P_\cT^*)$ and
the given problem data (e.g., impedance coefficients $\{z_l\}_{l=1}^L$), and all the quantities involved
in $\eta_{\cT}$ are computable. {Further, the regularization parameter $\alpha$ enters the
estimator only through the face residual $J_{F,2}(\sigma^*_\cT)$.} It will be shown in Section \ref{ssec:convergence} that
this error estimator is sufficient for the convergence of the resulting adaptive algorithm.

Now we can formulate an adaptive algorithm for the EIT inverse problem, cf. Algorithm \ref{alg_afem_eit}.
Below we indicate the dependence on the triangulation $\cT_k$ by the iteration number $k$ in the subscript.
\begin{algorithm}
\caption{Adaptive finite element method for EIT}\label{alg_afem_eit}
\begin{algorithmic}[1]
  \STATE Specify a shape regular initial mesh $\cT_{0}$, and set $k:=0$.
  \STATE {(\texttt{SOLVE})} Solve problem \eqref{eqn:dispoly}-\eqref{eqn:discopt} over $\cT_{k}$ for the minimizer $(\sigma_k^{\ast},u_{k}^{\ast},U_{k}^{\ast})\in\mathcal{A}_{k}\times\mathbb{H}_{k}$ and the adjoint solution $(p_k^\ast,P_k^\ast)\in \mathbb{H}_k$; see \eqref{eqn:cem-discoptsys}.
  \STATE {(\texttt{ESTIMATE})} Compute the error estimator $\eta_{k}(\sigma_{k}^{\ast},u_{k}^{\ast},U_{k}^{\ast},p_{k}^{\ast},P_{k}^{\ast})$ by \eqref{def:estimator}.
  \STATE{(\texttt{MARK})} Mark a subset $\mathcal{M}_{k}\subseteq\cT_{k}$ with at least one element $\widetilde{T}\in\cT_{k}$ with the largest error indicator:
    \begin{equation}\label{eqn:marking}
        \eta_{k}(\sigma_k^{\ast},u_k^{\ast},U_k^{\ast},p_k^{\ast},P_k^{\ast},\widetilde{T})
        =\max_{T\in\cT_{k}}\eta_{k}(\sigma_k^{\ast},u_k^{\ast},U_k^{\ast},p_k^{\ast},P_k^{\ast},T).
    \end{equation}
   \STATE {(\texttt{REFINE})} Refine each element $T$ in $\mathcal{M}_{k}$ by bisection to get $\cT_{k+1}$.
   \STATE Set $k=k+1$, and return to Step 2, until a certain stopping criterion is fulfilled.
\end{algorithmic}
\end{algorithm}

\begin{remark}
The solver in the module \texttt{SOLVE} can be either a (projected) gradient descent method or iteratively regularized
Gauss-Newton method, each equipped with a suitable step size selection rule.
\end{remark}

\begin{remark}\label{rmk:marking&complex}
Assumption \eqref{eqn:marking} in the module \texttt{MARK} is fairly general, and it covers several commonly
used {collective} marking strategies, e.g., maximum strategy, equidistribution, modified equidistribution strategy, and
D\"{o}rfler's strategy \cite[pp. 962]{Siebert:2011}. Our convergence analysis in Section \ref{sect:conv} covers
all these marking strategies. {In the module \texttt{MARK}, one may also consider separate
marking. The motivation is to be able to neglect data oscillations, which have no importance for sufficiently fine
meshes. Numerically, this adds little computational overheads, since the module \texttt{SOLVE} is the most expensive step
at each iteration.}
\end{remark}

Last, we give an important geometric observation on the mesh sequence
$\{\cT_{k}\}$ and a stability result on error indicators
$\eta_{k,1}(\sigma_{k}^\ast,u^\ast_{k},U^\ast_{k})$,
$\eta_{k,2}(\sigma_{k}^\ast,p^\ast_{k},P^\ast_{k})$ and $\eta_{k,3}(\sigma_k^\ast,u^\ast_{k},p^\ast_{k})$
given in Algorithm \ref{alg_afem_eit}. Let
\begin{equation*}
    \cT_{k}^{+}:=\bigcap_{l\geq k}\cT_{l},\quad
    \cT_{k}^{0}:=\cT_{k}\setminus\cT_{k}^{+},\quad
    \Omega_{k}^{+}:=\bigcup_{T\in\cT^{+}_{k}}D_{T},\quad
    \Omega_{k}^{0}:=\bigcup_{T\in\cT^{0}_{k}}D_{T}.
\end{equation*}
That is, the set $\cT_{k}^{+}$ consists of all elements not refined after the $k$-th iteration while
all elements in $\cT_{k}^{0}$ are refined at least once after the $k$-th iteration. Clearly, $\cT_{l}^{+}
\subset\cT_{k}^{+}$ for $l<k$. We also define a mesh-size function $h_{k}:\overline{\Omega}\rightarrow
\mathbb{R}^{+}$ almost everywhere by $h_{k}(x)=h_{T}$ for $x$ in the interior of an element $T\in\cT_{k}$
and $h_{k}(x)=h_{F}$ for $x$ in the relative interior of an edge $F\in\mathcal{F}_{k}$. It has the following
important property in the region of $\Omega$ involving marked elements \cite[Corollary 3.3]{Siebert:2011}.
\begin{lemma} Let $\chi^{0}_{k}$ be the characteristic function of $\Omega_{k}^{0}$. Then
\label{lem:conv_zero_mesh}
$
    \lim_{k\rightarrow\infty}\|h_{k}\chi^{0}_{k}\|_{L^\infty(\Omega)}=0.
$
\end{lemma}

The next result gives preliminary bounds on the a posteriori error estimators. {Note that
only the constant $c$ for the estimator $\eta_{k,3}$ depends on the regularization parameter $\alpha$,
via the face residuals $J_{F,2}(\sigma_k^*)$, and all the constants can be naturally made independent of $\alpha$, if desired.}
\begin{lemma}\label{lem:stab-indicator}
Let the sequence of discrete solutions $\{(\sigma_{k}^\ast,u^\ast_{k},U^\ast_{k},p^\ast_{k},P^\ast_{k})\}$
be generated by Algorithm \ref{alg_afem_eit}. Then for each $T\in\cT_k$ with its face $F$, there hold
\begin{equation*}
   \begin{aligned}
    \displaystyle \eta_{k,1}^2(\sigma_{k}^\ast,u^\ast_{k},U^\ast_{k},T) &\leq c(\|\nabla u^\ast_{k}\|^2_{L^2(D_{T})}+h_{F}\|u^\ast_{k}-U^{\ast}_{k,l}\|^2_{L^2(F\cap e_{l})}),\\
    \displaystyle \eta_{k,2}^2(\sigma_{k}^\ast,p^\ast_{k},P^\ast_{k},T)&\leq c(\|\nabla p^\ast_{k}\|^2_{L^2(D_{T})}+h_{F}\|p_k^\ast-P^{\ast}_{k,l}\|^2_{L^2(F\cap e_{l})}),\\
   \displaystyle \eta_{k,3}^2(\sigma_{k}^\ast,u^\ast_{k},p^\ast_{k},T)&\leq
    c(h_{T}^{4-d}\|\nabla u^{\ast}_{k}\|^2_{L^2(T)}\|\nabla p^{\ast}_{k}\|^2_{L^2(T)}+h_{T}^2\|\nabla\sigma_{k}^\ast\|^2_{L^2(D_{T})}),
  \end{aligned}
\end{equation*}
{where $e_l$ denotes the electrode intersecting with the element $T\in\cT_k$.}
\end{lemma}
\begin{proof}
We only prove the third estimate, and the first two follow analogously.
By the inverse estimates and the trace theorem, the local quasi-uniformity
of $\cT_{k}$ yields
\begin{equation*}
  \begin{aligned}
    h_{T}^4\|\nabla u^*_k\cdot\nabla p^*_k\|^2_{L^2(T)} &\leq ch_T^{4-d}\|\nabla u^*_k\cdot\nabla p^*_k\|^2_{L^1(T)}\leq ch_{T}^{4-d}\|\nabla u^{\ast}_{k}\|^2_{L^2(T)}\|\nabla p^{\ast}_{k}\|_{L^2(T)}^2,\\
    \sum_{F\subset\partial T}h_F^3\|J_{F,2}(\sigma_k^*)\|_{L^2(F)}^2&\leq ch^2_T\|\nabla\sigma_k^*\|^2_{L^2(D_T)}.
  \end{aligned}
\end{equation*}
\end{proof}

\subsection{Derivation of a posteriori error estimators}\label{ssec:discu}
Now we motivate the a posteriori error estimator $\eta_\cT $ defined in \eqref{def:estimator} underlying the module
\texttt{ESTIMATE} of Algorithm \ref{alg_afem_eit}. The algorithm generates a sequence of discrete solutions $\{(\sigma_k^*,
u_k^*,U_k^*,p_k^*,P_k^*)\}$ in a sequence of finite element spaces $\{V_k\}$ and discrete admissible sets $\{\mathcal{A}_k\}$
over a sequence of meshes $\{\cT_k\}$. Naturally, some arguments in the a posteriori error estimation for direct problems will be
employed. We shall need the following results on the Lagrange interpolation operator $I_k:H^2(\Omega)\rightarrow V_{k}$
\cite{Ciarlet:2002} and the Scott-Zhang interpolation operator $I_k^{sz}:H^1(\Omega)\rightarrow V_k$ \cite{ScottZhang:1990}
over the triangulation $\mathcal{T}_{k}$.
\begin{lemma}\label{lem:est_int}
Let $\omega_{F}$ is the union of elements with $F$ as a face.    For any $T\in\mathcal{T}_{k}$ and any $F\in\mathcal{F}_{k}$,
    \begin{equation*}
      \begin{aligned}
        & \|v-I_kv\|_{L^2(T)}\leq ch^2_T\|v\|_{H^2(T)},\quad\|v-I_kv\|_{L^2(F)}\leq ch^{3/2}_T\|v\|_{H^2(\omega_{F})},\\
        &  \|v-I_k^{sz}v\|_{L^2(T)}\leq ch_T\|v\|_{H^1(D_T)},\quad\|v-I_k^{sz}v\|_{L^2(F)}\leq ch^{1/2}_T\|v\|_{H^1(D_{F})}.
      \end{aligned}
    \end{equation*}
\end{lemma}

To motivate the error estimator $\eta_\cT$, we begin with two auxiliary problems: find $(\widetilde{u}(\sigma_{k}^\ast),\widetilde{U}(\sigma_{k}^*))\in\mathbb{H}$ and $(\widetilde{p}(\sigma_{k}^\ast),\widetilde{P}(\sigma_{k}^*))\in\mathbb{H}$ such that
\begin{align}
    a(\sigma_{k}^*,(\widetilde{u},\widetilde{U}),(v,V)) & = \langle I,V\rangle \quad \forall (v,V)\in \mathbb{H},\label{eqn:cemwf_aux1}\\
    a(\sigma_k^*,(\widetilde{p},\widetilde{P}),(v,V)) & = \langle \widetilde{U}-U^\delta,V\rangle       \quad\forall(v,V)\in\mathbb{H}.\label{eqn:cemwf_aux2}
\end{align}
The first line in \eqref{eqn:cem-discoptsys} is actually the finite element scheme of \eqref{eqn:cemwf_aux1}
over $\mathcal{T}_k$. Hence, the standard a posteriori error analysis for forward problems can be applied. By setting
$v_k=I_k^{sz}v\in V_k$ in the first line in \eqref{eqn:cem-discoptsys} for any $(v,V)\in\mathbb{H}$, applying elementwise
integration by parts and Lemma \ref{lem:est_int}, there hold
\begin{equation*}
    \begin{aligned}
        a(\sigma_{k}^*,(\widetilde{u}-u_k^*,\widetilde{U}-U_k^*),(v,V))&=\langle I,V\rangle-(\sigma_{k}^\ast\nabla u_{k}^\ast,\nabla v)-\sum_{l=1}^Lz^{-1}_l(u_k^*-U^*_{k,l},v-V_l)_{L^2(e_l)}\\
        &=(\sigma_{k}^\ast\nabla u_{k}^\ast,\nabla(I_{k}^{sz}v-v))
        +\sum_{l=1}^Lz^{-1}_l(u_k^*-U^*_{k,l},I_{k}^{sz}v-v)_{L^2(e_l)}\\
        &\leq c\Big(\sum_{T\in\mathcal{T}_k}\eta^2_{k,1}(\sigma^*_k,u_k^*,U_k^*,T)\Big)^{1/2}\|v\|_{H^1(\Omega)}.
    \end{aligned}
\end{equation*}
Taking $(v,V)=(\widetilde{u}-u^*_k,\widetilde{U}-U^*_k)\in\mathbb{H}$ and using Lemma \ref{lem:normequiv} yield
\begin{equation}\label{ub_state_aux1}
    \|(\widetilde{u}-u^*_k,\widetilde{U}-U^*_k)\|_{\mathbb{H},\ast}\leq c\Big(\sum_{T\in\mathcal{T}_k}\eta^2_{k,1}(\sigma^*_k,u_k^*,U_k^*,T)\Big)^{1/2}.
\end{equation}
Further, from the first equation in \eqref{eqn:cem-optsys} and \eqref{eqn:cemwf_aux1} we find for any $(v,V)\in\mathbb{H}$
\begin{equation*}
        a(\sigma_k^\ast,(u^*-\widetilde{u},U^*-\widetilde{U}),(v,V))
        =((\sigma_k^*-\sigma^\ast)\nabla u^*,\nabla v)\leq \|(\sigma^\ast-\sigma_k^*)\nabla u^\ast\|_{L^2(\Omega)}\|\nabla v\|_{L^2(\Omega)}.
\end{equation*}
Consequently,
\begin{equation}\label{ub_state_aux2}
    \|(u^*-\widetilde{u},U^*-\widetilde{U})\|_{\mathbb{H},\ast}\leq
    c\|(\sigma^\ast-\sigma_k^*)\nabla u^\ast\|_{L^2(\Omega)}.
\end{equation}
Likewise, for $(p^\ast-p^\ast_k,P^*-P^*_k)$, we appeal to the second equation in the discrete optimality
system \eqref{eqn:cem-discoptsys} and the auxiliary problem \eqref{eqn:cemwf_aux2} to deduce
\begin{equation*}
    \begin{aligned}
        a(\sigma_k^*,(\widetilde{p}-p_k^*,&\widetilde{P}-P_k^*),(v,V))=\langle \widetilde{U}-U^\delta,V\rangle-a(\sigma_k^*,(p_k^*,P_k^*),(v,V))\\
        &=\langle \widetilde{U}-U_k^*,V\rangle+\langle U_k^*-U^\delta,V\rangle-a(\sigma_k^*,(p_k^*,P_k^*),(v,V))\\
        &=\langle\widetilde{U}-U_k^*,V\rangle+(\sigma_{k}^\ast\nabla p_{k}^\ast,\nabla(I_{k}^{sz}v-v))
        +\sum_{l=1}^Lz^{-1}_l(p_k^*-P^*_{k,l},I_{k}^{sz}v-v)_{L^2(e_l)}\\
        &\leq c\Big(\big(\sum_{T\in\mathcal{T}_k}\eta^2_{k,2}(\sigma^*_k,p_k^*,P_k^*,T)\big)^{1/2}+\|\widetilde{U}-U^*_k\|\Big)\|(v,V)\|_{\mathbb{H},\ast},
    \end{aligned}
\end{equation*}
and further
\begin{equation*}
    \begin{aligned}
    a(\sigma_k^\ast,(p^*-\widetilde{p},P^\ast-\widetilde{P}),(v,V))&=
    ((\sigma_k^*-\sigma^\ast)\nabla p^*,\nabla v)+\langle U^*-\widetilde{U},V\rangle\\
    &\leq\left(\|(\sigma^\ast-\sigma_k^*)\nabla p^\ast\|_{L^2(\Omega)}+\|U^*-\widetilde{U}\|\right)\|(v,V)\|_{\mathbb{H},\ast},
    \end{aligned}
\end{equation*}
which, together with \eqref{ub_state_aux1} and \eqref{ub_state_aux2} and Lemma \ref{lem:normequiv}, implies
\begin{equation}\label{ub_state_aux3}
    \begin{split}
    \|(p^\ast-p^\ast_k,P^*-P^*_k)\|_{\mathbb{H},\ast}
    &\leq c\Big(\big(\sum_{T\in\mathcal{T}_{k}}\eta^2_{k,1}(\sigma^*_k,u_k^*,U_k^*,T)+\eta^2_{k,2}(\sigma^*_k,p_k^*,P_k^*,T)\big)^{1/2}\\
     &\quad+\|(\sigma^\ast-\sigma_k^*)\nabla u^\ast\|_{L^2(\Omega)}+ \|(\sigma^\ast-\sigma_k^*)\nabla p^\ast\|_{L^2(\Omega)}\Big).
     \end{split}
\end{equation}

In view of \eqref{ub_state_aux1}-\eqref{ub_state_aux3}, the estimators $\eta_{k,1}$ and $\eta_{k,2}$ can bound
$(u^\ast-u^\ast_k,U^*-U^*_k)$ and $(p^\ast-p^\ast_k,P^*-P^*_k)$ from above up to the terms
$\|(\sigma^*-\sigma^*_k)\nabla u^*\|_{L^2(\Omega)}$ and $\|(\sigma^*-\sigma^*_k)\nabla p^*\|_{L^2(\Omega)}$,
which are not computable but asymptotically vanishing, provided that $\sigma_k^*\to\sigma^*$ pointwise.
This motivates our choice of a computable upper bound for $\sigma^*-\sigma^*_k$, upon discarding the
uncomputable terms.

To bound the term $\|\nabla(\sigma_k^*-\sigma^*)\|_{L^2(\Omega)}$, we appeal to the variational
inequalities in \eqref{eqn:cem-optsys} and \eqref{eqn:cem-discoptsys}. Since
$I_k\mu\in\mathcal{A}_k$ for any $\mu\in\mathcal{A}\cap C^\infty(\overline{\Omega})$, we deduce
\begin{equation*}
    \begin{aligned}
        \alpha\|\nabla(\sigma^*-\sigma^*_k)\|_{L^2(\Omega)}^2&\leq
        \alpha(\nabla\sigma_{k}^*,\nabla(\sigma_k^*-\sigma^\ast))-((\sigma_k^*-\sigma^\ast)\nabla u^*,\nabla p^*)\\
        &=\alpha(\nabla\sigma_k^*,\nabla(\sigma_k^*-\sigma^\ast))
        -((\sigma_k^*-\sigma^\ast)\nabla u_k^*,\nabla p_k^*)\\
        &\quad +(\nabla u_k^*\cdot\nabla p_k^*-\nabla u^*\cdot\nabla p^*,\sigma_k^*-\sigma^\ast)\\
        &\leq\alpha(\nabla\sigma_{k}^*,\nabla(I_k\mu-\sigma^\ast))-
        ((I_k\mu-\sigma^\ast)\nabla u_k^*,\nabla p_k^*)\\
        &\quad +(\nabla u_k^*\cdot\nabla p_k^* -\nabla u^*\cdot\nabla p^*,\sigma_k^*-\sigma^\ast)\\
        &=\alpha(\nabla\sigma_{k}^*,\nabla(I_k\mu-\mu))-
        (\nabla u_k^*,\nabla p_k^*(I_k\mu-\mu))\\
        &\quad+(\nabla u_k^*\cdot\nabla p_k^*-\nabla u^*\cdot\nabla p^*,\sigma_k^*-\sigma^\ast)\\
        &\quad+\alpha(\nabla\sigma_{k}^*,\nabla(\mu-\sigma^\ast))-
        ((\mu-\sigma^\ast)\nabla u_k^*,\nabla p_k^*):=\mathrm{I}+\mathrm{II}+\mathrm{III}.
    \end{aligned}
\end{equation*}
Now Lemma \ref{lem:est_int} and elementwise integration by parts yield
\begin{equation}\label{ub_state_aux4}
    |{\rm I}|\leq c(\sum_{T\in\cT_{k}}\eta_{k,3}^2(\sigma_{k}^\ast,u^*_k,p^*_k,T))^{1/2}\|\mu\|_{H^2(\Omega)}\quad
    \forall\mu\in\mathcal{A}\cap C^\infty(\overline{\Omega}).
\end{equation}
By the minimizing property of $\sigma_k^*$ for $J_k(\cdot)$, $\|\nabla\sigma_k^*\|_{L^2(\Omega)}$ is
bounded. Then the estimate \eqref{stab-discpolyadj} and the density of $\mathcal{A}\cap C^\infty(\overline{\Omega})$
in $\mathcal{A}$ ensure that the term $\mathrm{III}$ can be made arbitrarily small. For the term $\mathrm{II}$, we
have
\begin{equation*}
  \begin{aligned}
    |\mathrm{II}| & = |(\nabla u_k^*\cdot\nabla p_k^* - \nabla u_k^*\cdot\nabla p^* + \nabla u_k^*\cdot\nabla p^* - \nabla u^*\cdot\nabla p^*,\sigma^k-\sigma^*)|\\
     & \leq \|\nabla(p_k^*-p^*)\|_{L^2(\Omega)}\|(\sigma_k^*-\sigma^*)\nabla u_k^*\|_{L^2(\Omega)} + \|\nabla(u_k^*-u^*)\|_{L^2(\Omega)}\|(\sigma_k^*-\sigma^*)\nabla p^*\|_{L^2(\Omega)},
  \end{aligned}
\end{equation*}
which are expected to be higher order terms. Upon discarding the uncomputable terms $\|(\sigma^\ast-\sigma_k^*)
\nabla u^\ast\|_{L^2(\Omega)}$ and $\|(\sigma^\ast-\sigma_k^*)\nabla p^\ast\|_{L^2(\Omega)}$ in
\eqref{ub_state_aux2}-\eqref{ub_state_aux3} and the nonlinear term $\mathrm{II}$, we get all computable
quantities in \eqref{ub_state_aux1}, \eqref{ub_state_aux3} and \eqref{ub_state_aux4}, which are exactly the
a posteriori error estimator $\eta_{k}$ defined in \eqref{def:estimator}. Thus we may view it as a reliable upper
bound for the error and employ it in the module \texttt{ESTIMATE} to drive the adaptive
refinement process. Moreover the derivation of \eqref{ub_state_aux4} suggests itself a natural way to handle the
variational inequality in \eqref{eqn:cem-optsys} in the convergence analysis,
which will be presented in Section \ref{sect:conv} below.

\section{Convergence analysis}\label{sect:conv}

In this section, we shall establish the main theoretical result of this work,
the convergence of Algorithm \ref{alg_afem_eit}, namely the sequence of discrete solutions
$\{(\sigma_k^*,u_k^*,U_k^*,p_k^*,P_k^*)\}$ to the optimality system \eqref{eqn:cem-discoptsys} generated by Algorithm
\ref{alg_afem_eit}, contains a subsequence converging in $H^1(\Omega)\times\mathbb{H}\times\mathbb{H}$
to a solution to the optimality system \eqref{eqn:cem-optsys}. The main technical difficulty lies in the lack of density
of the adaptively generated FE space $V_{k}$ in the space $H^1(\Omega)$. To overcome the challenge, the proof is
carried out in two steps. In the first step (Section \ref{ssec:limitingproblem}), we analyze a ``limiting'' optimization problem posed over a limiting set
induced by $\{\mathcal{A}_{k}\}$, and show that the sequence of discrete solutions contains a convergent
subsequence to a minimizer to the limiting problem. In the second step (Section \ref{ssec:convergence}),
we show that the solution to the optimality system for the limiting problem actually solves the optimality system \eqref{eqn:cem-optsys}.
It is worth noting that all the proofs in Section \ref{ssec:limitingproblem} only depends on the nestedness
of finite element spaces $\{V_k\}$ and discrete admissible sets $\{\mathcal{A}_k\}$, and the error estimator \eqref{def:estimator} and
the marking assumption \eqref{eqn:marking} are used only in Section \ref{ssec:convergence}.

\subsection{Limiting optimization problem}\label{ssec:limitingproblem}

For the sequences $\{\mathbb{H}_{k}\}$ and $\{\mathcal{A}_{k}\}$ generated by Algorithm \ref{alg_afem_eit},
we define a limiting finite element space $\mathbb{H}_\infty$ and a limiting admissible set $\mathcal{A}_\infty$ respectively by
\begin{equation*}
    \mathbb{H}_{\infty}:=\overline{\bigcup_{k\geq 0}\mathbb{H}_{k}}~ (\mbox{in}~\mathbb{H},\ast\mbox{-norm})\quad\mbox{and}\quad
    \mathcal{A}_{\infty}:=\overline{\bigcup_{k\geq 0}\mathcal{A}_{k}}~(\mbox{in}~H^{1}(\Omega)\mbox{-norm}).
\end{equation*}
It is easy to see that $\mathbb{H}_{\infty}$ is a closed subspace of $\mathbb{H}$. For the set $\mathcal{A}_{\infty}$, we have the following lemma.

\begin{lemma}\label{lem:convexclosed}
    $\mathcal{A}_{\infty}$ is a closed convex subset of $\mathcal{A}$.
\end{lemma}
\begin{proof}
The definition of $\mathcal{A}_{\infty}$ implies its strong closedness. For any $\mu$ and $\nu$ in
$\mathcal{A}_{\infty}$, there exist two sequences $\{\mu_{k}\}$ and $\{\nu_{k}\}\subset\bigcup_{k\geq
0}\mathcal{A}_{k}$ such that $\mu_{k}\rightarrow\mu$ and $\nu_{k}\rightarrow\nu$ in $H^{1}(\Omega)$.
The convexity of the set $\mathcal{A}_{k}$ implies $\{t\mu_{k}+(1-t)\nu_{k}\}\subset\bigcup_{k\geq 0}
\mathcal{A}_{k}$ for any $t\in(0,1)$. Then $t\mu_{k}+(1-t)\nu_{k}\rightarrow t\mu+(1-t)\nu$ in
$H^{1}(\Omega)$, i.e. $t\mu+(1-t)\nu\in\mathcal{A}_{\infty}$ for any $t\in (0,1)$. Hence $\mathcal{A}_{
\infty}$ is convex. Moreover, we have $\mu_{k}\rightarrow\mu$ a.e. in $\Omega$ after (possibly) passing
to a subsequence, which, along with the constraint $\lambda\leq\mu_{k}\leq\lambda^{-1}$ a.e. in $\Omega$,
indicates that $\lambda\leq\mu\leq\lambda^{-1}$ a.e. in $\Omega$. Lastly, the fact that $\mathcal{A}_{\infty}
\subset H^1(\Omega)$ concludes  $\mathcal{A}_{\infty}\subset\mathcal{A}$.
\end{proof}

Over the limiting set $\mathcal{A}_{\infty}$, we introduce a limiting minimization problem:
\begin{equation}\label{eqn:medopt}
  \min_{\sigma_{\infty}\in\mathcal{A}_{\infty}}\left\{J_{\infty}(\sigma_{\infty}) = \tfrac{1}{2}\|U_{\infty}(\sigma_{\infty})-U^\delta\|^2 + \tfrac{\alpha}{2}\|\nabla\sigma_{\infty}\|_{L^2(\Omega)}^2\right\},
\end{equation}
where $(u_{\infty},U_{\infty})\equiv(u_{\infty}(\sigma_{\infty}),U_{\infty}(\sigma_{\infty}))
\in \mathbb{H}_{\infty}$ satisfies the variational problem:
\begin{equation}\label{eqn:medpoly}
  a(\sigma_\infty,(u_\infty,U_\infty),(v,V)) = \langle I,V\rangle \quad \forall (v,V)\in\mathbb{H}_\infty.
\end{equation}
By Lemma \ref{lem:normequiv} and Lax-Milgram theorem, the limiting variational problem
\eqref{eqn:medpoly} is well-posed for any fixed $\sigma_\infty\in\mathcal{A}_\infty$.
The next result shows the existence of a minimizer to the limiting problem \eqref{eqn:medopt}-\eqref{eqn:medpoly}.
\begin{theorem}\label{thm:existmedmin}
There exists at least one minimizer to problem \eqref{eqn:medopt}-\eqref{eqn:medpoly}.
\end{theorem}
\begin{proof}
It is clear that $\inf J_{\infty}(\sigma)$ is finite over $\mathcal{A}_{\infty}$,
so there exists a minimizing sequence $\{\sigma^{n}\}\subset\mathcal{A}_{\infty}$, i.e., $$\lim_{n\to\infty}J_{\infty}(\sigma^{n})=\inf_{\sigma\in\mathcal{A}_{\infty}}J_{\infty}(\sigma).$$
Thus, the sequence $\{\sigma^n\}$ is uniformly bounded in $H^1(\Omega)$, and
by Sobolev embedding theorem and Lemma \ref{lem:convexclosed}, there exists a subsequence,
relabeled as $\{\sigma^{n}\}$, and some $\sigma^{\ast}\in\mathcal{A}_{\infty}$ such that
$ \sigma^{n}\rightarrow\sigma^{\ast}$ weakly in $H^{1}(\Omega)$, $\sigma^{n}\rightarrow\sigma^{\ast}$ a.e. in $\Omega$.
By taking $\sigma_{\infty}=\sigma^{n}\in\mathcal{A}_\infty$ in \eqref{eqn:medpoly},
then $(u^{n},U^{n})\equiv(u^{n}(\sigma^{n}),U^{n}(\sigma^{n}))\in {\mathbb{H}_\infty\subset\mathbb{H}}$ satisfies
\begin{equation}\label{eqn:weak-medmin}
  a(\sigma^n,(u^n,U^n),(v,V)) = \langle I,V\rangle\quad\forall (v,V)\in \mathbb{H}_\infty.
\end{equation}
Then by Lemma \ref{lem:normequiv}, $\{(u^n,U^n)\}$ is uniformly bounded in $\mathbb{H}$,
which gives a subsequence, also denoted by $\{(u^{n},U^{n})\}$, and some
$(u^{\ast},U^{\ast})\in\mathbb{H}_{\infty}$ such that
\begin{equation}\label{eqn:medmin-weakconv}
   (u^{n},U^n)\rightarrow (u^{\ast},U^*)\quad\mbox{weakly in}~\mathbb{H}\quad\mbox{and}\quad
   u^{n}\rightarrow u^{\ast}\quad\mbox{in}~L^{2}(\Gamma).
\end{equation}
We claim that $(u^*,U^*)=(u^*(\sigma^*),U^*(\sigma^*))\in \mathbb{H}_\infty$. To this end,
first we observe the splitting
\begin{equation*}
   (\sigma^{n}\nabla u^{n}, \nabla v ) =((\sigma^{n}-\sigma^{\ast})\nabla u^{n}, \nabla v)
        +(\sigma^{\ast}\nabla u^{n},\nabla v).
\end{equation*}
The pointwise convergence of the sequence $\{\sigma^n\}$, Lebesgue's dominated
convergence theorem (\cite{Evans:1992a}) and the uniform boundedness of $\{u^n\}$ in $H^1(\Omega)$ imply that
\begin{equation*}
   |((\sigma^{n}-\sigma^{\ast})\nabla u^{n},\nabla v)|
   \leq\|\nabla u^{n}\|_{L^{2}(\Omega)}\|(\sigma^{n}-\sigma^{\ast})\nabla v\|_{L^2(\Omega)}\rightarrow 0.
\end{equation*}
This and the weak convergence of $\{u^n\}$ in $H^1(\Omega)$ give
\begin{equation*}
    (\sigma^{n}\nabla u^{n},\nabla v)_{L^2(\Omega)}\rightarrow
    (\sigma^{\ast}\nabla u^{\ast}, \nabla v)_{L^2(\Omega)}.
\end{equation*}
Then by \eqref{eqn:medmin-weakconv}, we obtain
\begin{equation*}
   (u^n-U^n_{l},v-V_l)_{L^2(e_l)} \rightarrow (u^\ast-U^\ast_{l},v-V_l)_{L^2(e_l)}.
\end{equation*}
Upon taking into account these relations, we deduce
\begin{equation*}
  a(\sigma^*,(u^*,U^*),(v,V)) = \langle I,V\rangle\quad \forall (v,V)\in\mathbb{H}_\infty,
\end{equation*}
i.e., the desired claim $(u^{\ast},U^\ast)=(u^{\ast}(\sigma^{\ast}),U^{\ast}(\sigma^\ast))\in \mathbb{H}_\infty$.
This and the weak lower semicontinuity of the norm imply that $\sigma^\ast$ is a minimizer of $J_{\infty}(\cdot)$
over $\mathcal{A}_{\infty}$, completing the proof of the theorem.
\end{proof}

The preceding proof together with the uniqueness of the solution to \eqref{eqn:medpoly} and
the standard subsequence argument yields the following weak continuity result.
\begin{lemma}\label{lem:weakcontmedpoly}
Let the sequence $\{\sigma_{k}\}\subset\bigcup_{k\geq 0}\mathcal{A}_{k}$ converge to some $\sigma^\ast
\in\mathcal{A}_{\infty}$ weakly in $H^1(\Omega)$ and let the solution to \eqref{eqn:medpoly} with
$\sigma_\infty=\sigma^\ast$ be $(u(\sigma^{\ast}),U(\sigma^{\ast}))\in\mathbb{H}_{\infty}$.
Then the sequence of solutions $\{(u_{k}(\sigma_{k}),U_{k}(\sigma_{k}))\}\subset
\bigcup_{k\geq 0}\mathbb{H}_{k}$ to \eqref{eqn:dispoly} over $\cT_{k}$ satisfies
\begin{equation*}
    (u_{k}(\sigma_{k}),U_{k}(\sigma_{k})) \rightarrow (u(\sigma^\ast),U(\sigma^*))\quad \mbox{weakly in}~\mathbb{H}.
\end{equation*}
\end{lemma}

Now we analyze the limiting behavior of the sequence
$\{(\sigma_{k}^{\ast},u_{k}^{\ast},U_{k}^{\ast})\}$ generated by Algorithm \ref{alg_afem_eit}:
It contains a subsequence converging in $H^1(\Omega)\times \mathbb{H}$ to a minimizer of the limiting problem \eqref{eqn:medopt}-\eqref{eqn:medpoly}.
This result will play a crucial role in the convergence analysis in Section \ref{ssec:convergence}.
\begin{theorem}\label{thm:conver_medmin}
Let $\{\mathcal{A}_{k}\times\mathbb{H}_{k}\}$ be a sequence of discrete admissible sets and finite
element spaces generated by Algorithm \ref{alg_afem_eit}. Then the sequence of discrete solutions
$\{(\sigma_{k}^{\ast},u^{\ast}_{k},U_{k}^{\ast})\}$ to problem \eqref{eqn:discopt} has a subsequence
$\{(\sigma_{k_{m}}^{\ast},u_{k_{m}}^{\ast},U_{k_{m}}^\ast)\}$ converging to a minimizer
$(\sigma_{\infty}^{\ast},u_{\infty}^{\ast},U_{\infty}^\ast)$ to problem \eqref{eqn:medopt}-\eqref{eqn:medpoly} in the sense that
\begin{equation*}
  \begin{aligned}
     &\sigma_{k_{m}}^\ast\rightarrow\sigma_{\infty}^\ast\quad\mbox{ in}~H^1(\Omega),
     \quad \sigma_{k_{m}}^\ast\rightarrow\sigma_{\infty}^\ast\quad\mbox{a.e. in}~\Omega,
    \quad  (u_{k_{m}}^\ast,U_{k_m}^*)\rightarrow (u_{\infty}^\ast,U_\infty^*)\quad\mbox{ in}~\mathbb{H}.
  \end{aligned}
\end{equation*}
\end{theorem}
\begin{proof}
Since the function $\sigma\equiv1\in\mathcal{A}_{k}$ for all $k$ and $J_{k}(\sigma^{\ast}_{k})$ attains its minimum
at $\sigma_k^*\in\mathcal{A}_{k}$, the sequence $\{\sigma^\ast_k\}$ is
uniformly bounded in $H^1(\Omega)$. By Sobolev embedding theorem,
there exists a subsequence $\{\sigma^{\ast}_{k_{m}}\}$ and some
$\sigma_\infty^\ast\in\mathcal{A}_{\infty}$ such that
$\sigma^{\ast}_{k_{m}}\rightarrow\sigma^\ast_{\infty}$ weakly in $H^1(\Omega)$,
    $\sigma^{\ast}_{k_{m}}\rightarrow\sigma_{\infty}^\ast$ a.e. in $\Omega$.
By Lemma \ref{lem:weakcontmedpoly}, there exists a subsequence of
$\{(u^\ast_{k_m},U^\ast_{k_m})\}$ such that
\begin{equation*}
  (u^{\ast}_{k_{m}},U^{\ast}_{k_{m}})\rightarrow (u^{\ast}_{\infty}(\sigma_\infty^{\ast}),U^{\ast}_{\infty}(\sigma_\infty^{\ast}))\quad\mbox{weakly in}~\mathbb{H},
\end{equation*}
where $(u^{\ast}_{\infty}(\sigma_{\infty}^{\ast}),U^{\ast}_{\infty}(\sigma_{\infty}^{\ast}))$
solves \eqref{eqn:medpoly} with $\sigma_{\infty}=\sigma_{\infty}^{\ast}$. We claim that
$\sigma_\infty^{\ast}$ is a minimizer to $J_{\infty}$ over $\mathcal{A}_{\infty}$. For
any $\sigma\in\mathcal{A}_{\infty}$, the definition of $\mathcal{A}_{\infty}$ ensures the existence
of a sequence $\{\sigma_k\}\subset\bigcup_{k\geq 0}\mathcal{A}_{k}$ such that $\sigma_{k}\rightarrow
\sigma$ in $H^1(\Omega)$. By Lemma \ref{lem:weakcontmedpoly}, the sequence of solutions $(u_k(\sigma_k),U_k(\sigma_k))$ to problem
\eqref{eqn:dispoly} over $\cT_k$ satisfies
\begin{equation*}
  (u_k(\sigma_k),U_k(\sigma_k)) \rightarrow (u_{\infty}(\sigma),U_{\infty}(\sigma))\quad\mbox{weakly in}~\mathbb{H}.
\end{equation*}
By the minimizing property of $\sigma_{k}^{\ast}$ to the functional $J_{k}$ over $\mathcal{A}_{k}$,
there holds $J_{k}(\sigma^\ast_{k})\leq J_k(\sigma_k).$ Consequently,
\begin{equation*}
   J_{\infty}(\sigma^{\ast}_{\infty})\leq\liminf_{m\rightarrow\infty}J_{k_{m}}(\sigma_{k_{m}}^\ast)
        \leq\limsup_{k\rightarrow\infty}J_{k}(\sigma^\ast_k)\leq\limsup_{k\rightarrow\infty}J_k(\sigma_k)
        =J_{\infty}(\sigma)\quad\forall\sigma\in\mathcal{A}_{\infty}.
\end{equation*}
Further, by taking $\sigma=\sigma_{\infty}^{\ast}$, we derive $\lim_{m\rightarrow\infty}J_{k_{m}}
(\sigma^\ast_{k_m})=J_{\infty}(\sigma^\ast_{\infty})$, and thus
$
\lim_{m\rightarrow\infty}\|\nabla\sigma_{k_m}^\ast\|^2_{L^{2}(\Omega)}
=\|\nabla\sigma_\infty^\ast\|^2_{L^{2}(\Omega)}.
$
This and the weak convergence of $\sigma_{k_m}^*$ in $H^1(\Omega)$ shows the first assertion.
It remains to show the convergence of $\{u_{k_m}^\ast\}$ in $H^1(\Omega)$, which
follows directly from the identity $\|\nabla(u_{k_m}^\ast-u_\infty^\ast)
\|_{L^2(\Omega)}\to0$. Using the discrete problem \eqref{eqn:dispoly} over $\cT_{k_m}$, the convergence
of $\{U^\ast_{k_m}\}$ and the limiting problem \eqref{eqn:medpoly} imply
\begin{equation*}
    a(\sigma_{k_m},(u^*_{k_m},U^*_{k_m}),(u^*_{k_m},U^*_{k_m})) = \langle I,U^*_{k_m}\rangle\to  \langle I,U^*_\infty\rangle =
    a(\sigma_\infty,(u^*_\infty,U^*_\infty),(u^*_\infty,U^*_\infty)),
\end{equation*}
By the compact embedding from {the trace $H^{1/2}(\Gamma)$ of $H^1(\Omega)$} into $L^2(\Gamma)$,
the sequence $\{u^\ast_{k_m}\}$ converges to $u^\ast_{\infty}$ in $L^2(\Gamma)$,
and the convergence of $\{U^\ast_{k_m}\}$ yield
$
(\sigma_{k_{m}}^\ast\nabla u^\ast_{k_m},\nabla u^\ast_{k_m})\rightarrow
(\sigma_{\infty}^\ast\nabla u^\ast_{\infty},\nabla u^\ast_\infty).
$
By the identity
\begin{equation*}
  \begin{aligned}
    \|\sqrt{\sigma_{k_m}^\ast}\nabla(u_{k_m}^\ast-u_{\infty}^\ast)\|_{L^2(\Omega)}^2
    &=\|\sqrt{\sigma_{k_m}^\ast}\nabla u_{k_m}^\ast \|_{L^2(\Omega)}^2-2(\sigma_{k_m}^\ast\nabla u_{k_m}^\ast,\nabla u_{\infty}^\ast )+\|\sqrt{\sigma_{k_m}^\ast}\nabla u_{\infty}^\ast\|_{L^2(\Omega)}^2
  \end{aligned}
\end{equation*}
and the triangle inequality, we deduce
\begin{equation*}
  \begin{aligned}
   \|\nabla(u_{k_m}^\ast-u_\infty^*)\|_{L^2(\Omega)}^2 & \leq c(|(\sigma_{k_{m}}^\ast\nabla u^\ast_{k_m},\nabla u^\ast_{k_m})-
    (\sigma_{\infty}^\ast\nabla u^\ast_{\infty},\nabla u^\ast_\infty)| + |(\sigma_{k_m}^*-\sigma_\infty^\ast,|\nabla u_\infty^*|^2)| \\
     &\quad +|(\sigma_{k_m}^\ast\nabla u_{k_m}^\ast-\sigma_\infty^\ast\nabla u_\infty^\ast,\nabla u_\infty^\ast)|) := \mathrm{I} + \mathrm{II} + \mathrm{III}
  \end{aligned}
\end{equation*}
The second term $\textrm{II}$ tends to zero by the pointwise convergence of the sequence $\{\sigma^{\ast}_{k_{m}}\}$ and
Lebesgue's dominated convergence theorem \cite[pp. 20]{Evans:1992a}. For the third term $\textrm{III}$, there holds
\begin{equation*}
  \begin{aligned}
  \mathrm{III} & \leq |((\sigma_{k_m}^\ast-\sigma_\infty^\ast)\nabla u_{k_m}^\ast,\nabla u_\infty^\ast)|
    + |(\sigma_\infty^\ast\nabla(u_{k_m}^*-u_\infty^\ast),\nabla u_\infty^\ast)|\\
      & \leq \|\nabla u_{k_m}^\ast\|_{L^2(\Omega)}\|(\sigma_{k_m}^\ast-\sigma_\infty^\ast)\nabla u_\infty^\ast\|_{L^2(\Omega)}+
      |(\sigma_\infty^\ast\nabla(u_{k_m}^*-u_\infty^\ast),\nabla u_\infty^\ast)| \to 0
  \end{aligned}
\end{equation*}
by the weak convergence of $\{u_{k_m}^\ast\}$ in ${H}^1(\Omega)$ and the pointwise convergence of $\{\sigma_{k_m}^\ast\}$.
The preceding three estimates together complete the proof of the theorem.
\end{proof}

Next we turn to the optimality system of problem \eqref{eqn:medopt}. Like in the
continuous case, the optimality condition for the minimizer $(\sigma_{\infty}^\ast,u^\ast_{\infty},U^\ast_{\infty})$ and the
adjoint solution $(p^\ast_{\infty},P^\ast_{\infty})\in\mathbb{H}_\infty$ is given by
\begin{equation}\label{eqn:cem-medoptsys}
   \begin{aligned}
     a(\sigma^\ast_\infty,(u^*_\infty,U^*_\infty),(v,V)) = \langle I,V\rangle \quad \forall (v,V)\in \mathbb{H}_\infty,\\
     a(\sigma_\infty^*,(p^*_\infty,P_\infty^*),(v,V)) = \langle U_\infty^*-U^\delta,V\rangle\quad \forall (v,V)\in\mathbb{H}_\infty,\\
     \alpha(\nabla\sigma_{\infty}^{\ast},\nabla(\mu-\sigma_{\infty}^{\ast}))-(\nabla u_{\infty}^{\ast},\nabla p_{\infty}^{\ast}(\mu-\sigma_{\infty}^{\ast}))\geq 0\quad\forall \mu\in\mathcal{A}_{\infty}.
    \end{aligned}
\end{equation}

The next result shows the convergence of the sequence of adjoint solutions.
\begin{theorem}\label{thm:conver_medadj}
Under the condition of Theorem \ref{thm:conver_medmin}, the subsequence of
adjoint solutions $\{(p_{k_{m}}^\ast,P_{k_{m}}^\ast)\}$
generated by Algorithm \ref{alg_afem_eit} converges to the solution $(p_{\infty}^\ast,
P_{\infty}^\ast)$ to the limiting adjoint problem in \eqref{eqn:cem-medoptsys}:
\begin{equation*}
   \lim_{m\rightarrow\infty}\|(p_{k_m}^\ast-p_{\infty}^\ast,P_{k_m}^\ast-P_{\infty}^\ast)\|_{\mathbb{H},\ast}=0.
\end{equation*}
\end{theorem}
\begin{proof}
The discrete version of the limiting adjoint problem \eqref{eqn:cem-medoptsys} reads:
find $(\widetilde{p}_{k_m},\widetilde{P}_{k_m})\in\mathbb{H}_{k_m}$ such that
\begin{equation}\label{eqn:conver_medadj_bilin}
    a(\sigma_\infty^\ast,(\widetilde{p}_{k_m},\widetilde{P}_{k_m}),(v,V)) = \langle U_\infty^\ast-U^\delta,V\rangle\quad \forall (v,V)\in \mathbb{H}_{k_m}.
\end{equation}
By Cea's lemma and the construction of the space $\mathbb{H}_{\infty}$, we deduce
    \begin{equation}\label{eqn:conv_medadj_cea}
        \|(p^\ast_{\infty}-\widetilde{p}_{k_m},P^\ast_{\infty}-\widetilde{P}_{k_m})\|_{\mathbb{H},\ast}\leq c\inf_{(v,V)\in\mathbb{H}_{k_{m}}}\|(p^\ast_{\infty}-v,P^\ast_{\infty}-V)\|_{\mathbb{H},\ast}
        \rightarrow 0.
    \end{equation}
By taking
$(v_{k_m},V_{k_m})=(\widetilde{p}_{k_{m}}-p^\ast_{k_{m}},\widetilde{P}_{k_m}-P^\ast_{k_{m}})$ in the second equation of \eqref{eqn:cem-discoptsys} and  $(v,V)=(\widetilde{p}_{k_{m}}-p^\ast_{k_{m}},\widetilde{P}_{k_{m}}-P^\ast_{k_{m}})$ in \eqref{eqn:conver_medadj_bilin}, we obtain
\begin{align*}
      \|\sqrt{\sigma_{k_m}}\nabla(\widetilde{p}_{k_m}&-p_{k_m}^\ast)\|_{L^2(\Omega)}^2+
        \sum_{l=1}^Lz_l^{-1}\|\widetilde{p}_{k_m}-p^\ast_{k_m}-\widetilde{P}_{k_m,l}+P^\ast_{k_m,l}\|_{L^2(e_l)}^2\\
     = &  \langle U^\ast_{\infty}-U^\ast_{k_m},\widetilde{P}_{k_m}-P^\ast_{k_m}\rangle + ((\sigma_{k_m}^\ast-\sigma^\ast_{\infty})\nabla(\widetilde{p}_{k_m}-p^\ast_{\infty}),\nabla(\widetilde{p}_{k_m}-p_{k_m}^\ast))\\
      & +((\sigma_{k_m}^\ast-\sigma^\ast_{\infty})\nabla p_{\infty}^\ast,\nabla(\widetilde{p}_{k_m}-p_{k_m}^\ast)) := \textrm{I} + \textrm{II} + \textrm{III}.
\end{align*}
The Cauchy-Schwarz inequality and the box constraints on $\sigma_{k_m}^\ast$ and $\sigma_\infty^\ast$ give
\begin{equation*}
\begin{aligned}
   |\textrm{I}|&\leq \|U^\ast_{\infty}-U^\ast_{k_m}\|_{\mathbb{R}^L}\|\widetilde{P}_{k_m}-P^\ast_{k_m}\|_{\mathbb{R}^L},\\
   |\textrm{II}| &\leq c\|\nabla(\widetilde{p}_{k_m}-p^\ast_{\infty})\|_{L^2(\Omega)}\|\nabla(\widetilde{p}_{k_m}-p_{k_m}^\ast)\|_{L^2(\Omega)},\\
   |\mathrm{III}|& \leq\|(\sigma^{\ast}_{k_m}-\sigma^{\ast}_{\infty})\nabla p^{\ast}_{\infty}\|_{L^2(\Omega)}\|\nabla(\widetilde{p}_{k_m}-p^{\ast}_{k_m})\|_{L^{2}(\Omega)},\\
\end{aligned}
\end{equation*}
which, together with Lemma \ref{lem:normequiv}, implies
\[
    \|(\widetilde{p}_{k_m}-p_{k_m}^\ast,\widetilde{P}_{k_m}-P_{k_m}^\ast)\|_{\mathbb{H},\ast}\leq
    c(\|U^\ast_{\infty}-U^\ast_{k_m}\|_{\mathbb{R}^L}+\|\nabla(\widetilde{p}_{k_m}-p^\ast_{\infty})\|_{L^2(\Omega)}+\|(\sigma^{\ast}_{k_m}-\sigma^{\ast}_{\infty})\nabla p^{\ast}_{\infty}\|_{L^2(\Omega)}).
\]
Thanks to the convergence of $\{U_{k_m}^\ast\}$, the pointwise convergence of $\{\sigma_{k_m}^\ast\}$ in
Theorem \ref{thm:conver_medmin} and \eqref{eqn:conv_medadj_cea}, the right-hand side tends to zero. Now
the desired assertion follows from the triangle inequality and \eqref{eqn:conv_medadj_cea}.
\end{proof}

\subsection{Convergence of AFEM}\label{ssec:convergence}
Now we establish the main theoretical result of this work: the sequence of discrete solutions generated by Algorithm
\ref{alg_afem_eit} contains a convergent subsequence $\{(\sigma_{k_m}^\ast,u_{k_m}^\ast,U_{k_m}^\ast,p_{k_m}^\ast,
P_{k_m}^\ast)\}$, and the limit satisfies the optimality system \eqref{eqn:cem-optsys}. By Theorems
\ref{thm:conver_medmin} and \ref{thm:conver_medadj}, it suffices to show that the limit $\{(\sigma_{\infty}^\ast,
u_{\infty}^\ast,U_{\infty}^\ast,p_{\infty}^\ast,P_{\infty}^\ast)\}$ solves
\eqref{eqn:cem-optsys}. Our arguments begin with the observation that the maximal error indicator over marked
elements has a vanishing limit, cf. Lemma \ref{lem:estmarked}. Then we show that the sequences of
residuals with respect to $(u_{k_m}^\ast,U_{k_m}^\ast)$ and $(p_{k_m}^\ast,P_{k_m}^\ast)$ converge to zero weakly in Lemma
\ref{lem:res_weakconvzero}. This and Theorems \ref{thm:conver_medmin} and \ref{thm:conver_medadj}
verify the first two lines in \eqref{eqn:cem-optsys} in Lemma \ref{lem:vp_mc}, and the variational inequality in Lemma \ref{lem:gat_mc}.

First we show that the maximal error indicator over the marked elements has a vanishing limit.
\begin{lemma}\label{lem:estmarked}
Let $\{\cT_k,\mathcal{A}_{k}\times\mathbb{H}_{k},(\sigma^{\ast}_k,u^{\ast}_k,U^{\ast}_k,p^{\ast}_k,P^{\ast}_k)\}$
be the sequence of meshes, discrete admissible sets, finite element spaces and discrete solutions generated by
Algorithm \ref{alg_afem_eit} and $\mathcal{M}_{k}$ the set of marked elements by \eqref{eqn:marking}. Then for each
convergent subsequence $\{(\sigma^{\ast}_{k_m},u^{\ast}_{k_m},U^{\ast}_{k_m},p^{\ast}_{k_m},P^{\ast}_{k_m})\}$, there holds
\begin{equation*}
   \lim_{m\rightarrow\infty}\max_{T\in\mathcal{M}_{k_m}}\eta_{k_{m}}(\sigma^{\ast}_{k_m},u^{\ast}_{k_m},U^{\ast}_{k_m},p^{\ast}_{k_m},P^{\ast}_{k_m},T)=0.
\end{equation*}
\end{lemma}
\begin{proof}
We denote by $\widetilde{T}$ the element with the largest error indicator in $\mathcal{M}_{k_{m}}$. Since the set
$D_{\widetilde{T}}\subset\Omega_{k_m}^0$, it follows from Lemma \ref{lem:conv_zero_mesh} that
\begin{equation}\label{eqn:estmarked}
   |D_{\widetilde{T}}|\leq c\|h_{k_{m}}\|^d_{L^{\infty}(\Omega_{k_m}^0)}\rightarrow 0,\quad|\partial\widetilde{T}\cap e_l|\leq c\|h_{k_{m}}\|^{d-1}_{L^{\infty}(\Omega_{k_m}^0)}\rightarrow 0\quad\mbox{as}~m\rightarrow\infty.
\end{equation}
By Lemma \ref{lem:stab-indicator}, the local quasi-uniformity of $\cT_{k_m}$, inverse estimates, trace theorem
\cite[pp. 133]{Evans:1992a} and the triangle inequality, we have
\begin{equation*}
   \begin{aligned}
        \eta_{k_m,1}^2(\sigma^{\ast}_{k_m},u^\ast_{k_m},U^\ast_{k_m},\widetilde{T})
        &\leq c(\|\nabla u^\ast_{k_m}\|_{L^2(D_{\widetilde{T}})}^2+h_{\widetilde{T}}\|u_{k_m,l}^\ast-U_{k_m,l}^\ast\|_{L^2(\partial \widetilde{T}\cap e_l)}^2)\\
        &\leq c(\|(u^\ast_{k_m}-u^\ast_{\infty},U^\ast_{k_m}-U_{\infty}^\ast)\|_{\mathbb{H},\ast}^2+\|\nabla u^\ast_{\infty}\|_{L^2(D_{\widetilde{T}})}^2+\|u^\ast_{\infty,l}-U_{\infty,l}^\ast\|_{L^2(\partial \widetilde{T}\cap e_l)}^2),\\
        \eta_{k_m,2}^2(\sigma^{\ast}_{k_m},p^\ast_{k_m},P^\ast_{k_m},\widetilde{T})
        &\leq c(\|\nabla p^\ast_{k_m}\|_{L^2(D_{\widetilde{T}})}^2+h_{\widetilde{T}}\|p_{k_m,l}^\ast-P_{k_m,l}^\ast\|_{L^2(\partial \widetilde{T}\cap e_l)}^2)\\
        &\leq
        c(\|(p^\ast_{k_m}-p^\ast_{\infty},P^\ast_{k_m}-P_{\infty}^\ast)\|_{\mathbb{H},\ast}^2+\|\nabla p^\ast_{\infty}\|_{L^2(D_{\widetilde{T}})}^2+\|p^\ast_{\infty,l}-P_{\infty,l}^\ast\|_{L^2(\partial \widetilde{T}\cap e_l)}^2),\\
   \end{aligned}
\end{equation*}
\begin{equation*}
   \begin{aligned}
        \eta_{k_m,3}^2(\sigma^{\ast}_{k_m},u^\ast_{k_m},p^\ast_{k_m},\widetilde{T})
        &\leq c(h_{\widetilde{T}}^{4-d}\|\nabla
        u^\ast_{k_m}\|_{L^2(\widetilde{T})}^2\|\nabla
        p^\ast_{k_m}\|_{L^2(\widetilde{T})}^2+h^2_{\widetilde{T}}\|\nabla\sigma_{k_m}^\ast\|^2_{L^2(D_{\widetilde{T}})})\\
        &\leq c|D_{\widetilde{T}}|^{4/d-1}\big((\|\nabla (
        u^\ast_{k_m}-u^\ast_{\infty})\|_{L^2(\widetilde{T})}^2+\|\nabla
        u^\ast_{\infty}\|_{L^2(\widetilde{T})}^2)(\|\nabla
        (p^\ast_{k_m}-p^\ast_{\infty})\|_{L^2(\widetilde{T})}^2+\|\nabla p^\ast_{\infty}\|_{L^2(\widetilde{T})}^2)\\
        &\qquad+(\|\nabla(\sigma_{k_m}^\ast-\sigma_{\infty}^\ast)\|^2_{L^2(D_{\widetilde{T}})}
        +\|\nabla\sigma_{\infty}^\ast\|^2_{L^2(D_{\widetilde{T}})})\big).
     \end{aligned}
\end{equation*}
The desired result follows from Theorems \ref{thm:conver_medmin} and \ref{thm:conver_medadj},
\eqref{eqn:estmarked}, and the absolute continuity of the norms $\|\cdot\|_{L^2
(\Omega)}$ and $\|\cdot\|_{L^2(\Gamma)}$ with respect to the Lebesgue measure.
\end{proof}

Now we define two residuals with respect to $(u^\ast_k,U^\ast_k)$ and $(p^\ast_k,P^\ast_k)$ as
\begin{equation*}
\begin{aligned}
    \langle\mathcal{R}(u^\ast_k,U^\ast_k),(v,V)\rangle&:=a(\sigma_k^*,(u_k^*,U_k^\ast),(v,V))-\langle I,V\rangle \quad \forall (v,V)\in\mathbb{H},\\
    \langle\mathcal{R}(p^\ast_k,P^\ast_k),(v,V)\rangle&:=a(\sigma_k^*,(p_k^\ast,P_k^\ast),(v,V))-\langle U_k^\ast-U^\delta,V\rangle\quad \forall (v,V)\in\mathbb{H}.
\end{aligned}
\end{equation*}
By definition, we have the following Galerkin orthogonality
\begin{equation}\label{eqn:galortho}
   \begin{aligned}
    &\langle\mathcal{R}(p^\ast_k,P^\ast_k),(v,V)\rangle=0\quad\forall(v,V)\in\mathbb{H}_{k},\quad\\
    &\langle\mathcal{R}(p^\ast_k,P^\ast_k),(v,V)\rangle=0\quad\forall(v,V)\in\mathbb{H}_{k}.
   \end{aligned}
\end{equation}

To relate the limit
$\{(\sigma_{\infty}^\ast,u_{\infty}^\ast,U_{\infty}^\ast,p_{\infty}^\ast,P_{\infty}^\ast)\}$
to the optimality system \eqref{eqn:cem-optsys}, we exploit the marking assumption
\eqref{eqn:marking} in Algorithm \ref{alg_afem_eit}. The next result gives the weak convergence of
the residuals to zero.

\begin{lemma}\label{lem:res_weakconvzero}
For the convergent subsequence $\{(\sigma_{k_m}^\ast,u_{k_m}^\ast,U_{k_m}^\ast,p_{k_m}^\ast,P_{k_m}^\ast)\}$ given in
Theorems \ref{thm:conver_medmin} and \ref{thm:conver_medadj}, there hold
\begin{equation*}
   \begin{aligned}
    &\lim_{m\rightarrow\infty}\langle\mathcal{R}(u^\ast_{k_m},U^\ast_{k_m}),(v,V)\rangle=0\quad\forall
   (v,V)\in\mathbb{H},\\
    &\lim_{m\rightarrow\infty}\langle\mathcal{R}(p^\ast_{k_m},P^\ast_{k_m}),(v,V)\rangle=0\quad\forall
   (v,V)\in\mathbb{H}.
   \end{aligned}
\end{equation*}
\end{lemma}
\begin{proof}
We only prove the first assertion since the second follows analogously, and
relabel the index $k_m$ by $k$. Let $I_{k}$ and $I_{k}^{sz}$ be
the Lagrange and Scott-Zhang interpolation operators respectively associated with $V_{k}$. Then by \eqref{eqn:galortho}, elementwise
integration by parts and Lemma \ref{lem:est_int}, we deduce for $k>l$ and any $(\psi,V)\in
C^\infty(\overline{\Omega})\times\mathbb{R}_{\diamond}^L$
\begin{equation*}
    \begin{aligned}
    \big|\langle\mathcal{R}(u^\ast_{k},U^\ast_{k}),(\psi,V)\rangle\big|&=\big|\langle\mathcal{R}(u^\ast_{k},U^\ast_{k}),(\psi-I_{k}\psi,0)\rangle\big|
        =\big|\langle\mathcal{R}(u^\ast_{k},U^\ast_{k}),(w-I_{k}^{sz}w,0)\rangle\big|\\
        &=\Big|(\sigma^{\ast}_{k}\nabla u^{\ast}_{k}, \nabla (w-I_{k}^{sz}w)) +\sum_{l=1}^Lz_l^{-1}((u^{\ast}_{k}-U^{\ast}_{k,l}),(w-I_{k}^{sz}w))_{L^2(e_l)}\Big|\\
        &\leq c\sum_{T\in\cT_{k}}\eta_{k,1}(\sigma_{k}^\ast,u_{k}^\ast,U_{k}^\ast,T)\|w\|_{H^1({D}_{T})}\\
        &=c\big(\sum_{T\in\cT_k\setminus\cT_{l}^{+}}\eta_{k,1}(\sigma_{k}^\ast,u_{k}^\ast,U_{k}^\ast,T)\|w\|_{H^1({D}_{T})}
        +\sum_{T\in\cT_{l}^+}\eta_{k,1}(\sigma_{k}^\ast,u_{k}^\ast,U_{k}^\ast,T)\|w\|_{H^1({D}_{T})}\big).
    \end{aligned}
\end{equation*}
where $w= \psi-I_k\psi$. By appealing to Lemma \ref{lem:stab-indicator} and \eqref{stab-discpolyadj}, we deduce
\begin{equation*}
    \big(\displaystyle{\sum_{T\in\cT_k\setminus\cT_{l}^{+}}}\eta^2_{k,1}(\sigma_{k}^\ast,u_{k}^\ast,U_{k}^\ast,T)\big)^{1/2}\leq c
\end{equation*}
and further by the error estimate of the interpolation operator $I_{k}$ from Lemma \ref{lem:est_int}, we arrive at
\begin{equation*}
   \big|\langle\mathcal{R}(u^\ast_{k},U^\ast_{k}),(\psi,V)\rangle\big|\leq
    c_{1}\|h_{l}\|_{L^{\infty}(\Omega_{l}^0)}\|\psi\|_{H^2(\Omega)}+c_2
    \Big(\sum_{T\in\cT_{l}^{+}}\eta^2_{k,1}(\sigma_{k}^\ast,u_{k}^\ast,U_{k}^\ast,T)\Big)^{1/2}\|\psi\|_{H^2(\Omega)}.
\end{equation*}
By Lemma \ref{lem:conv_zero_mesh}, $c_{1}\|h_{l}\|_{L^{\infty}(\Omega_{l}^0)}\|\psi\|_{2}\rightarrow 0$ as
$l\rightarrow\infty$. From $\cT^+_l\subset\cT^+_k\subset\cT_k\subset \mathcal{M}_k$ for $k>l$ and the marking
condition \eqref{eqn:marking}, we deduce
\begin{align*}
    (\sum_{T\in\cT_{l}^{+}}\eta^2_{k,1}(\sigma_{k}^\ast,u_{k}^\ast,U_{k}^\ast,T))^{1/2}
    &\leq\sqrt{|\cT^+_l|}\max_{T\in\cT^+_l}\eta_{k,1}(\sigma_{k}^\ast,u_{k}^\ast,U_{k}^\ast,T)
    \leq\sqrt{|\cT^+_l|}\max_{T\in\cT^+_k}\eta_{k,1}(\sigma_{k}^\ast,u_{k}^\ast,U_{k}^\ast,T)\\
    &\leq\sqrt{|\cT^+_l|}\max_{T\in\mathcal{M}_k}\eta_{k}(\sigma_{k}^\ast,u_{k}^\ast,U_{k}^\ast,p_{k}^\ast,P_{k}^\ast,T).
\end{align*}
Now Lemma \ref{lem:estmarked} implies that for any fixed large $l_1$, we can choose some $k_1>l_1$ such that
\begin{equation*}
  c_2(\sum_{T\in\cT_{l}^{+}}\eta^2_{k,1}(\sigma_{k}^\ast,u_{k}^\ast,U_{k}^\ast,T))^{1/2}\|\psi\|_{2}<\varepsilon
\end{equation*}
for any positive small number $\varepsilon$ and $k>k_1$. Thus, we arrive at
\begin{equation*}
  \lim_{m\rightarrow\infty}\langle\mathcal{R}(u^\ast_{k_m},U^\ast_{k_m}),(v,V)\rangle=0\quad\forall(v,V)\in C^{\infty}(\overline{\Omega})\times\mathbb{R}_{\diamond}^L,
\end{equation*}
which, together with the density of $C^\infty(\overline{\Omega})$ in $H^1(\Omega)$, gives the desired assertion.
\end{proof}

Next we show that the limit $(\sigma^\ast_{\infty},u^\ast_{\infty},U^\ast_{\infty},p^\ast_{\infty},P^\ast_{\infty})$ actually
solves the variational equations in \eqref{eqn:cem-optsys}.
\begin{lemma}\label{lem:vp_mc}
The solution to problem \eqref{eqn:cem-medoptsys} solves the two variational equations in \eqref{eqn:cem-optsys}, i.e.,
\begin{equation*}
  \begin{aligned}
   a(\sigma^*_\infty,(u^\ast_\infty,U_\infty^\ast),(v,V)) &= \langle I, V\rangle\quad \forall (v,V)\in\mathbb{H},\\
   a(\sigma^*_\infty,(p^\ast_\infty,P_\infty^\ast),(v,V)) &= \langle U_\infty^*-U^\delta, V\rangle\quad \forall (v,V)\in\mathbb{H}.\\
  \end{aligned}
\end{equation*}
\end{lemma}
\begin{proof}
We prove only the first assertion, since the proof of the second is analogous.
Given the convergent subsequence $\{(\sigma_{k_m}^\ast,u_{k_m}^\ast,U_{k_m}^\ast,p_{k_m}^\ast,P_{k_m}^\ast)\}$
in Theorems \ref{thm:conver_medmin} and \ref{thm:conver_medadj},
for any $(v,V)\in\mathbb{H}$, there holds
\begin{equation*}
  \begin{aligned}
  \Big|a(\sigma_\infty^\ast,(u_\infty^*,U_\infty^*),(v,V))-&\langle I,V\rangle\Big|
    \leq  \sum_{l=1}^Lz_l^{-1}\Big|(u^{\ast}_{\infty}-U^{\ast}_{\infty,l}-u^{\ast}_{k_m}+U^{\ast}_{k_m,l},v-V_{l})_{L^2(e_l)}\Big|\\
    & +\Big|((\sigma^{\ast}_{\infty}\nabla u^{\ast}_{\infty} -\sigma^{\ast}_{k_m}\nabla u^{\ast}_{k_m}),\nabla v)_{L^2(\Omega)}\Big|
   +\big|\langle\mathcal{R}(u^\ast_{k_m},U^\ast_{k_m}),(v,V)\rangle\big|. 
  \end{aligned}
\end{equation*}
In view of Theorem \ref{thm:conver_medmin} and Lemma \ref{lem:res_weakconvzero}, the first
and third terms tend to zero. For the second term,
\begin{equation*}
  \begin{aligned}
  |((\sigma^{\ast}_{\infty}\nabla u^{\ast}_{\infty}-\sigma^{\ast}_{k_m}u^{\ast}_{k_m}),\nabla v)|
    & \leq |(\sigma^{\ast}_{\infty}\nabla(u^{\ast}_{\infty}-u^\ast_{k_m}),\nabla v)| + |((\sigma^{\ast}_{\infty}-\sigma^{\ast}_{k_m})\nabla
       u^{\ast}_{k_m},\nabla v)|\\
    & \leq |(\sigma^{\ast}_{\infty}\nabla(u^{\ast}_{\infty}-u^\ast_{k_m}),\nabla v)| + \|\nabla u^\ast_{k_m}\|_{L^2(\Omega)}\|(\sigma^{\ast}_{\infty}-\sigma^{\ast}_{k_m})\nabla v\|_{L^2(\Omega)}\rightarrow0,
  \end{aligned}
\end{equation*}
by the convergence of $\{u_{k_m}^\ast\}$, and the pointwise convergence of $\{\sigma_{k_m}^\ast\}$ in
Theorem \ref{thm:conver_medmin} and Lebesgue's dominated convergence theorem \cite[pp. 20]{Evans:1992a}.
\end{proof}

Now we turn to the variational inequality in \eqref{eqn:cem-optsys}. We resort again to a density argument:
we first show the assertion over a smooth subset, and then extend it to $\mathcal{A}$ by a density argument.

\begin{lemma}\label{lem:gat_mc}
The solution to the variational inequality of problem $\eqref{eqn:cem-medoptsys}$ satisfies
\begin{equation*}
  \alpha(\nabla\sigma^\ast_\infty,\nabla(\mu-\sigma^\ast_\infty))
   -(\nabla u_\infty^\ast,\nabla p_{\infty}^\ast(\mu-\sigma_{\infty}^\ast))\geq0\quad\forall\mu\in\mathcal{A}.
\end{equation*}
\end{lemma}
\begin{proof}
Like before, we relabel the index $k_m$ by $k$,
and let $I_k$ be the Lagrange interpolation operator associated with $V_k$.
Then for any $\mu\in\widetilde{\mathcal{A}}:=\mathcal{A}\cap C^\infty(\overline{\Omega})$, $I_{k}\mu\in\mathcal{A}_{k}$ and the
discrete variational inequality in \eqref{eqn:cem-discoptsys} yields
\begin{equation}
\begin{aligned}\label{eqn:gat_mc_01}
 &\alpha(\nabla\sigma^\ast_k,\nabla(\mu-\sigma^\ast_k))
    -((\mu-\sigma_{k}^\ast)\nabla u_k^\ast,\nabla p_{k}^\ast )\\
   =&\alpha(\nabla\sigma^\ast_k,\nabla(\mu-I_k\mu))
   -((\mu-I_{k}\mu)\nabla u_k^\ast,\nabla p_{k}^\ast)\\
   &\quad+\alpha(\nabla\sigma^\ast_k,\nabla(I_k\mu-\sigma_k^\ast))
   -((I_{k}\mu-\sigma_k^\ast)\nabla u_k^\ast,\nabla p_{k}^\ast)\\
   \geq&\alpha(\nabla\sigma^\ast_k,\nabla(\mu-I_k\mu))
   -((\mu-I_{k}\mu)\nabla u_k^\ast,\nabla p_{k}^\ast ).
\end{aligned}
\end{equation}
Using elementwise integration by parts, the definition of $\eta_{k,3}$ and error estimates for
$I_{k}$, cf. Lemma \ref{lem:est_int}, we deduce that for $k>l$, there holds
\begin{align*}
     \Big|\alpha(\nabla\sigma^\ast_k,\nabla(\mu-I_k\mu))
       & -((\mu-I_{k}\mu)\nabla u_k^\ast,\nabla p_{k}^\ast)\Big|
    \leq c\sum_{T\in\cT_{k}}\eta_{k,3}(\sigma_k^\ast,u_k^\ast,p_k^\ast,T)\|\mu\|_{H^2(T)}\\
    \leq& c_3\Big(\big(\sum_{T\in\cT_{k}\setminus\cT^+_l}\eta^2_{k,3}(\sigma_k^\ast,u_k^\ast,p_k^\ast,T)\big)^{1/2}
        +\big(\sum_{T\in\cT_{l}^+}\eta^2_{k,3}(\sigma_k^\ast,u_k^\ast,p_k^\ast,T)\big)^{1/2}\Big)\|\mu\|_{H^2(\Omega)}.
\end{align*}
The Lemma \ref{lem:stab-indicator}, \eqref{stab-discpolyadj}, Theorem \ref{thm:conver_medmin} and Lemma \ref{lem:conv_zero_mesh} give
\begin{align*}
    \sum_{T\in\cT_{k}\setminus\cT^+_l}\eta^2_{k,3}(\sigma_k^\ast,u_k^\ast,p_k^\ast,T)
    &\leq c(\|h_{l}\|^{4-d}_{L^\infty(\Omega_{0}^l)}\|\nabla p_k\|^2_{L^2(\Omega)}\sum_{T\in\cT_{k}\setminus\cT^+_l}\|\nabla u_{k}^\ast\|^2_{L^2(T)}+\|h_{l}\|^{2}_{L^\infty(\Omega_{0}^l)}\|\nabla\sigma_{k}^\ast\|^2_{L^2(\Omega)})\nonumber\\
   &\leq c(\|h_{l}\|^{4-d}_{L^\infty(\Omega_{0}^l)}+\|h_{l}\|^{2}_{L^\infty(\Omega_{0}^l)})\leq c\|h_{l}\|^{4-d}_{L^\infty(\Omega_{0}^l)}\rightarrow 0.
\end{align*}
Upon noting the inclusion $\cT_l^+\subset\cT_{k}$ for $k>l$, we deduce from the marking condition \eqref{eqn:marking}
\begin{equation*}
  (\sum_{T\in\cT^+_l}\eta^2_{k,3}(\sigma_k^\ast,u_k^\ast,p_k^\ast,T))^{1/2}
    \leq\sqrt{|\cT_{l}^+|}\max_{T\in\cT_l^+}\eta_{k,3}(\sigma_{k}^\ast,u_{k}^\ast,p_{k}^\ast,T)
   \leq\sqrt{|\cT^+_l|}\max_{T\in\mathcal{M}_k}\eta_{k}(\sigma_{k}^\ast,u_{k}^\ast,U_{k}^\ast,p_{k}^\ast,P_{k}^\ast,T).
\end{equation*}
Appealing again to Lemma \ref{lem:estmarked}, we can choose $k_2>l_2$ for some large fixed $l_2$ such that when $k>k_{2}$
$c_{3}({\sum_{T\in\cT_{l}^+}\eta^2_{k,3}}(\sigma_k^\ast,u_k^\ast,p_k^\ast,T))^{1/2}\|\mu\|_{H^2(\Omega)}$
is smaller than any given positive number. Hence
\begin{equation}\label{eqn:gat_mc_02}
   (\alpha\nabla\sigma^\ast_k,\nabla(\mu-I_k\mu)) -(\nabla
   u_k^\ast,\nabla p_{k}^\ast(\mu-I_{k}\mu))\rightarrow 0\quad\forall\mu\in\widetilde{\mathcal{A}}.
\end{equation}
Using the $H^1(\Omega)$-convergence of $\{\sigma_{k}^\ast\}$ from Theorem \ref{thm:conver_medmin}, we have
    \begin{equation}\label{eqn:gat_mc_03}
        (\alpha\nabla\sigma^\ast_k,\nabla(\mu-\sigma^\ast_k))
        \rightarrow(\alpha\nabla\sigma^\ast_\infty,\nabla(\mu-\sigma^\ast_\infty))\quad\forall\mu\in\widetilde{\mathcal{A}}.
    \end{equation}
The convergence of $\{p^\ast_{k}\}$ to $p_\infty^\ast$ in $H^1(\Omega)$ in Theorem \ref{thm:conver_medadj}, \eqref{stab-discpolyadj} and the box constraint in $\widetilde{\mathcal{A}}$ yield
\begin{equation*}
   (\mu\nabla u_k^\ast,\nabla (p_{k}^\ast-p_\infty^\ast))\leq c\|\nabla(p_{k}^\ast-p_\infty^\ast)\|_{L^2(\Omega)}\rightarrow 0,
\end{equation*}
and this together with Theorem \ref{thm:conver_medmin} implies
\begin{equation}\label{eqn:gat_mc_04}
    (\mu\nabla u_k^\ast,\nabla p_{k}^\ast )=(\mu \nabla u_k^\ast,\nabla (p_{k}^\ast-p_\infty^\ast))+
    (\mu\nabla u_k^\ast,\nabla p_\infty^\ast)
    \rightarrow   (\mu\nabla u^\ast_\infty,\nabla p_\infty^\ast )\quad\forall\mu\in\widetilde{\mathcal{A}}.
\end{equation}
By elementary calculations, we derive
\begin{align*}
   (\sigma_k^\ast\nabla u_{k}^\ast, \nabla p_k^\ast)-(\sigma_\infty^\ast\nabla u_{\infty}^\ast, \nabla p_\infty^\ast )&=
   (\sigma_k^\ast\nabla u_{k}^\ast, \nabla(p_k^\ast-p_\infty^\ast))
   +((\sigma_k^\ast-\sigma_{\infty}^\ast)\nabla u_{k}^\ast, \nabla p_\infty^\ast)\\
   &\quad+(\sigma_{\infty}^\ast\nabla(u_{k}^\ast-u_\infty^\ast),\nabla p_\infty^\ast).
\end{align*}
Repeating the arguments for \eqref{eqn:gat_mc_04} yields that for the first and third terms there hold
$(\sigma_k^\ast\nabla u_{k}^\ast, \nabla(p_k^\ast-p_\infty^\ast) )\rightarrow 0$ and
$(\sigma_{\infty}^\ast\nabla(u_{k}^\ast-u_\infty^\ast),\nabla p_\infty^\ast)\rightarrow 0.$
The stability estimate \eqref{stab-discpolyadj}, the pointwise convergence of $\{\sigma_{k}^\ast\}$ of
Theorem \ref{thm:conver_medmin} and Lebesgue's dominated convergence theorem \cite[pp. 20]{Evans:1992a} show
\begin{equation*}
    ((\sigma_k^\ast-\sigma_{\infty}^\ast)\nabla u_{k}^\ast, \nabla p_\infty^\ast)\leq c\|(\sigma_k^\ast-\sigma_{\infty}^\ast)\nabla p_\infty^\ast \|_{L^2(\Omega)}\rightarrow 0.
\end{equation*}
Hence
\begin{equation}\label{eqn:gat_mc_05}
   (\sigma_k^\ast\nabla u_{k}^\ast, \nabla p_k^\ast)
    \rightarrow (\sigma_\infty^\ast\nabla u_{\infty}^\ast, \nabla
    p_\infty^\ast).
\end{equation}
Now by passing both sides of \eqref{eqn:gat_mc_01} to the limit and combining \eqref{eqn:gat_mc_02}-\eqref{eqn:gat_mc_05}, we obtain
\begin{equation*}
    \alpha(\nabla\sigma^\ast_\infty,\nabla(\mu-\sigma^\ast_\infty))_{L^2(\Omega)}
    -(\nabla u^\ast_\infty,\nabla p_\infty^\ast(\mu-\sigma_\infty^\ast))_{L^2(\Omega)}\geq 0\quad\forall\mu\in\widetilde{\mathcal{A}}.
\end{equation*}
By means of the density of $C^\infty(\overline{\Omega})$ in $H^1(\Omega)$ and the construction via a
standard mollifier \cite[pp. 122]{Evans:1992a}, for any $\mu\in\mathcal{A}$ there exists a sequence $\{\mu^n\}
\subset\widetilde{\mathcal{A}}$ such that $\|\mu^n-\mu\|_{H^1(\Omega)}\rightarrow 0$ as $n\rightarrow\infty$.
Then by Lebesgue's dominated convergence theorem \cite[pp. 20]{Evans:1992a}, we deduce
\begin{equation*}
   (\alpha\nabla\sigma^\ast_\infty,\nabla \mu^n )\rightarrow
   (\alpha\nabla\sigma^\ast_\infty,\nabla\mu )\quad \mbox{and}\quad
        (\mu^n\nabla u^\ast_\infty,\nabla p_\infty^\ast )\rightarrow
        (\mu\nabla u^\ast_\infty,\nabla p_\infty^\ast)
    \end{equation*}
after possibly passing to a subsequence. The desired result follows from the preceding two estimates.
\end{proof}

Finally, by combining preceding results, we obtain the main theoretical result: the sequence of solutions generated by the
AFEM contains a subsequence converging to a solution of  \eqref{eqn:cem-optsys}.
\begin{theorem}\label{thm:conv_alg}
The sequence of discrete solutions $\{(\sigma_{k}^\ast,u_{k}^\ast,U_{k}^\ast,p_{k}^\ast,P_{k}^\ast)\}$
generated by Algorithm \ref{alg_afem_eit} has a subsequence $\{(\sigma_{k_m}^\ast,u_{k_m}^\ast,U_{k_m}^\ast,p_{k_m}^\ast,P_{k_m}^\ast)\}$ converging to a solution $(\sigma^\ast,u^\ast,U^\ast,p^\ast,P^\ast)$ to the continuous optimality system
\eqref{eqn:cem-optsys} in the following sense:
\begin{equation*}
   \|\sigma^\ast_{k_m}-\sigma^\ast\|_{H^1(\Omega)},~\|(u^\ast_{k_m}-u^\ast,U^\ast_{k_m}-U^\ast)\|_{\mathbb{H},\ast},~\|(p^\ast_{k_m}-p^\ast,P^\ast_{k_m}-P^\ast)\|_{\mathbb{H},\ast}\rightarrow 0\quad\mbox{as}~m\rightarrow\infty.
\end{equation*}
\end{theorem}

\begin{remark}\label{rmk:practical-mark}{
Theorem \ref{thm:conv_alg} is only concerned with the convergence of the adaptive solution
to the continuous Tikhonov solution, which is limited by the data accuracy (i.e.,
the noise level $\delta$) and regularization parameter $\alpha$. In the spirit of
the classical discrepancy principle \cite{ItoJin:2014}, it is unnecessary to make
the adaptive FEM approximation of the forward model far more accurate than the data accuracy. In practice,
it is advisable to terminate the refinement step when the estimator $\eta_k$ falls below a
multiple of the noise level $\delta$, however, the regularizing property (and the convergence
rate) of such a procedure is still to be studied.}
\end{remark}

\section{Numerical experiments and discussions}\label{sec:numer}

In this section, we present numerical results to illustrate the convergence and efficiency of the adaptive algorithm. {All
the computations were carried out using \texttt{MATLAB} 2013a on a personal laptop with 6.00 GB RAM and 2.5 GHz CPU.}
The setup of the numerical experiments is as follows. The domain $\Omega$ is taken to be a square $\Omega=(-1,1)^2$.
There are sixteen electrodes $\{e_l\}_{l=1}^L$ (with $L=16$) evenly distributed along the boundary $\Gamma$,
each of the length $1/4$, thus occupying one half of the boundary $\Gamma$. The contact impedances
$\{z_l\}_{l=1}^L$ on the electrodes $\{e_l\}_{l=1}^L$ are all set to unit, and the background
conductivity $\sigma_0$ is taken to be $\sigma_0\equiv 1$. For each example, we measure the electrode
voltages $U$ for the first ten sinusoidal input currents, in order to gain enough information about the true
conductivity $\sigma^\dag$. Then the noisy data $U^\delta$ is generated by adding componentwise Gaussian noise to
the exact data $U(\sigma^\dag)$ as follows
\begin{equation*}
  U^\delta_l = U_l(\sigma^\dag) + \epsilon \max_l|U_l(\sigma^\dag)|\xi_l,\ \ l=1,\ldots, L,
\end{equation*}
where $\epsilon$ is the (relative) noise level, and $\{\xi_l\}$ follow the standard normal distribution. {
The exact data $U(\sigma^\dag)$ is computed on a much finer mesh generated adaptively (and thus completely
different from the one used in the inversion), in order to avoid the most obvious form of ``inverse crime". }In all the
experiments, the marking strategy \eqref{eqn:marking} in the module \texttt{MARK} is represented by a
specific maximum strategy, cf. Remark 3.2, i.e., mark a minimal subset $\mathcal{M}_k\subseteq\cT_k$, i.e., the refinement set, such that
\begin{equation*}
   \eta_{k}(\sigma_k^{\ast},u_k^{\ast},U_k^{\ast},p_k^{\ast},P_k^{\ast},\mathcal{M}_k)
        \geq\theta 
        \eta_{k}(\sigma_k^{\ast},u_k^{\ast},U_k^{\ast},p_k^{\ast},P_k^{\ast},\cT_k),
\end{equation*}
with a threshold $\theta\in(0,1]$. In the computation, we fix the threshold $\theta$ at $\theta=0.7$.
{For the adaptive refinement, we employ the newest vertex bisection to subdivide the marked
triangles; see \cite{Mitchell:1989} for implementation details.} The discrete nonlinear optimization problems
\eqref{eqn:dispoly}-\eqref{eqn:discopt} are solved by a nonlinear conjugated gradient method, {
where the box constraints are enforced by pointwise projection into the admissible set $\mathcal{A}$ after each update}, and
the initial guess of the conductivity at the coarsest mesh $\cT_0$ is initialized to the background conductivity $\sigma_0=1$,
and then for $k=1,2,\ldots$, the recovery on the mesh $\cT_{k-1}$ is interpolated to the mesh $\cT_k$ to warm start the
(projected) conjugate gradient iteration for the discrete optimization problem on the mesh $\cT_k$. Throughout the
adaptive loop, the regularization parameter $\alpha$
in the model \eqref{eqn:tikh} is fixed and determined in a trial-and-error manner, {and the
chosen values of $\alpha$ in the experiments below are roughly of the order of the noise level $\delta$,
which is a popular a priori parameter choice;} see \cite{ItoJin:2014} for further discussions about parameter
choice. It is an interesting research question to adapt the choice of $\alpha$ with the a posterior estimator $\eta_k$
within the adaptive algorithm; see Remark \ref{rmk:practical-mark}.

\begin{example}\label{exam1}
The true conductivity $\sigma^\dag$ is given by $\sigma^\dag(x) = \sigma_0(x) + e^{-8(x_1^2+(x_2-0.55)^2)}$, with the background conductivity $\sigma_0(x)=1$.
\end{example}

In this example, the true conductivity $\sigma^\dag$ consists of a very smooth blob in a constant background,
and the profile is shown in Fig. \ref{fig:exam1-recon}(a). The final recovered conductivity fields from the voltage
measurements with $\epsilon=0.1\%$ data noise are shown in Fig. \ref{fig:exam1-recon}. For both uniform and adaptive
refinements, the recoveries capture well the location and height of the blob: it is very smooth,
due to the use of a smoothness prior. Hence, it does not induce any grave solution singularity. The
recoveries by both methods are similar to each other in terms of location
and magnitude. {Both suffer from a slight loss of the contrast,
which is typical for EIT recoveries with a smoothness penalty; see, e.g.,
\cite{LechleiterRieder:2006} and \cite{WinklerRieder:2014} for similar results by an iteratively regularized Gauss-Newton method.}

\vspace{-.2cm}
\begin{figure}[hbt!]
  \centering \setlength{\tabcolsep}{0em}
  \begin{tabular}{ccc}
   \includegraphics[trim = 0.5cm 0cm 1cm 0cm, clip=true,width=0.333\textwidth]{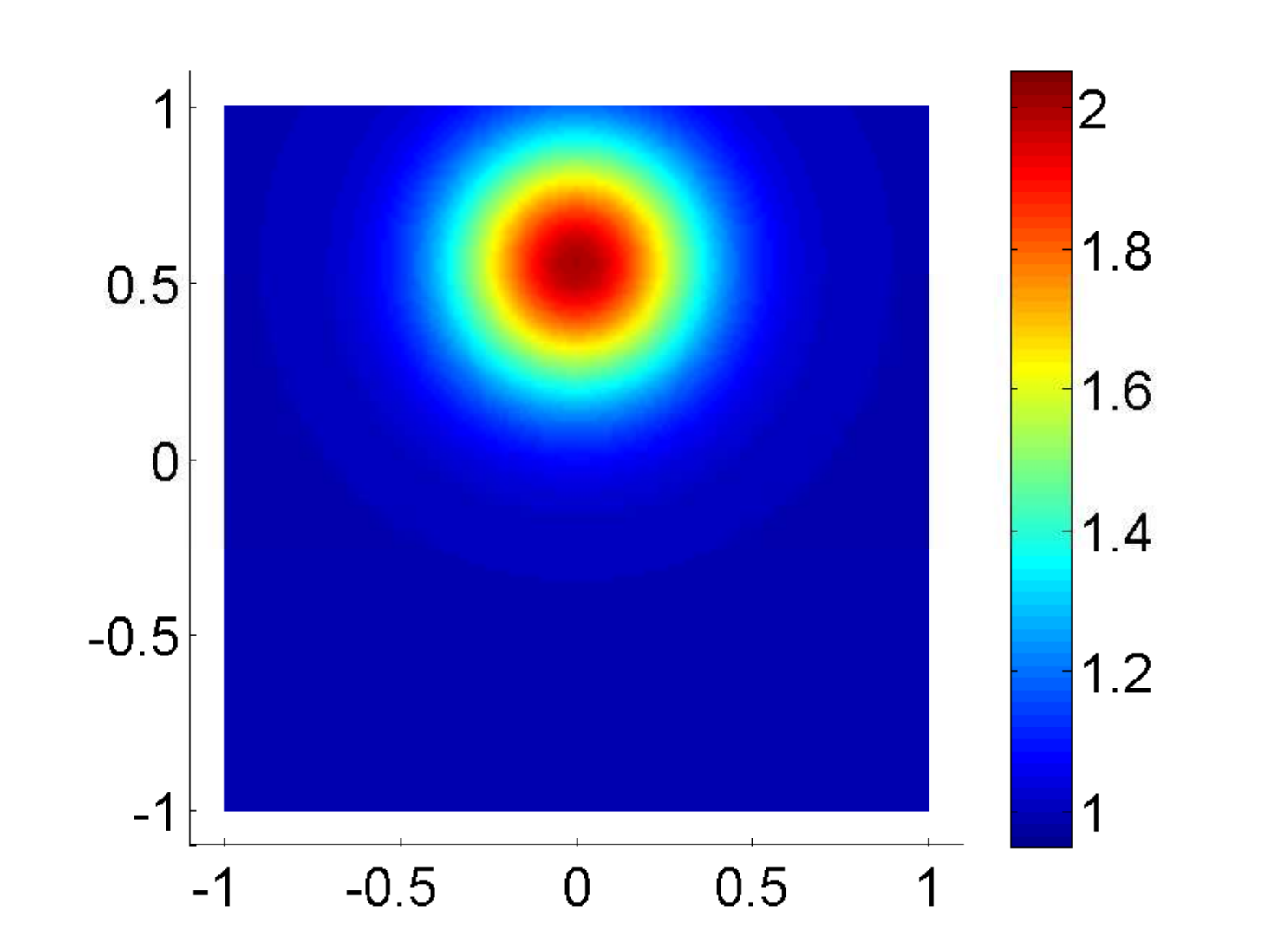} &
   \includegraphics[trim = 0.5cm 0cm 1cm 0cm, clip=true,width=.333\textwidth]{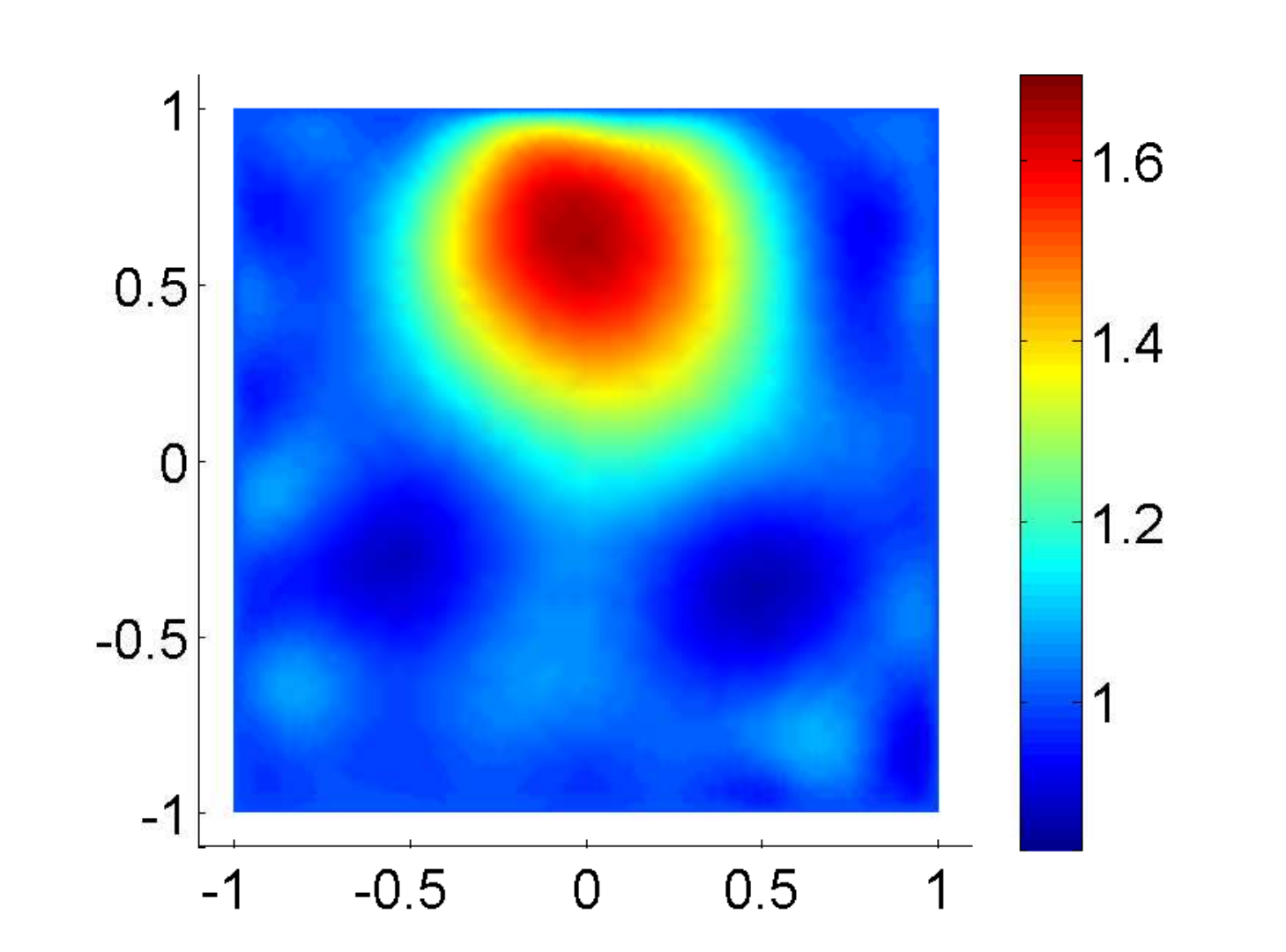} &
   \includegraphics[trim = .5cm 0cm 1cm 0cm, clip=true,width=.333\textwidth]{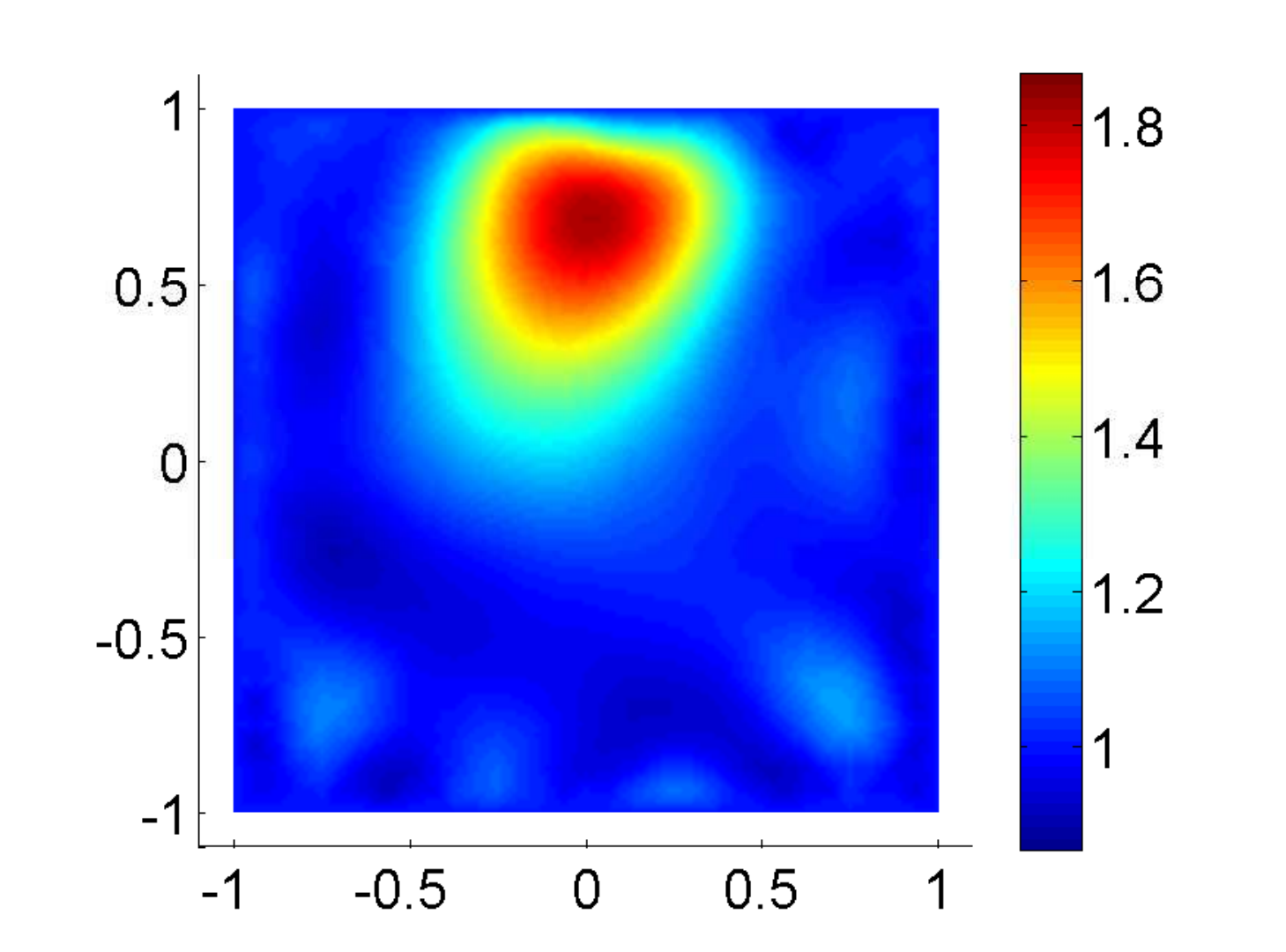} \\
 (a) true conductivity & (b) adaptive refinement & (c) uniform refinement
  \end{tabular}
  \caption{The final reconstructions by the uniform and adaptive refinements for Example
  \ref{exam1} with $\epsilon=0.1\%$ noise in the data. The degree of freedom is $9818$ and $ 16641$
  for the adaptive and uniform refinement, respectively. The regularization parameter $\alpha$ is fixed at
  $\alpha=2.5\times10^{-4}$.}\label{fig:exam1-recon}
\end{figure}

Next we examine the adaptive refinement more closely. On a very coarse initial mesh $\cT_0$, which is a uniform
triangulation of the domain $\Omega$, cf. Fig. \ref{fig:exam1-recon-iter}(a), the recovered conductivity tends
to have pronounced oscillations around the boundary, since the forward solution is not accurately resolved over there. In
particular, the discretization error significantly compromises the reconstruction accuracy, and it induces large errors in
the location and height of the recovered conductivity. This motivates the use of the adaptive strategy. The meshes during
the adaptive iteration and the corresponding recoveries are shown in Fig. \ref{fig:exam1-recon-iter}. The refinement
step first concentrates only on the region around the electrode surface. This is attributed to the change of the
boundary condition, which induces weak singularities in the direct and adjoint solutions. Then the AFEM starts to refine
also the interior of the domain, simultaneously with the boundary region. Accordingly, the spurious oscillations
in the recovery are suppressed as the iteration proceeds (provided that the regularization parameter
$\alpha$ is properly chosen). Interestingly, the central part of the domain $\Omega$ is refined only slightly during the
whole refinement procedure, and in the end, much coarse elements are used for the conductivity inversion in these regions.
This concurs with the empirical observation that the inclusion in the central part is much harder to resolve from the
boundary data. Hence, the adaptive algorithm tends to adapt automatically to the resolving power of the conductivity
(from the boundary data) in different regions.

In Fig. \ref{fig:exam1-efficiency}, we plot the $L^2(\Omega)$ and $H^1(\Omega)$ errors of the recoveries
versus the degree of freedom $N$ of the mesh $\cT_k$ for the adaptive and uniform refinement, where
the recovery on the finest mesh is taken as a respective reference solution, since the recoveries by
the uniform and adaptive refinements are not necessarily the same (although always close), even initialized
identically. The corresponding empirical convergence rates in $L^2(\Omega)$-norms and $H^1(\Omega)$-norms are given in Table \ref{tab:convrate}. It is observed
that with the same degree of freedom, the AFEM can give much more
accurate results than the uniform one (with respect to the respective reference solution). This is also
confirmed by the computing time: for the results in Fig. \ref{fig:exam1-recon}, the one by the
adaptive refinement takes about 30 minutes, whereas that by the uniform refinement takes about 80 minutes. This
is consistent with the fact that at each iteration of the algorithm, the module \texttt{SOLVE} is
predominant, and that the computational cost of the conjugate gradient descent
algorithm is proportional to the number of forward and adjoint solves at each iteration and
each forward/adjoint solve is determined by the degree of freedom of the system.
This shows clearly the computational efficiency of the proposed adaptive algorithm.

\begin{figure}[hbt!]
 \centering
 \setlength{\tabcolsep}{1pt}
 \begin{tabular}{ccccc}
 \includegraphics[trim = 0cm 0cm 0cm 0cm, width=.23\textwidth]{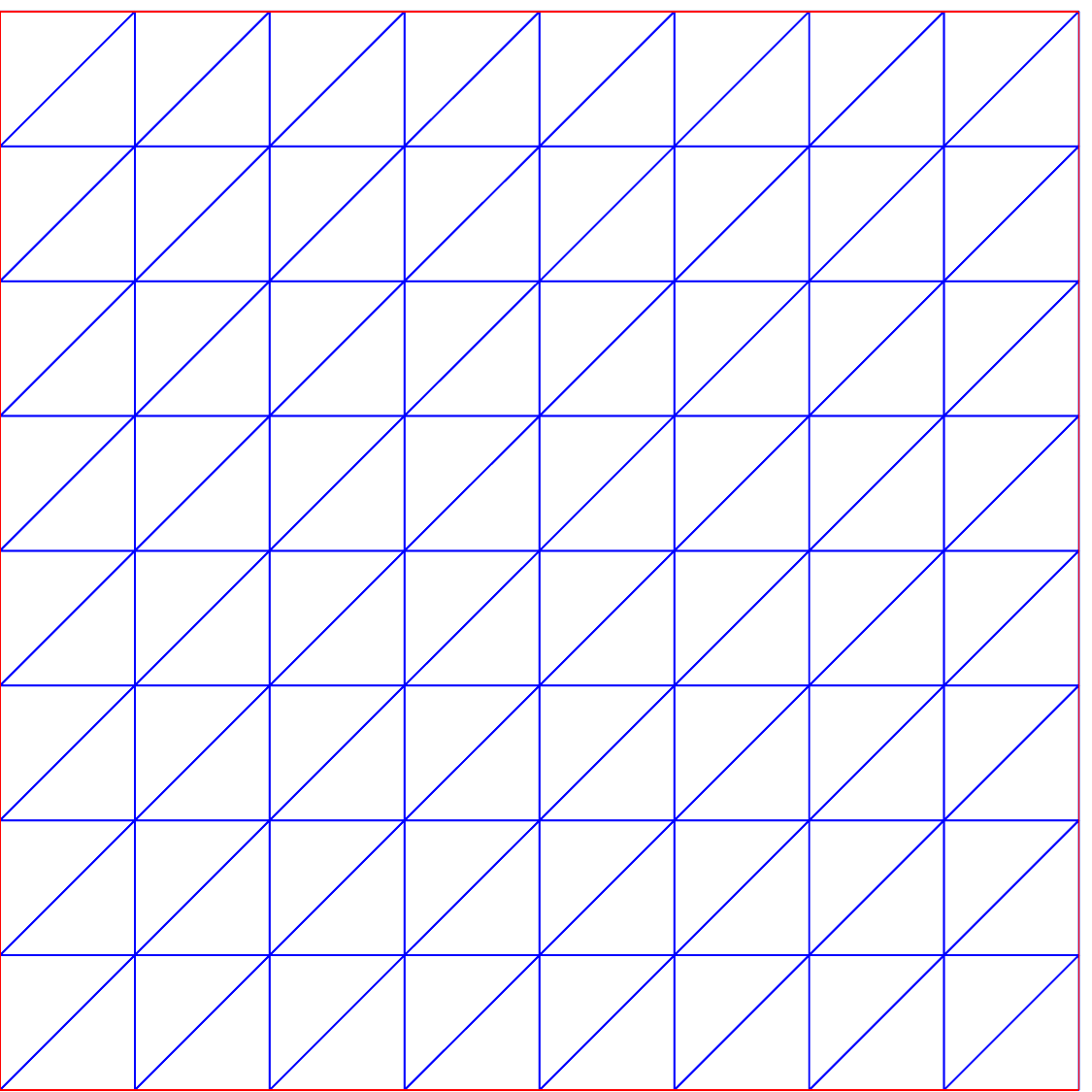}&
 \includegraphics[trim = 0cm 0cm 0cm 0cm, width=.23\textwidth]{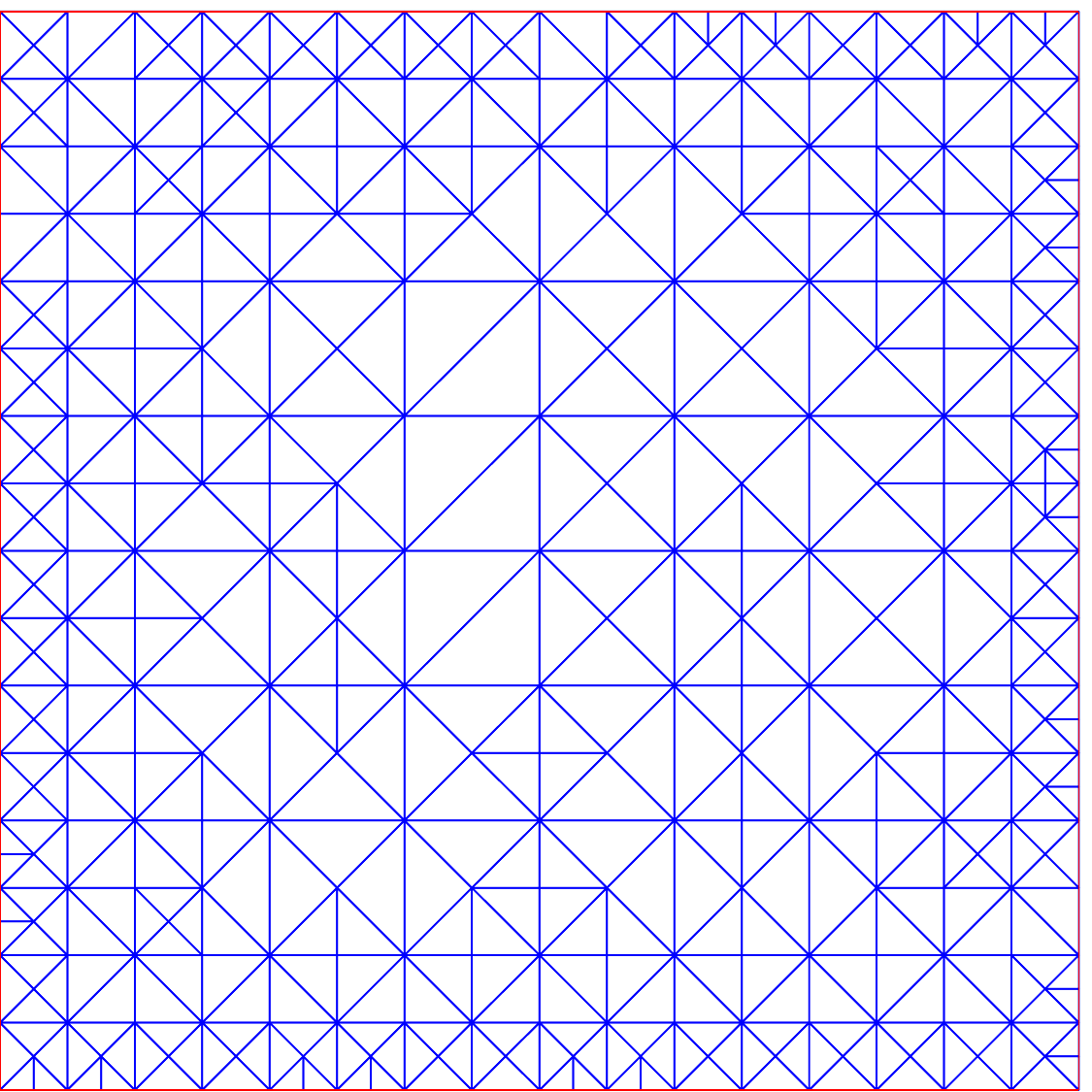}&
 \includegraphics[trim = 0cm 0cm 0cm 0cm, width=.23\textwidth]{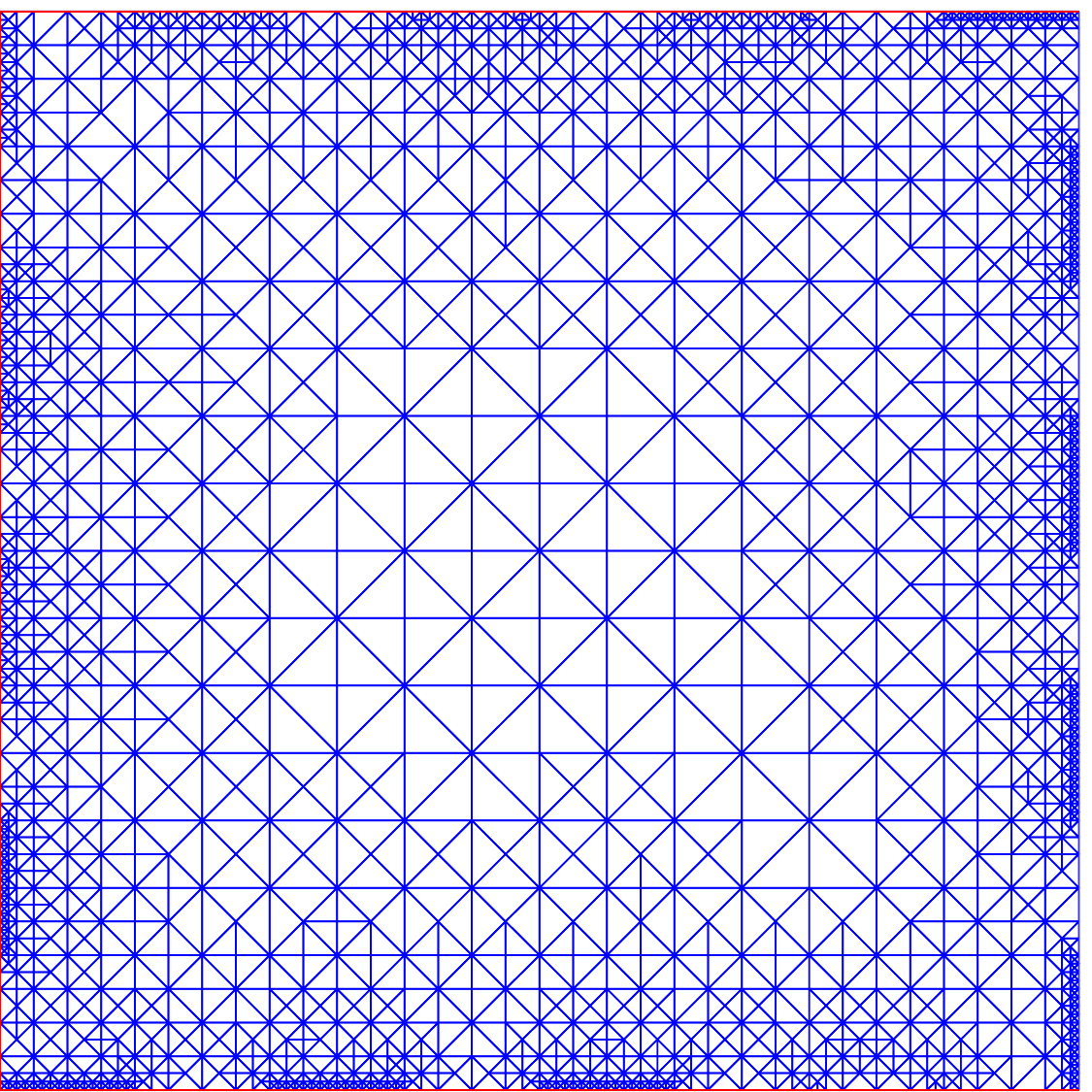}&
 \includegraphics[trim = 0cm 0cm 0cm 0cm, width=.23\textwidth]{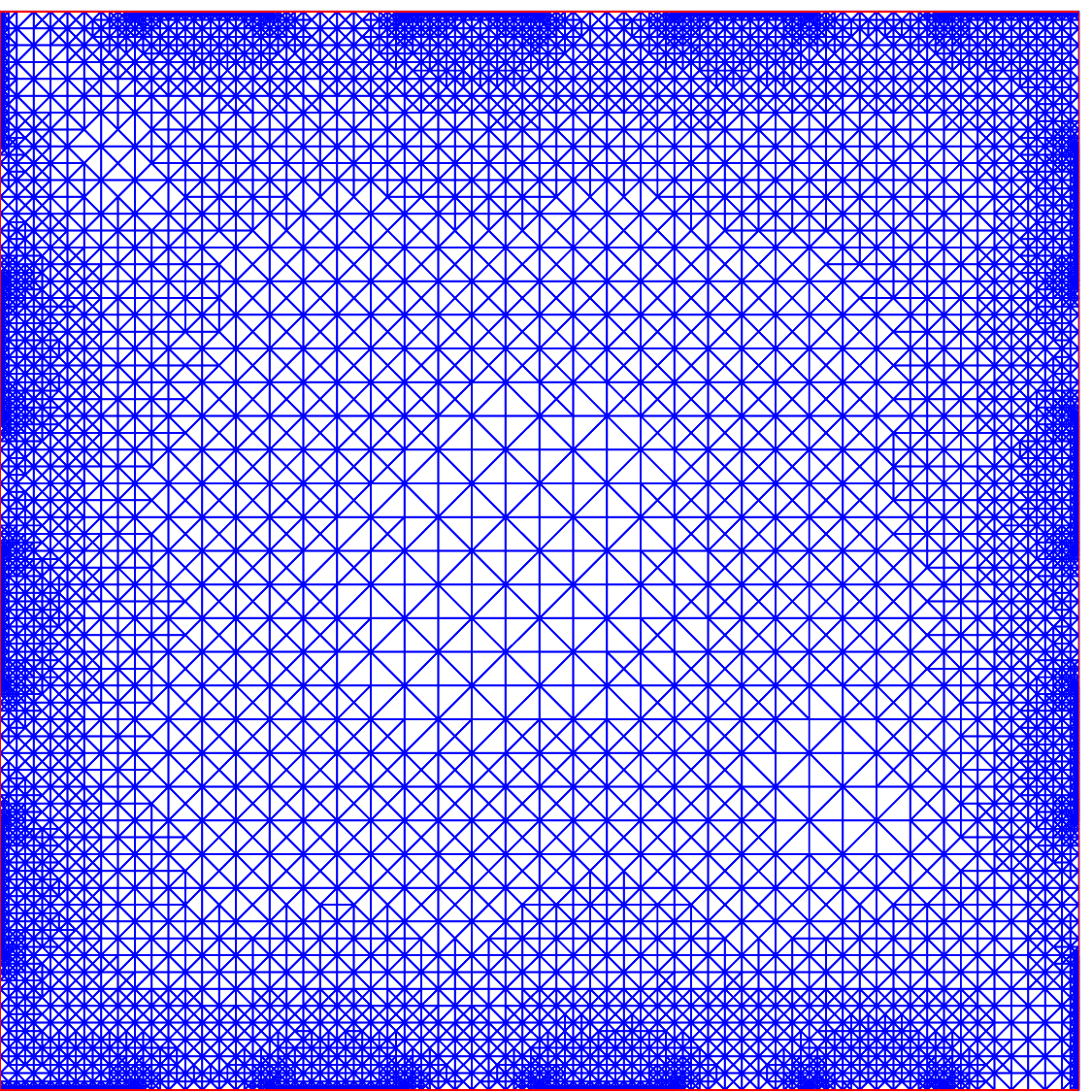}\\
  \includegraphics[trim = 1cm 0.5cm 0.5cm 0cm, width=.24\textwidth]{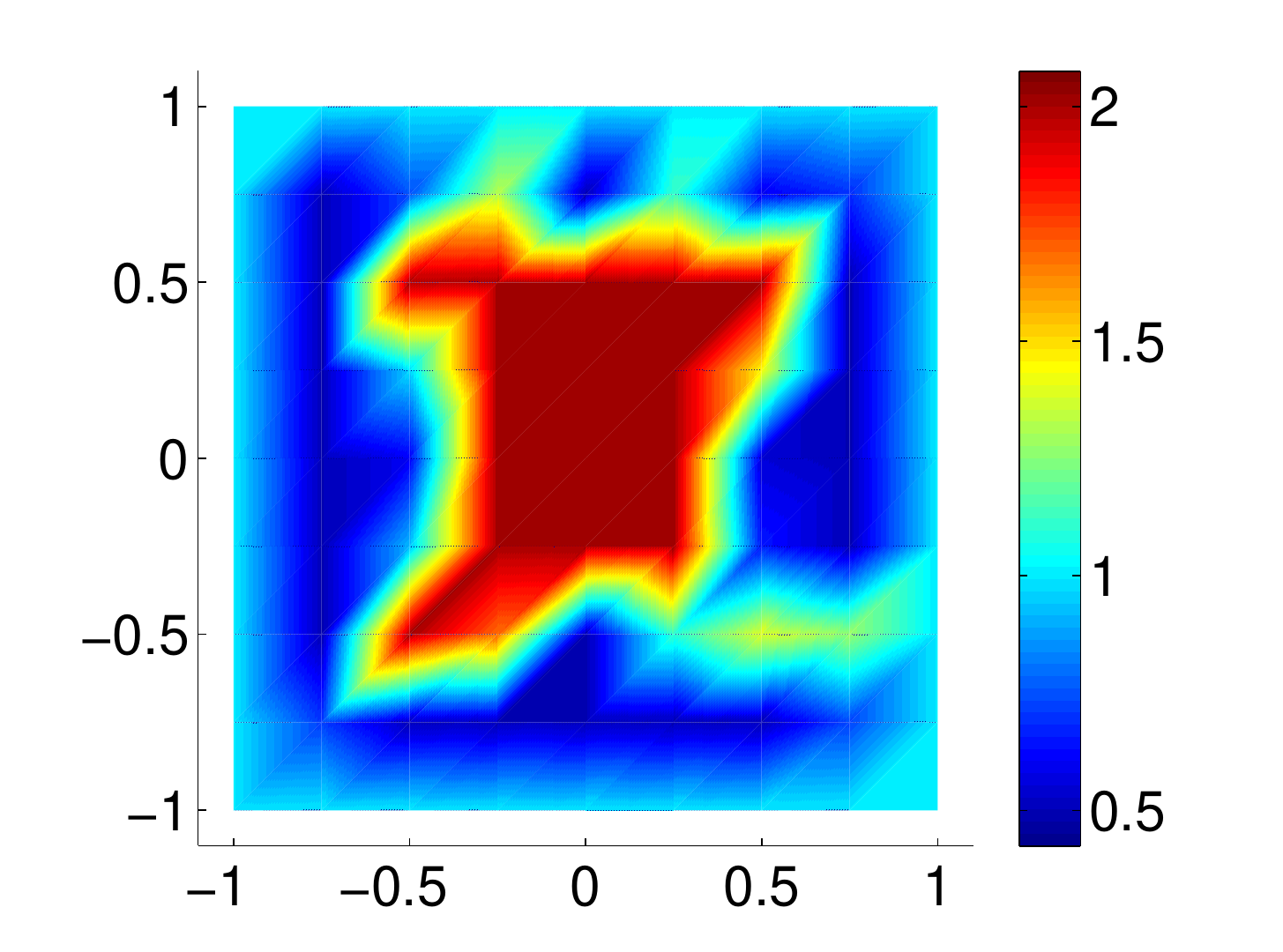} &
  \includegraphics[trim = 1cm 0.5cm 0.5cm 0cm, width=.24\textwidth]{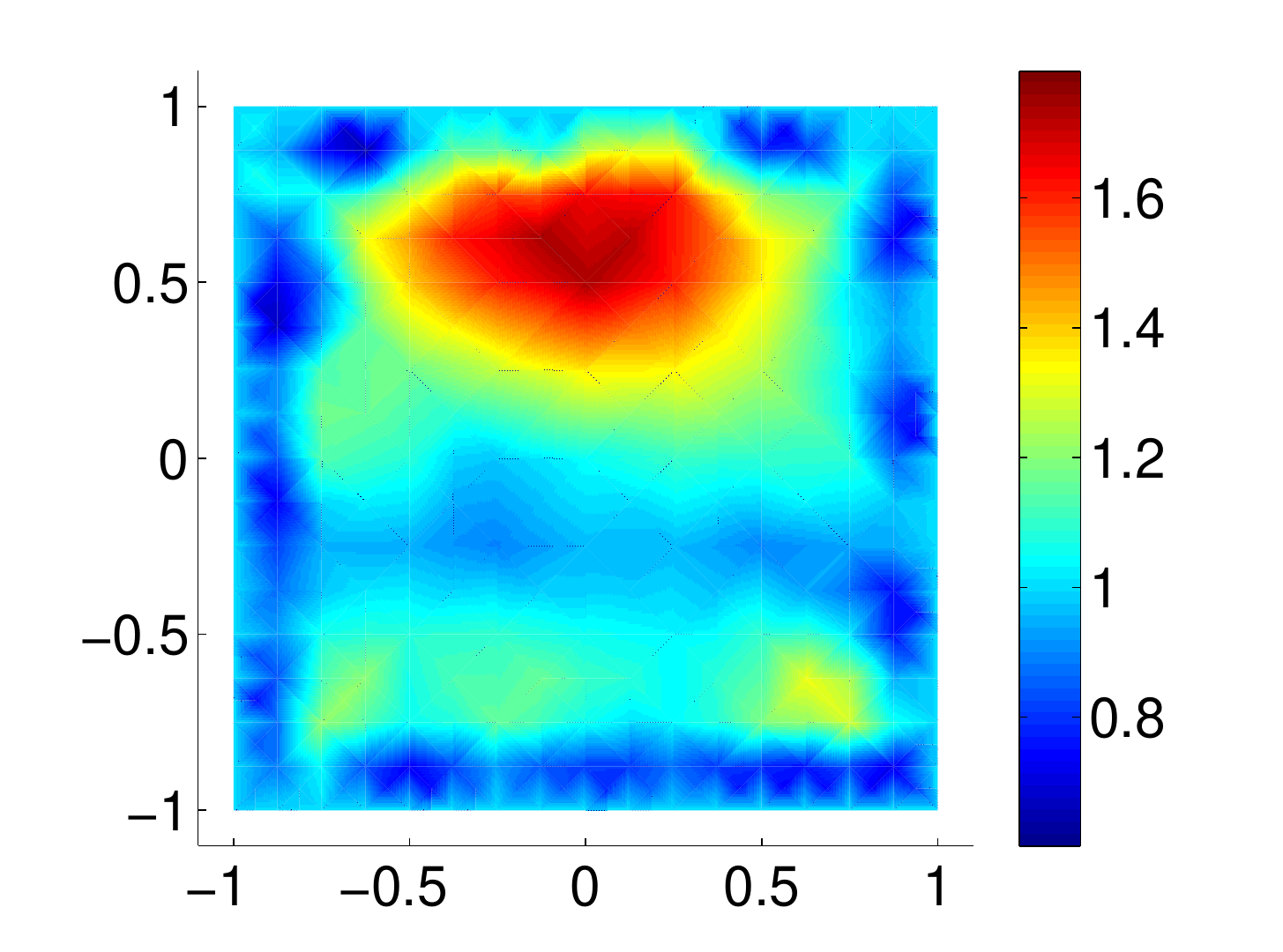} &
  \includegraphics[trim = 1cm 0.5cm 0.5cm 0cm, width=.24\textwidth]{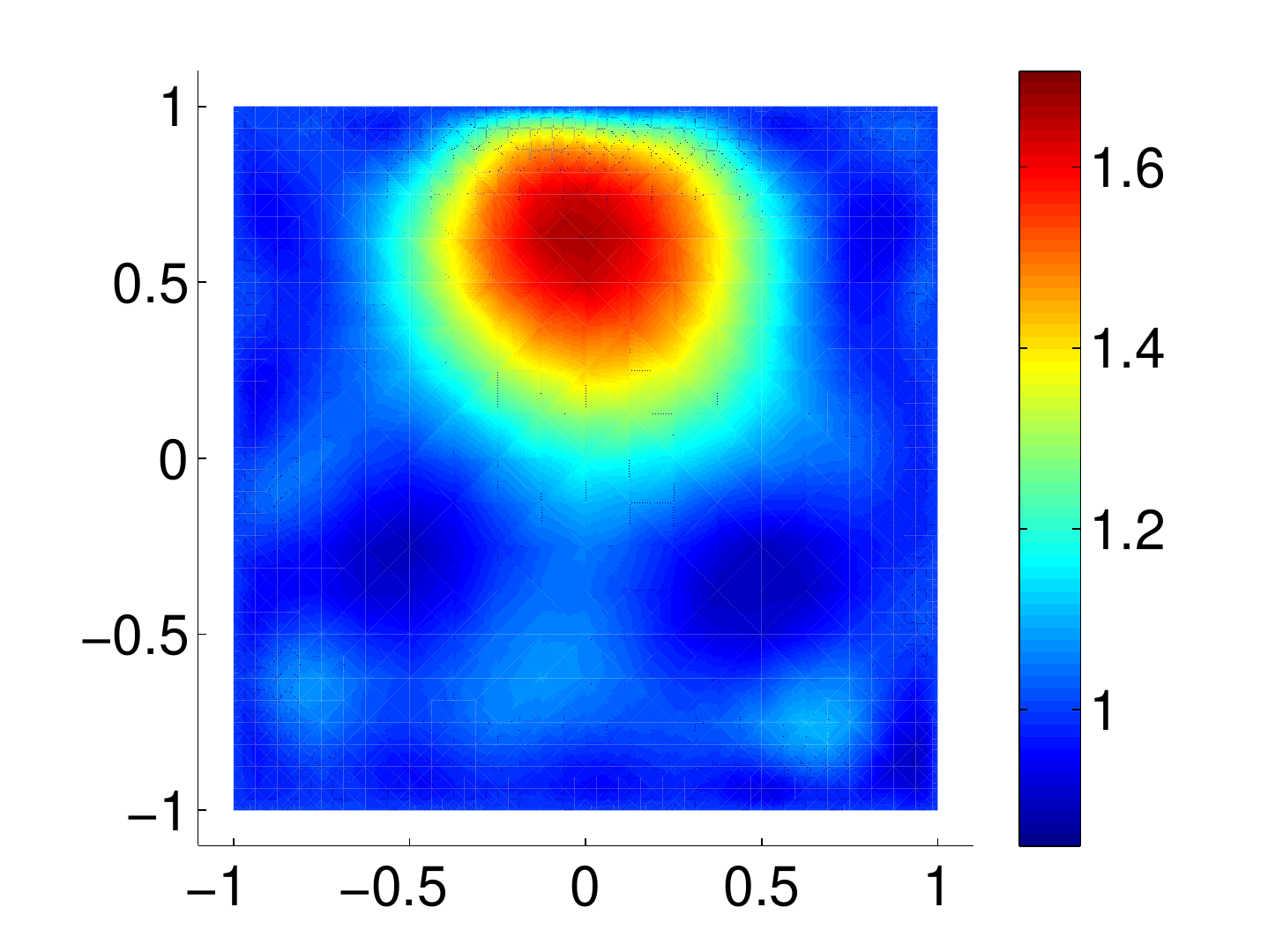} &
  \includegraphics[trim = 1cm 0.5cm 0.5cm 0cm, width=.24\textwidth]{ex1_1e3ad_iter15}\\
    (a) 0th step & (b) 4th step & (c) 9th step & (d) 14th step
 \end{tabular}
 \caption{The recovered conductivity distributions $\sigma_k^*$ during the adaptive refinement, for Example \ref{exam1} with $\epsilon=0.1\%$ noise.
 The regularization parameter $\alpha$ is fixed at $\alpha = 2.5\times10^{-4}$.}\label{fig:exam1-recon-iter}
\end{figure}

\begin{figure}[hbt!]
  \centering
  \begin{tabular}{cc}
    \includegraphics[trim = 1cm 0cm 1.5cm 0cm, clip=true,width=.35\textwidth]{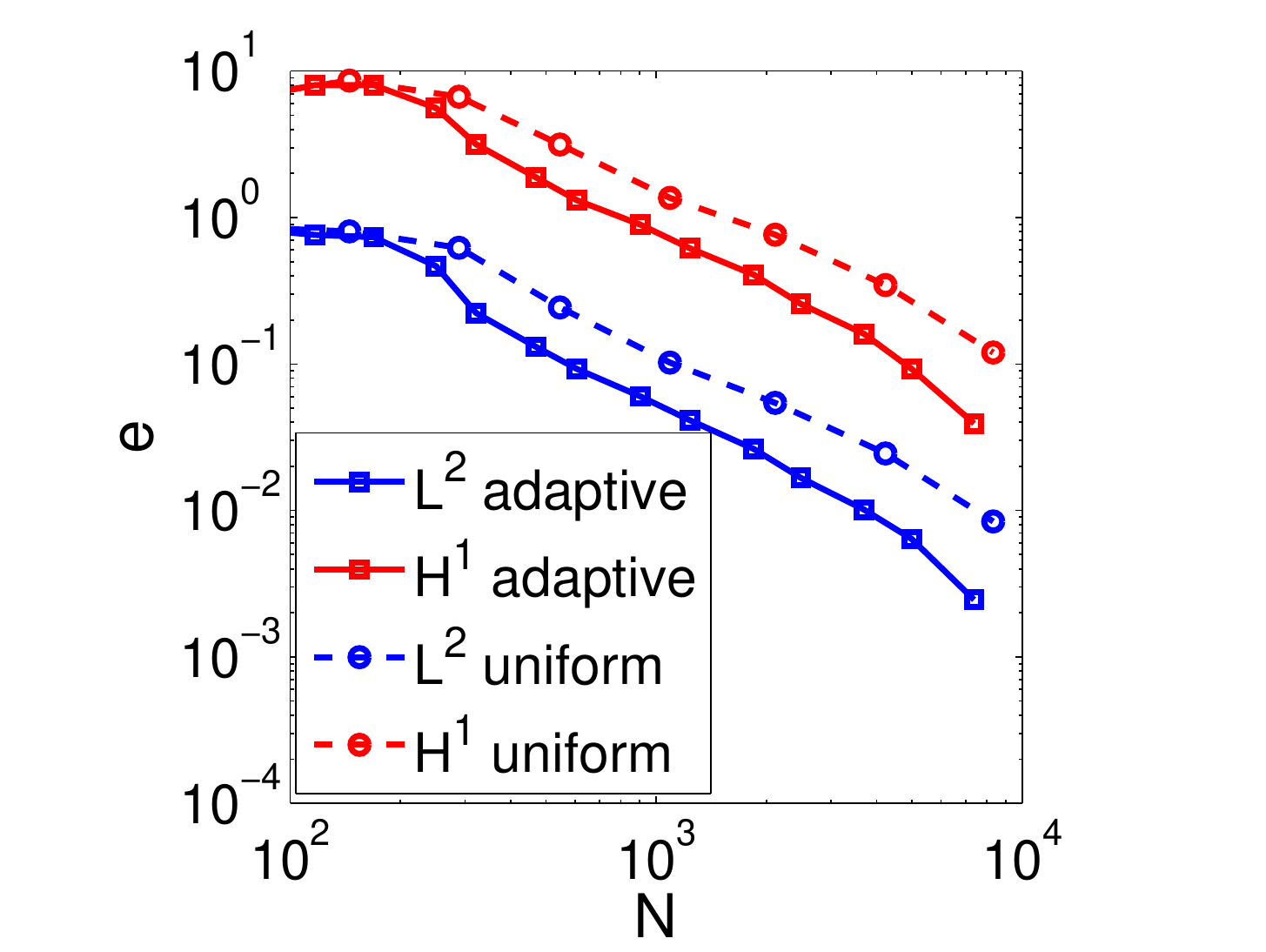}
    & \includegraphics[trim = 1cm 0cm 1.5cm 0cm, clip=true,width=.35\textwidth]{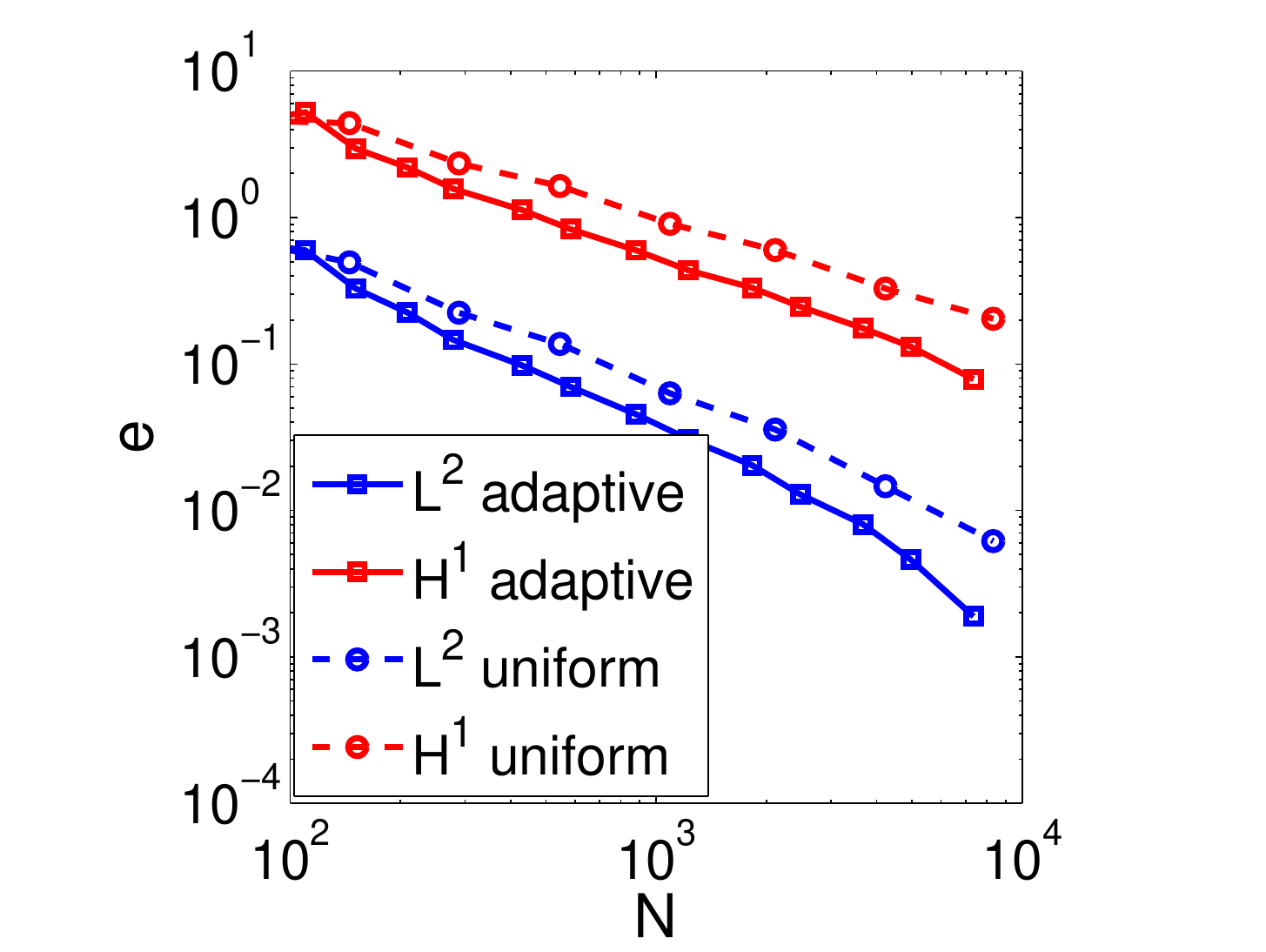}\\
    (a) $\epsilon=1\times10^{-3}$ & (b) $\epsilon=1\times10^{-2}$
  \end{tabular}
  \caption{The $L^2(\Omega)$ and $H^1(\Omega)$ errors versus the degree of freedom $N$ of the mesh, for Example \ref{exam1} at two different noise levels, using the adaptive refinement (solid line) and uniform refinement (dashed line).}\label{fig:exam1-efficiency}
\end{figure}

\begin{table}[hbt!]
  \caption{The empirical convergence rates $O(N^{-r})$, $N$ being the degree of freedom, of the recoveries in the $L^2(\Omega)$- and $H^1(\Omega)$-norms, for the numerical examples, where the exponent $r$ is presented.\label{tab:convrate}}
  \centering \setlength{\tabcolsep}{9pt}
  \begin{tabular}{|c|cc|cc|cc|cc|}
  \hline
   & \multicolumn{4}{|c|}{$\epsilon$=1e-3}& \multicolumn{4}{c|}{$\epsilon$=1e-2} \\
   \cline{2-5} \cline{6-9}
   Example  & \multicolumn{2}{|c|}{adaptive}& \multicolumn{2}{|c|}{uniform}& \multicolumn{2}{|c|}{adaptive}& \multicolumn{2}{|c|}{uniform}\\
   \cline{2-9}
            & $L^2$ & $H^1$ & $L^2$ & $H^1$ & $L^2$ & $H^1$ & $L^2$ & $H^1$\\
   \hline
   \ref{exam1} &  1.31 & 1.19 & 1.04 & 0.93  & 1.23  & 0.91 & 1.01  & 0.70 \\
   \ref{exam2} &  1.32 & 1.19 & 1.05 & 0.94 & 1.23 & 0.88 & 0.99 & 0.73\\
   \ref{exam3} &  1.08 & 0.88 & 0.83 & 0.73 & 0.91 & 0.67 & 0.72 & 0.40\\
   \hline
  \end{tabular}
\end{table}

A second example contains two neighboring smooth blobs.
\begin{example}\label{exam2}
The true conductivity $\sigma^\dag$ is given by $\sigma^\dag(x) = \sigma_0(x)
 +  e^{-20((x_1+0.7)^2+x_2^2)} + e^{-20(x_1^2+(x_2-0.7)^2)}$, and the background
conductivity $\sigma_0(x)=1$.
\end{example}

Like before, the true conductivity $\sigma^\dag$ is smooth (cf. Fig. \ref{fig:exam2-recon}(a) for the profile), and thus
the smoothness penalty is suitable. Overall, the observations from Example \ref{exam1} remain valid: the
recovered coefficient captures very well the supports of the inclusions, and the magnitude is also reasonable.
The recovery by the adaptive algorithm is comparable with that based on uniform one, but requiring
far less degrees of freedom. However, due to the smoothing nature of the $H^1(\Omega)$ penalty,
the recoveries tend to be diffusive, and the magnitude also suffers from a loss of about $20\%$ for
both uniform and adaptive refinements. Such smoothing is well-known in EIT imaging.
These drawbacks can be partially alleviated by sparsity-promoting
penalty \cite{JinKhanMaass:2012,JinMaass:2012}, to which it is of great interest to extend the proposed AFEM.

\begin{figure}[hbt!]
  \centering\setlength{\tabcolsep}{0em}
  \begin{tabular}{ccc}
    \includegraphics[trim = .5cm 0cm 1cm 0cm,clip=true,width=0.33\textwidth]{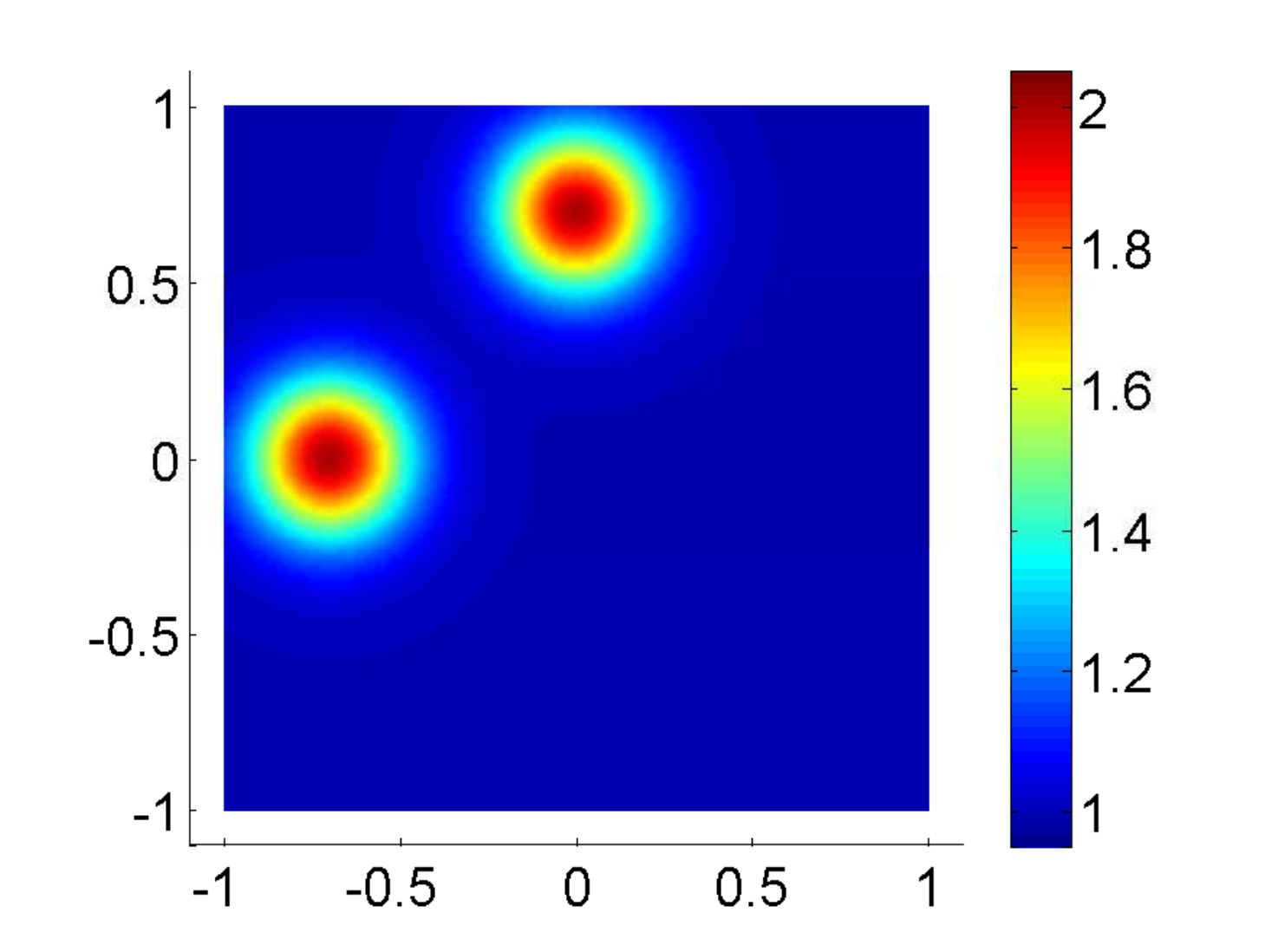}
    & \includegraphics[trim = .5cm 0cm 1cm 0cm, clip=true,width=.33\textwidth]{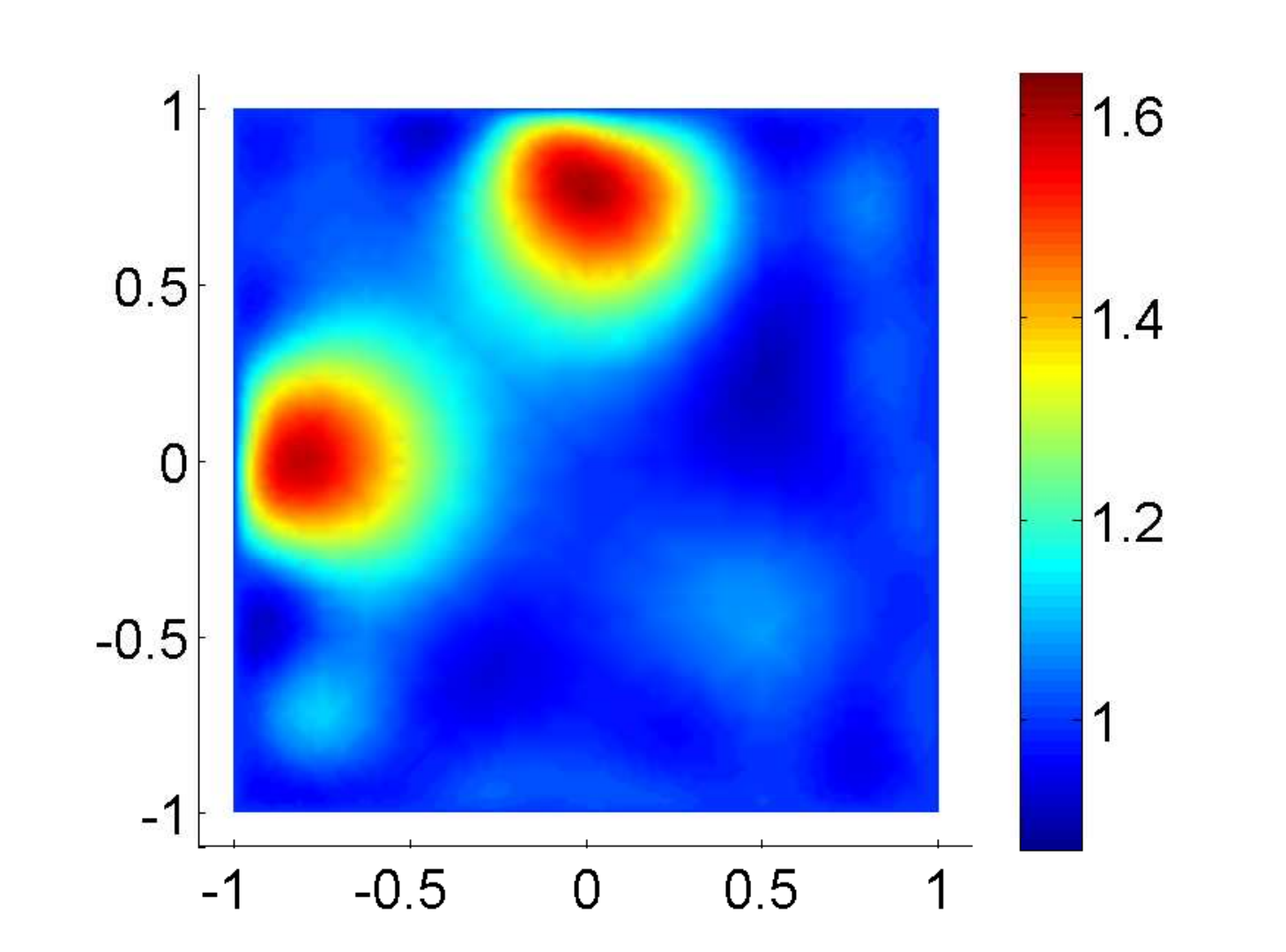}
    & \includegraphics[trim = .5cm 0cm 1cm 0cm, clip=true,width=.33\textwidth]{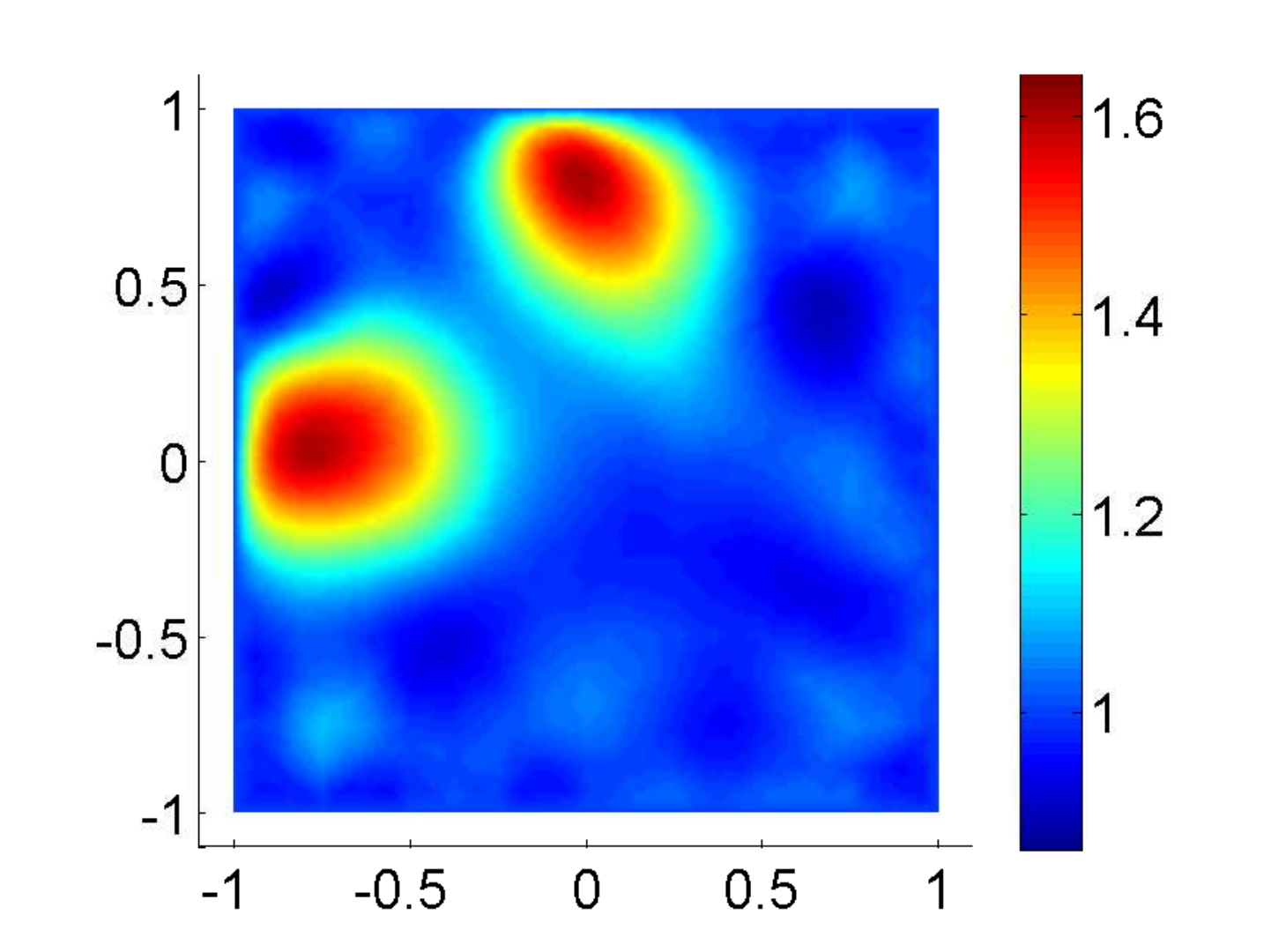}\\
    (a) true conductivity & (b) adaptive refinement & (c) uniform refinement
  \end{tabular}
  \caption{The final reveries by the adaptive and uniform refinements for Example
  \ref{exam2} with $\epsilon=0.1\%$ noise. The degree of freedom is $9803$ and $16641$
  for the adaptive and uniform refinement, respectively. The regularization parameter $\alpha$
  is fixed at $\alpha=2.5\times10^{-4}$.}\label{fig:exam2-recon}
\end{figure}

\begin{figure}[hbt!]
 \centering
 \setlength{\tabcolsep}{1pt}
 \begin{tabular}{ccccc}
 \includegraphics[trim = 0cm 0cm 0cm 0cm, width=.23\textwidth]{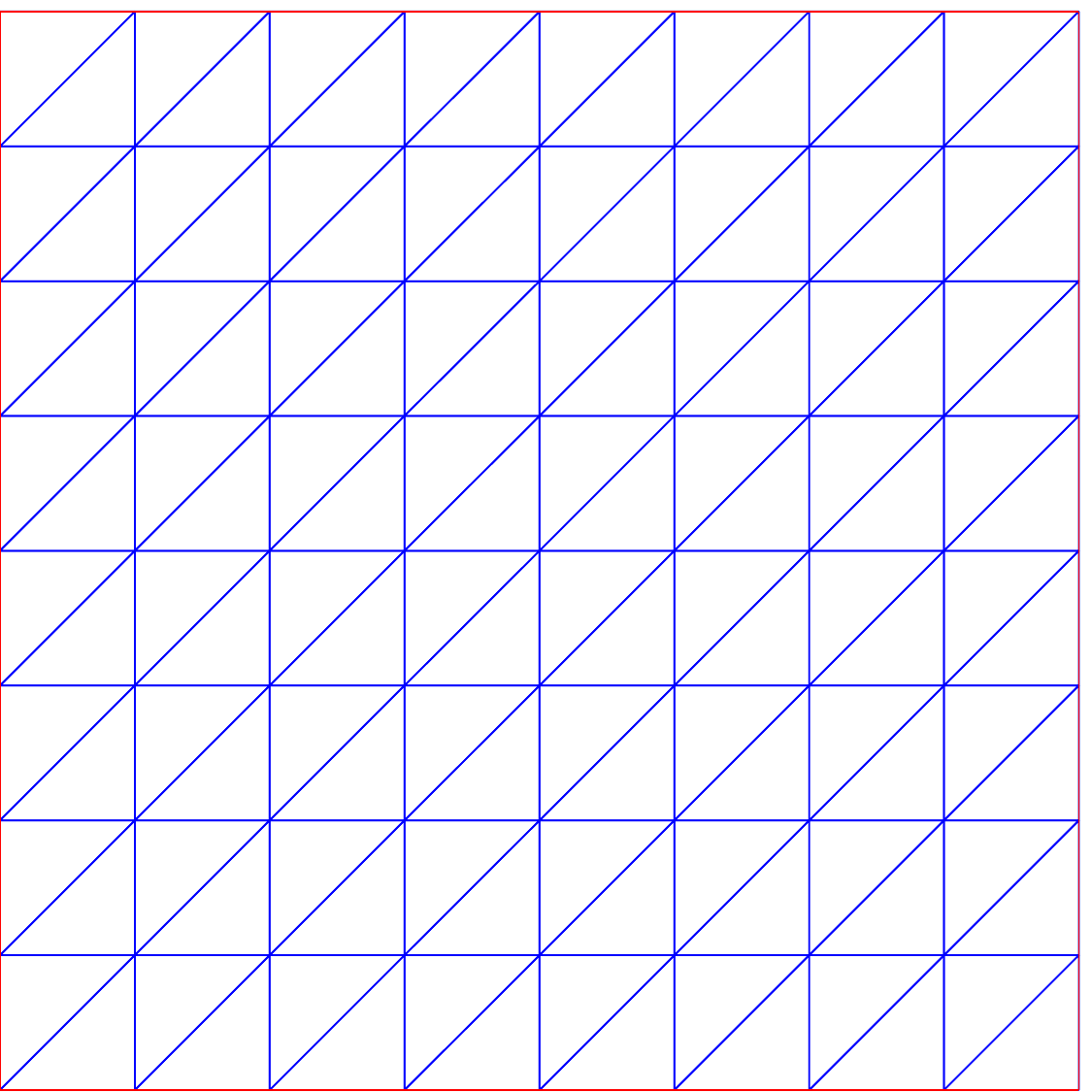}&
 \includegraphics[trim = 0cm 0cm 0cm 0cm, width=.23\textwidth]{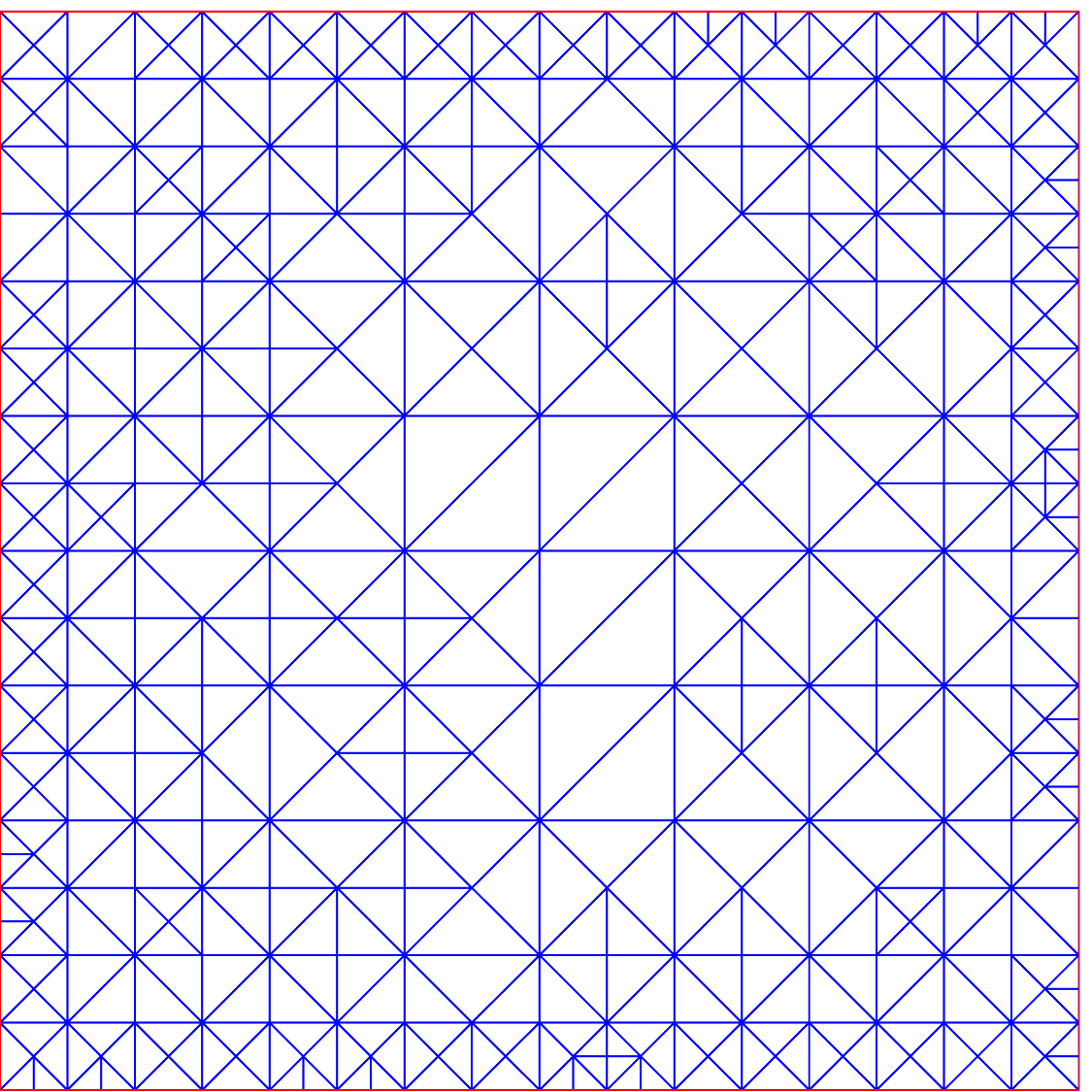}&
 \includegraphics[trim = 0cm 0cm 0cm 0cm, width=.23\textwidth]{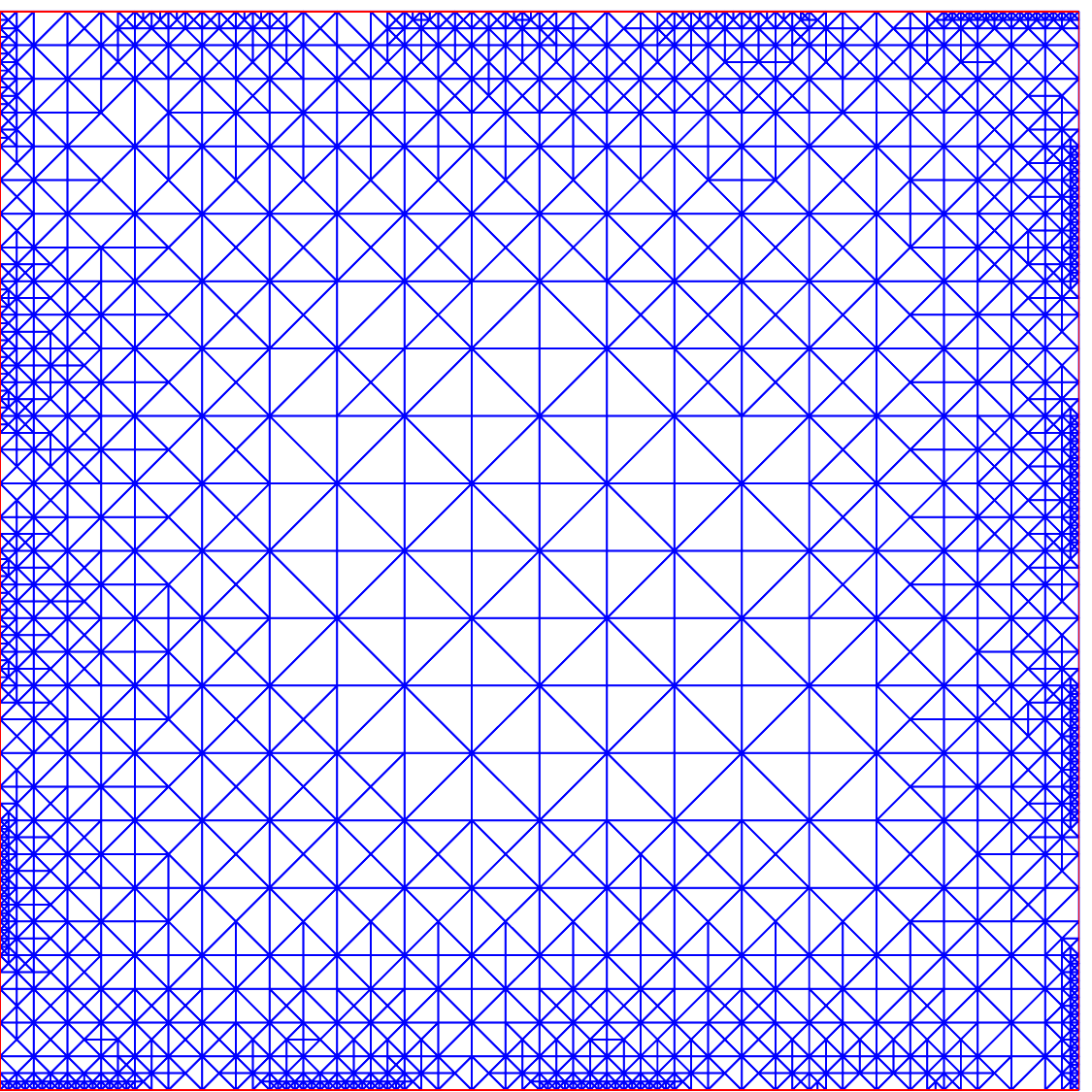}&
 \includegraphics[trim = 0cm 0cm 0cm 0cm, width=.23\textwidth]{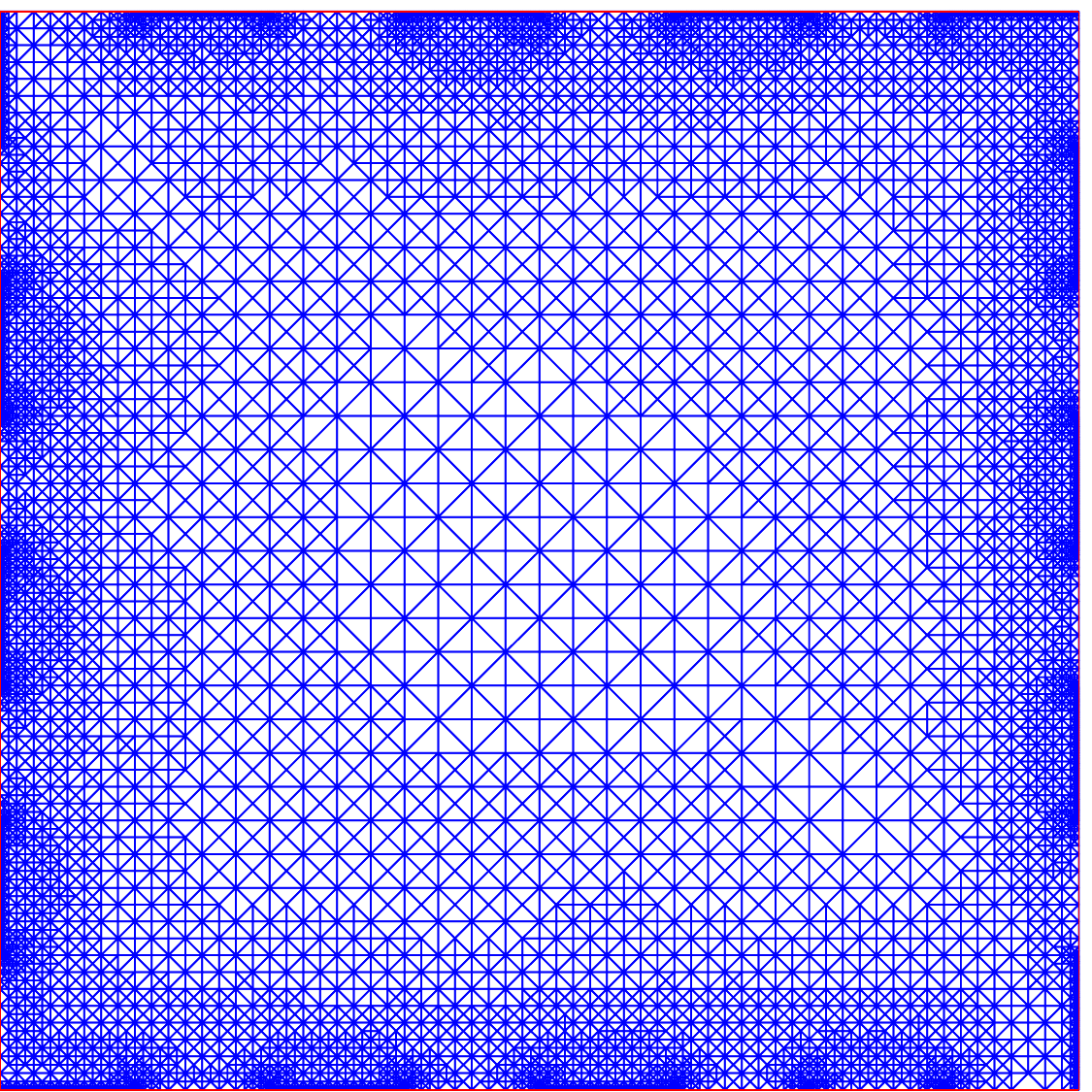}\\

  \includegraphics[trim = 1cm 0.5cm 0.5cm 0cm, width=.24\textwidth]{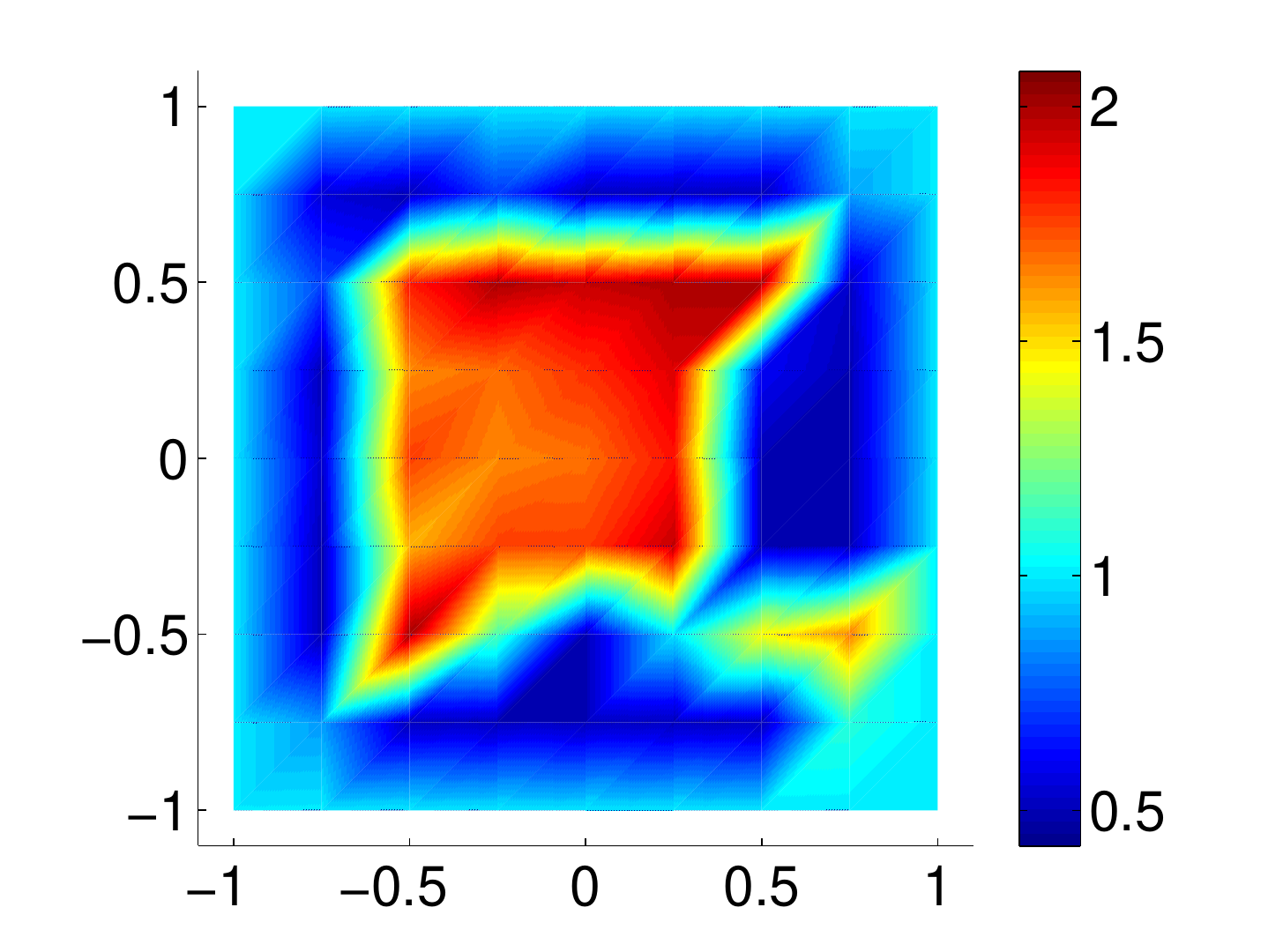} &
  \includegraphics[trim = 1cm 0.5cm 0.5cm 0cm, width=.24\textwidth]{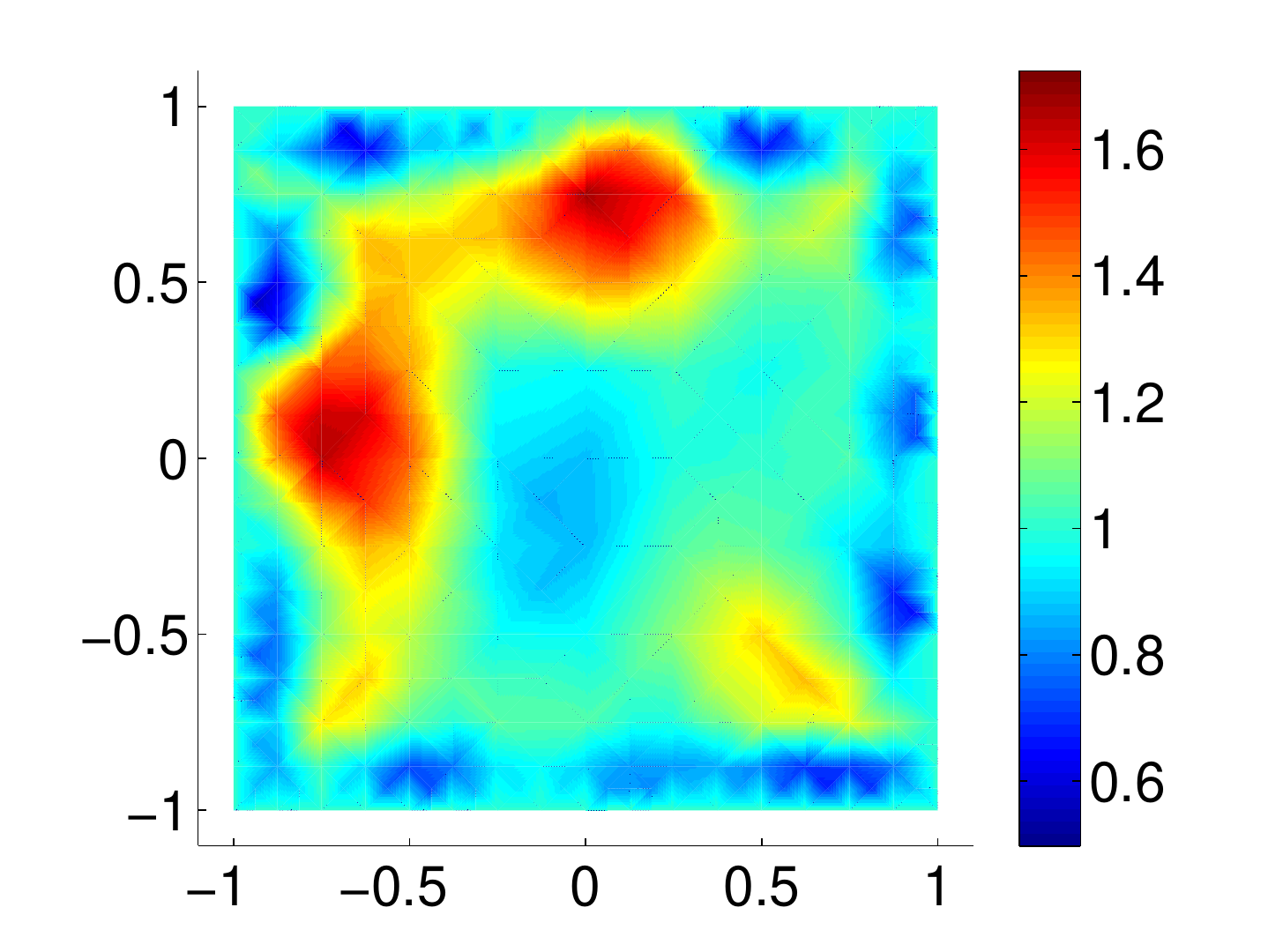} &
  \includegraphics[trim = 1cm 0.5cm 0.5cm 0cm, width=.24\textwidth]{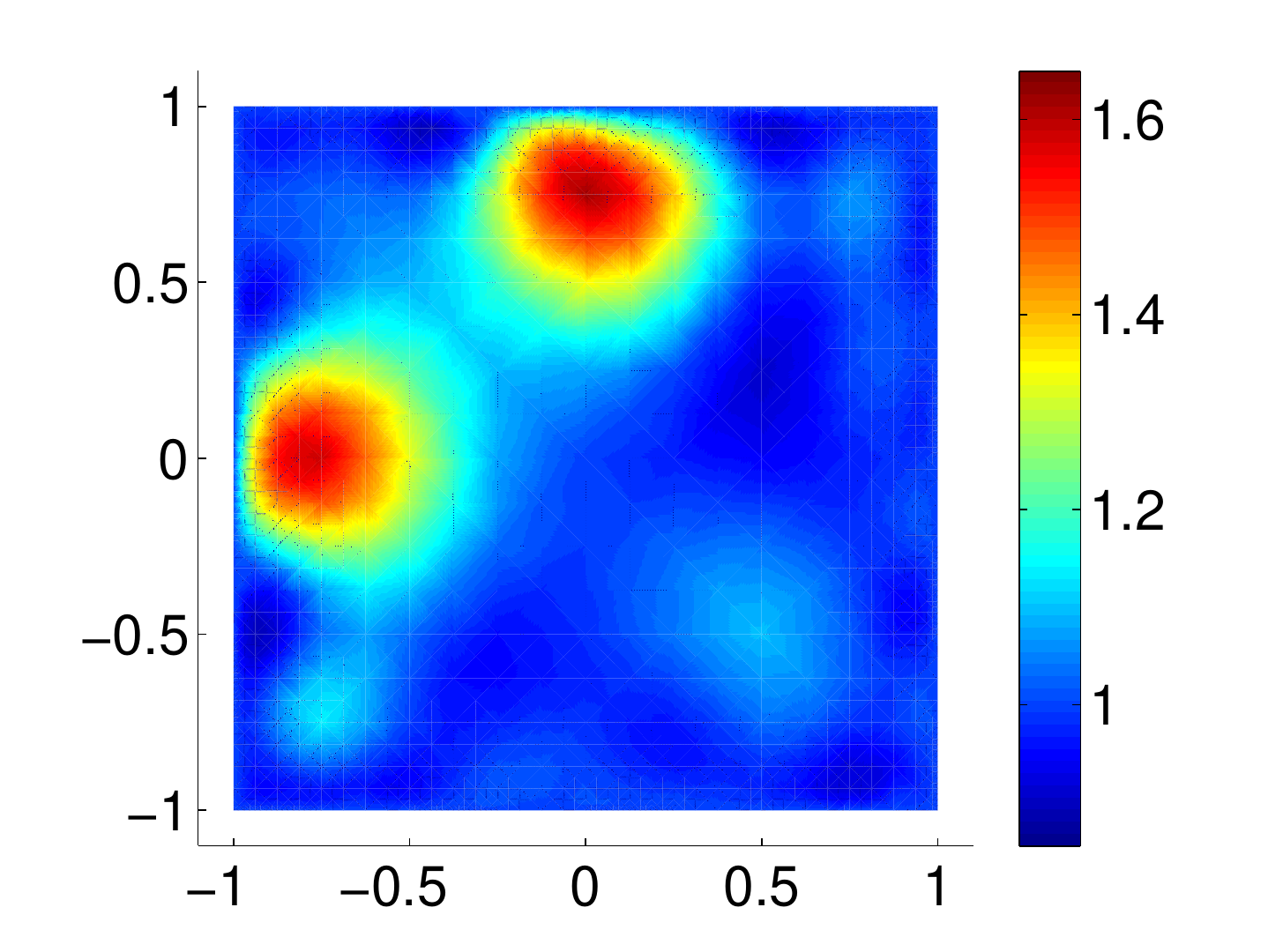} &
  \includegraphics[trim = 1cm 0.5cm 0.5cm 0cm,width=.24\textwidth]{ex2_1e3ad_iter15}\\
    (a) 0th step & (b) 4th step & (c) 9th step & (d) 14th step
 \end{tabular}
 \caption{The recovered conductivity during the adaptive refinement, for Example \ref{exam2} with
 $\epsilon=0.1\%$ noise. The regularization parameter is fixed at $\alpha=2.5\times10^{-4}.$} \label{fig:exam2-recon-iter}
\end{figure}

We plot in Fig. \ref{fig:exam2-recon-iter} the meshes and recoveries at the intermediate refinement steps.
At the initial stage, the refinement mainly occurs in the region around electrode surfaces, where the weak solution
singularity appears. As the refinement proceeds, the region away from the boundary is also refined, but to a much
lesser degree, especially for the central part of the domain. In case of a very coarse initial mesh, the recovery
even fails to correctly identify the number of inclusions, but as the AFEM proceeds, the spurious oscillations disappear, and
then it can identify reasonably the locations and magnitudes of the blobs from the recoveries, cf. Fig.
\ref{fig:exam2-recon-iter}. In Fig. \ref{fig:exam2-efficiency}, we show the $L^2(\Omega)$ and $H^1(\Omega)$ errors of
the recoveries versus the degree of freedom $N$ of the mesh $\cT_k$ for the adaptive and uniform refinement.
These plots fully show the efficiency of the adaptive algorithm, for both $\epsilon=0.1\%$ and $\epsilon=1\%$ noise;
see also Table \ref{tab:convrate} for the empirical convergence rates.

\begin{figure}[hbt!]
  \centering
  \begin{tabular}{cc}
    \includegraphics[trim = 1cm 0cm 1.5cm .5cm, clip=true,width=.35\textwidth]{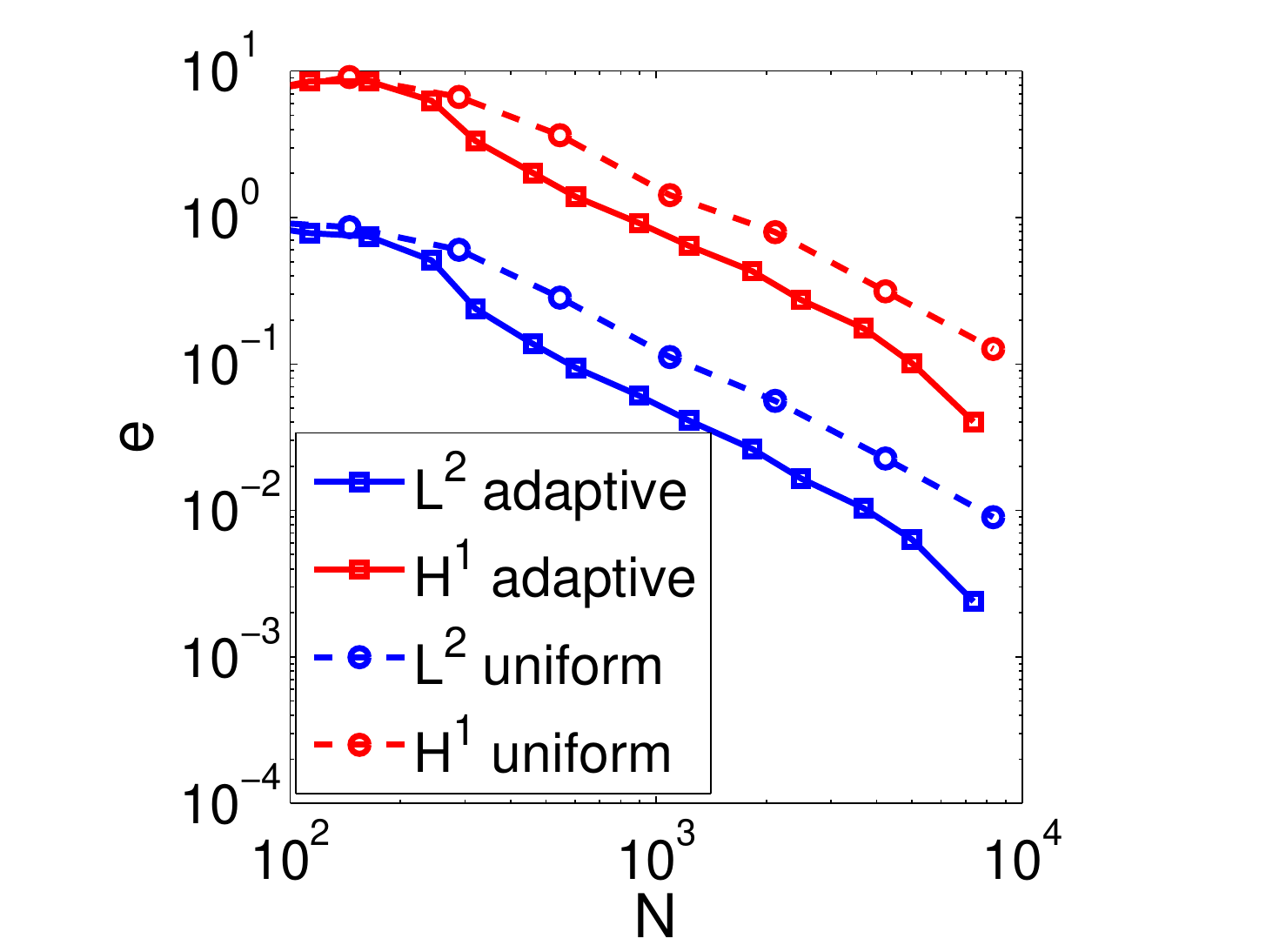}
    & \includegraphics[trim = 1cm 0cm 1.5cm .5cm, clip=true,width=.35\textwidth]{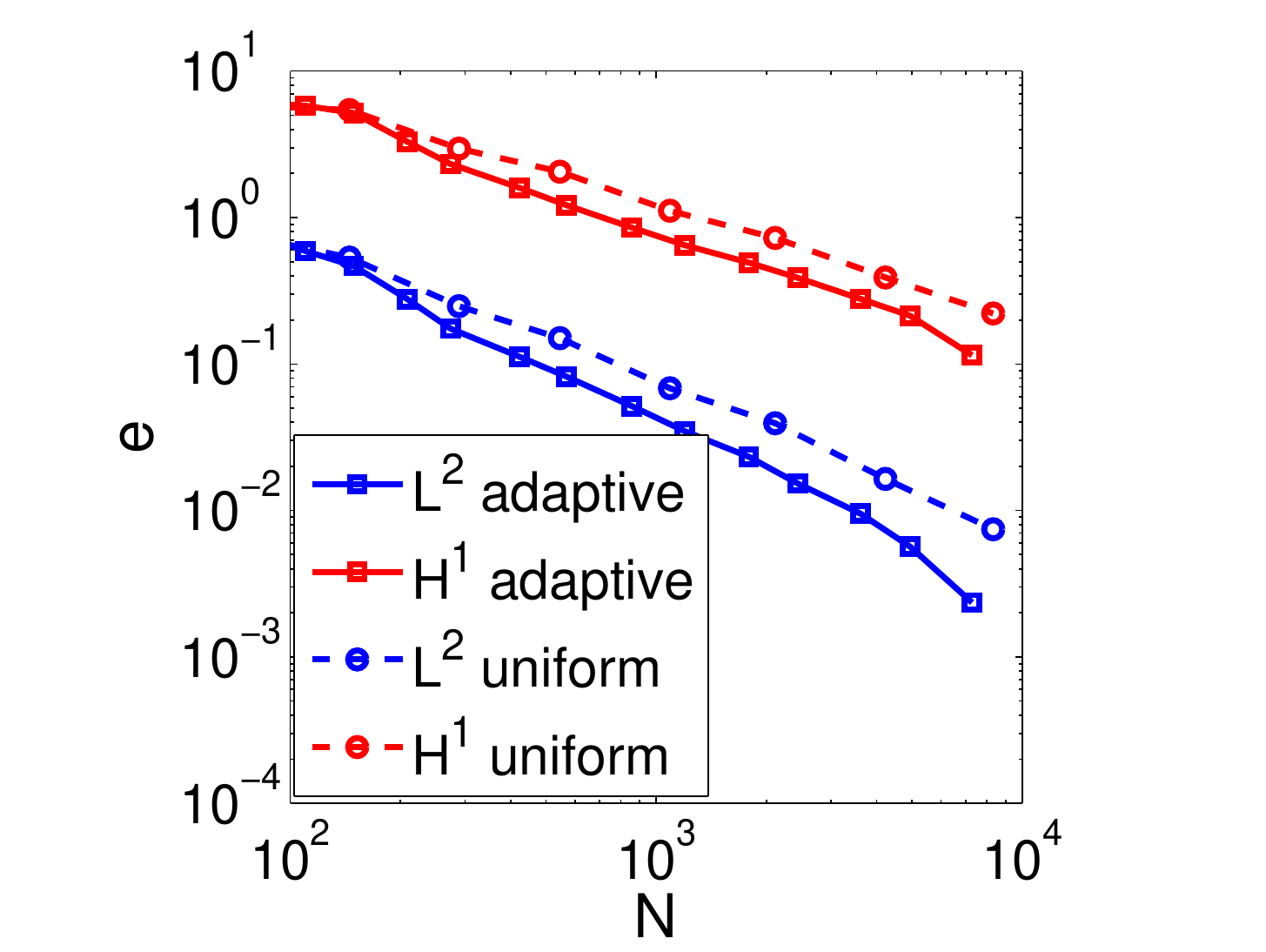}\\
    (a) $\epsilon=1\times10^{-3}$ & (b) $\epsilon=1\times10^{-2}$
  \end{tabular}
  \caption{The $L^2(\Omega)$ and $H^1(\Omega)$ errors versus the degree of freedom $N$ of the mesh, for Example \ref{exam2} at two different noise levels, using the adaptive refinement (solid line) and uniform refinement (dashed line).}\label{fig:exam2-efficiency}
\end{figure}

Last, we consider one example with a discontinuous conductivity field.
\begin{example}\label{exam3}
The true conductivity $\sigma^\dag$ is given by $\sigma^\dag(x) = \sigma_0(x) + (x_1/2+x_2)\chi_{\Omega'}$, where $\chi_{\Omega'}$ is the
characteristic function of the set $\Omega'=(1/4,3/4)\times(0,1/2)$, and the back ground conductivity $\sigma_0(x)=1$.
\end{example}

Since the $H^1(\Omega)$ penalty imposes a global smoothness condition, it is unsuitable for recovering
discontinuous conductivity fields. Hence, in this example we assume that the support $\Omega'$ of the true conductivity
field $\sigma^\dag$ is known, and aim at determining the variation within the support using the $H^1(\Omega')$ semi-norm penalty.
The adaptive algorithm and the convergence proof can be extended directly: the variational inequality
is now defined only on $\Omega'$, and the estimator
$\eta_{\cT,3}^{2}(\sigma_{\cT}^{\ast},u_{\cT}^{\ast},p_{\cT}^{\ast},T)$ is only for elements in $\Omega'$).

The numerical results for the example are presented in Figs. \ref{fig:exam3-recon}, \ref{fig:exam3-recon-iter} and
\ref{fig:exam3-efficiency}. The observations from the preceding two examples remain largely valid. The magnitude
of the conductivity is slightly reduced, but otherwise the profile is reasonable, and visually the recoveries
by the adaptive and the uniform refinements are close to each other, cf. Figs. \ref{fig:exam3-recon}(b) and
\ref{fig:exam3-recon}(c). Even though the conductivity field $\sigma$ is discontinuous, the adaptive algorithm
first mainly resolves the singularity due to the change of boundary conditions, i.e., around the boundary,
cf. Fig. \ref{fig:exam3-recon-iter}(b). As the adaptive iteration proceeds, the algorithm then starts to refine
the region near the boundary $\partial\Omega^\prime$ of the subdomain $\Omega'$: first the part close
to the boundary $\partial\Omega$, and then the part away from $\partial\Omega$, cf. Figs. \ref{fig:exam3-recon-iter}(c)
and \ref{fig:exam3-recon-iter}(d). This is consistent with the empirical observation that the further away from the boundary,
the more challenging it is to be resolved (from the boundary data), i.e., the boundary data allows better resolving
the regions close to the boundary. Hence, the solution singularity
induced by the conductivity discontinuity does not play an important role in the inversion as it was in direct
problems. The gain of computational efficiency is shown in Fig. \ref{fig:exam3-efficiency}: the $L^2(\Omega)$-
and the $H^1(\Omega)$-errors decrease faster with
the increase of degree of freedom for the adaptive algorithm than that for the uniform refinement.

\begin{figure}[hbt!]
  \centering\setlength{\tabcolsep}{0em}
  \begin{tabular}{ccc}
      \includegraphics[trim = .5cm 0cm 1cm 0cm,clip=true,width=0.33\textwidth]{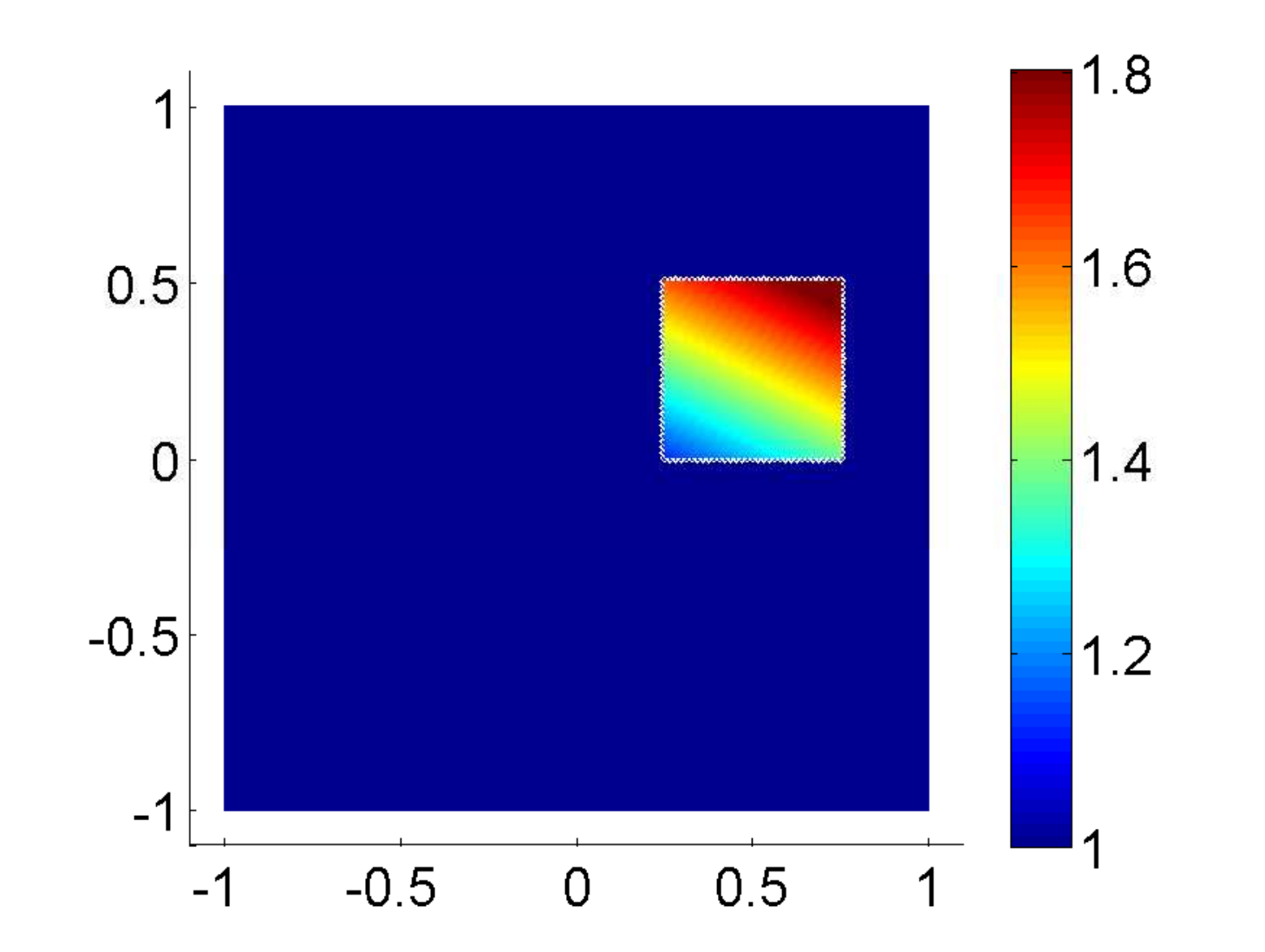}
    & \includegraphics[trim = .5cm 0cm 1cm 0cm, clip=true,width=.33\textwidth]{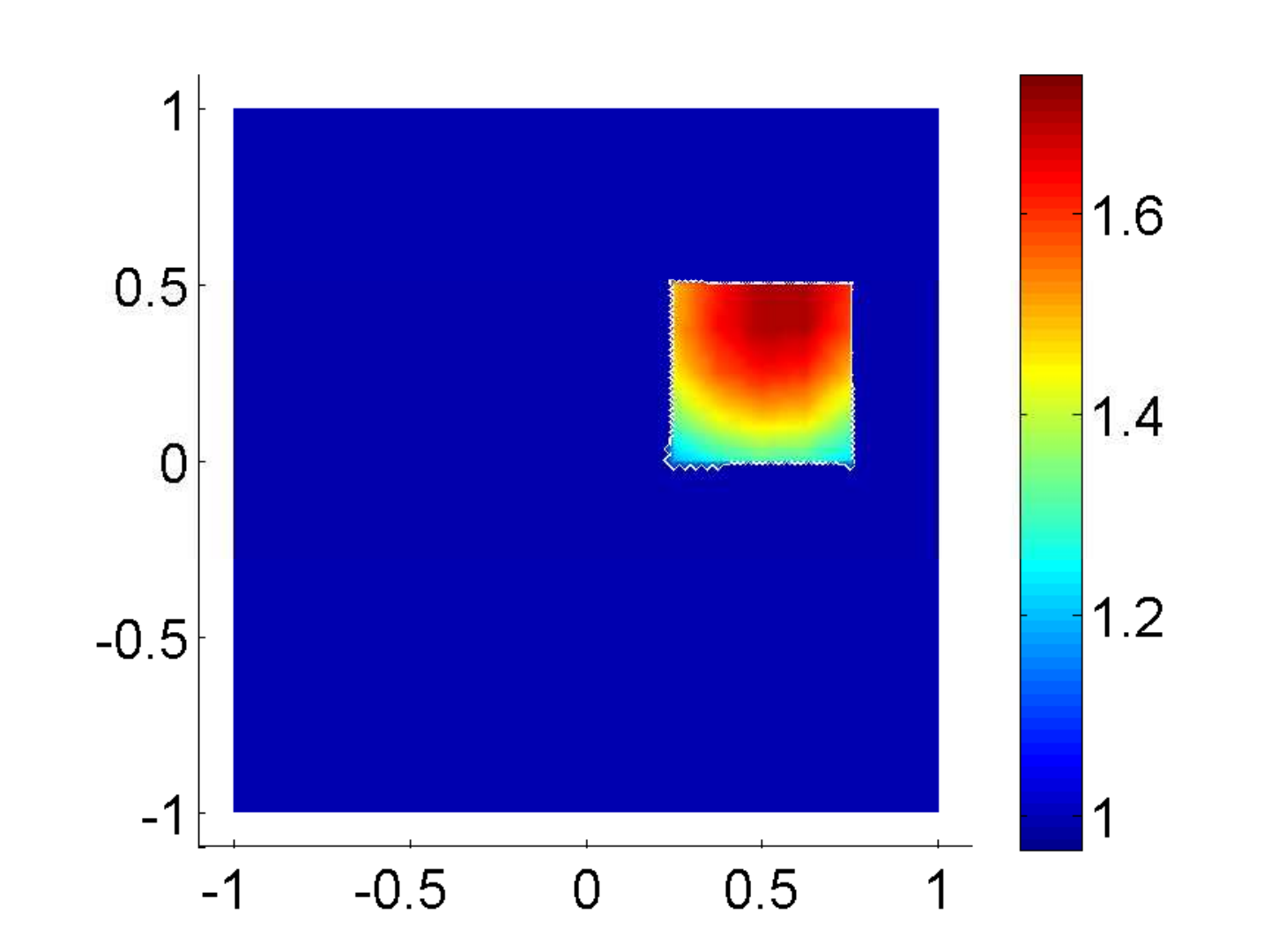}
    & \includegraphics[trim = .5cm 0cm 1cm 0cm, clip=true,width=.33\textwidth]{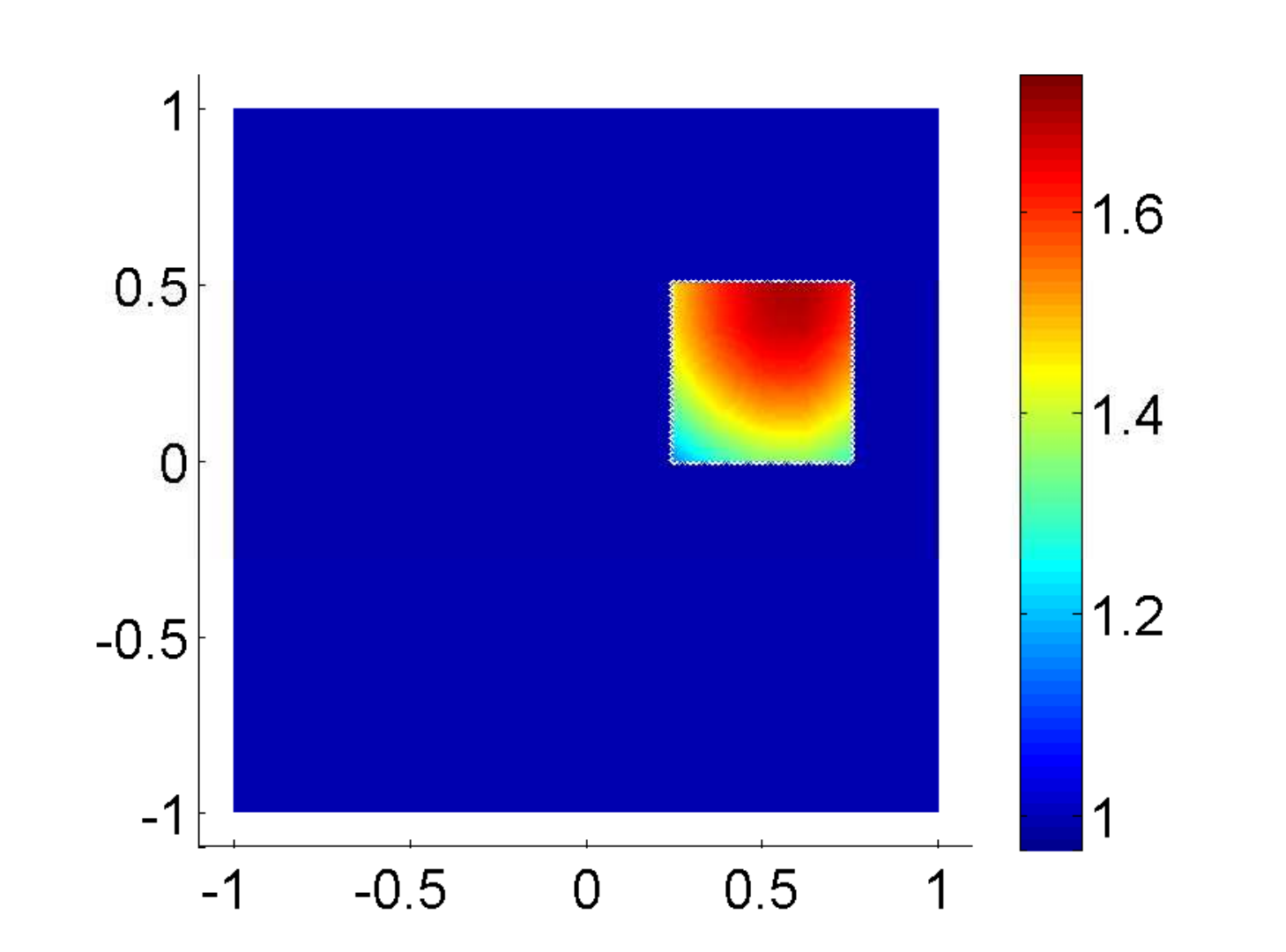}\\
    (a) true conductivity & (b) adaptive refinement & (c) uniform refinement
  \end{tabular}
  \caption{The final reveries by the adaptive and uniform refinements for Example
  \ref{exam3} with $\epsilon=0.1\%$ noise. The degree of freedom is $19608$ and $33025$
  for the adaptive and uniform refinement, respectively. The regularization parameter $\alpha$
  is fixed at $\alpha=3.2\times10^{-3}$.}\label{fig:exam3-recon}
\end{figure}

\begin{figure}[hbt!]
 \centering
 \setlength{\tabcolsep}{1pt}
 \begin{tabular}{ccccc}
 \includegraphics[trim = 0cm 0cm 0cm 0cm, width=.23\textwidth]{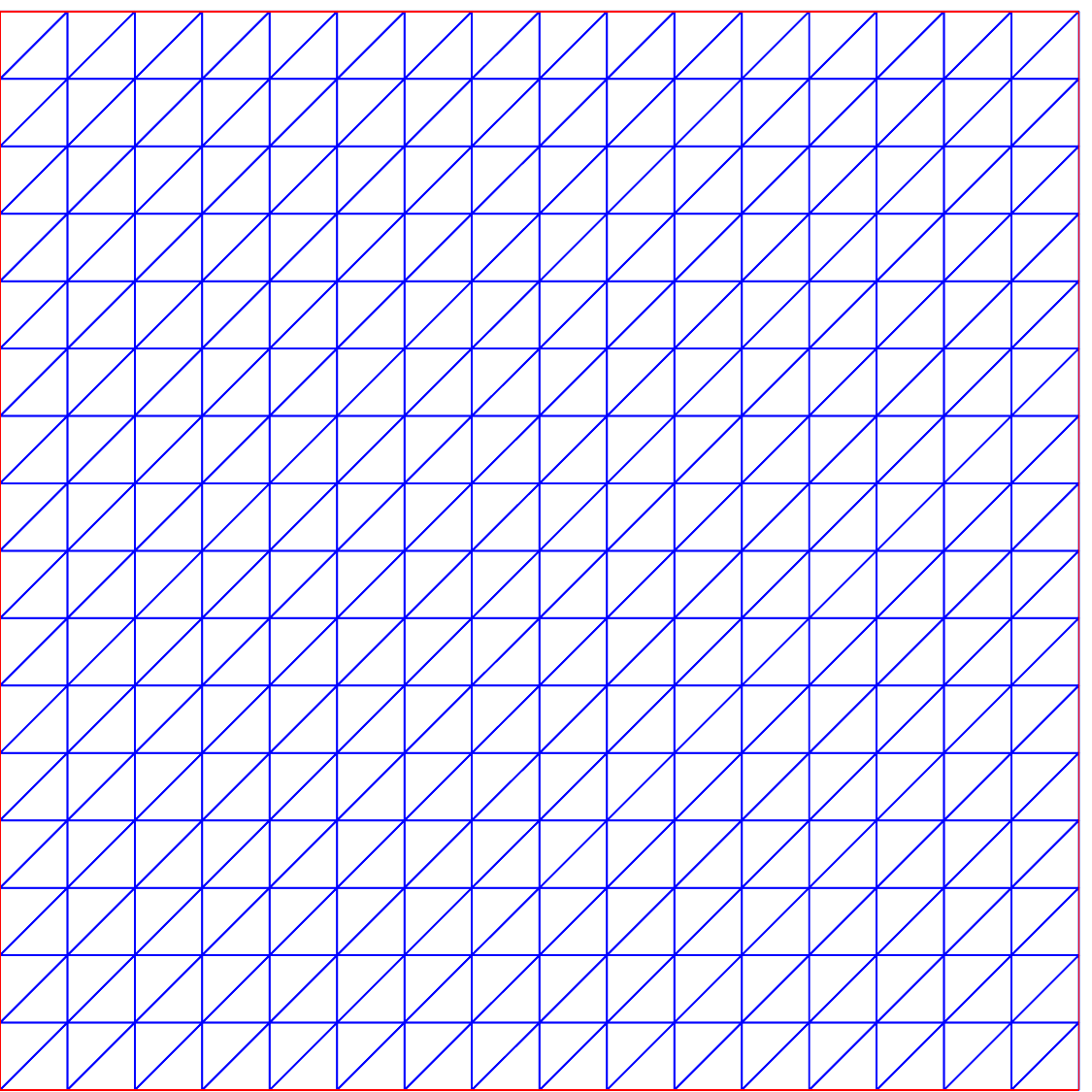}&
 \includegraphics[trim = 0cm 0cm 0cm 0cm, width=.23\textwidth]{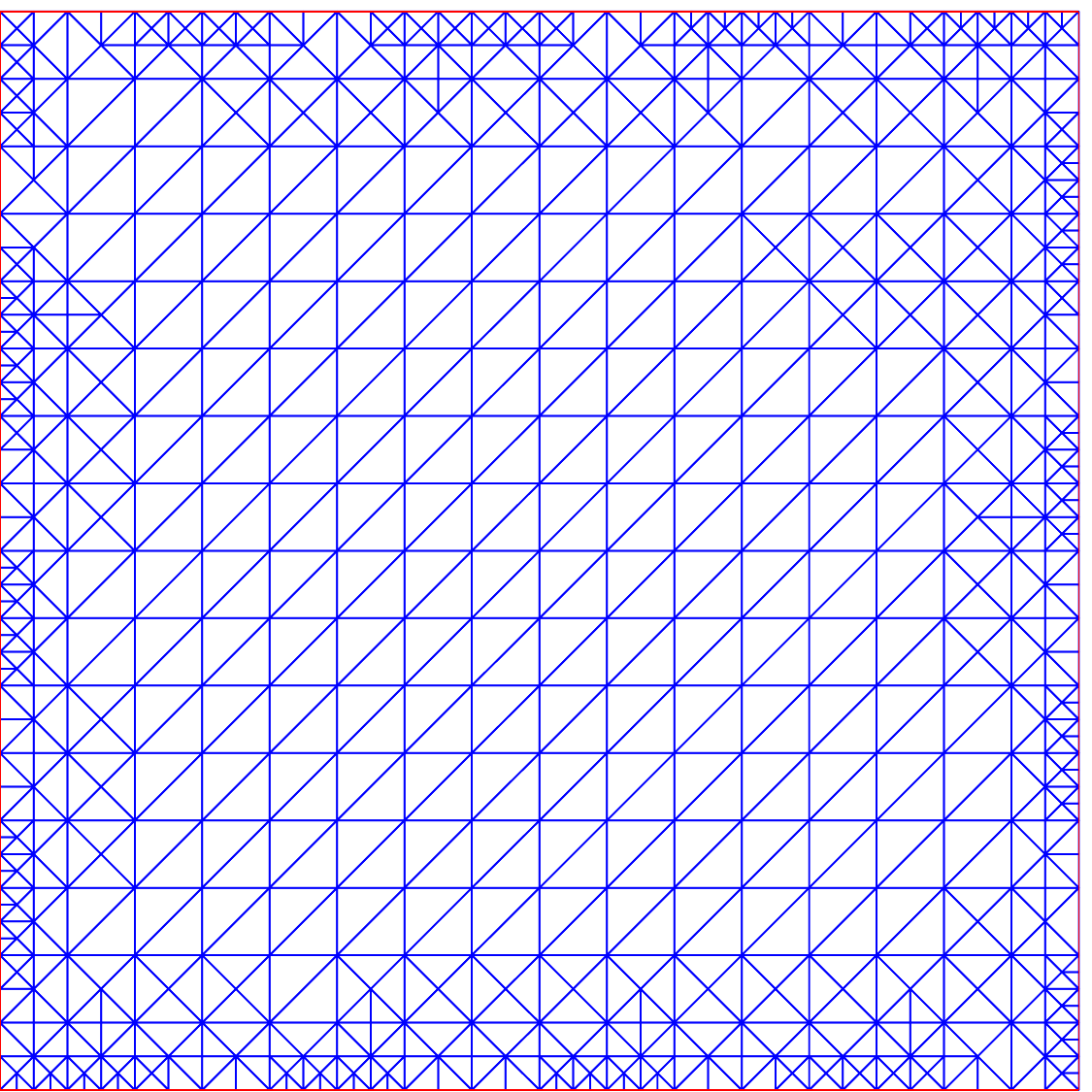}&
 \includegraphics[trim = 0cm 0cm 0cm 0cm, width=.23\textwidth]{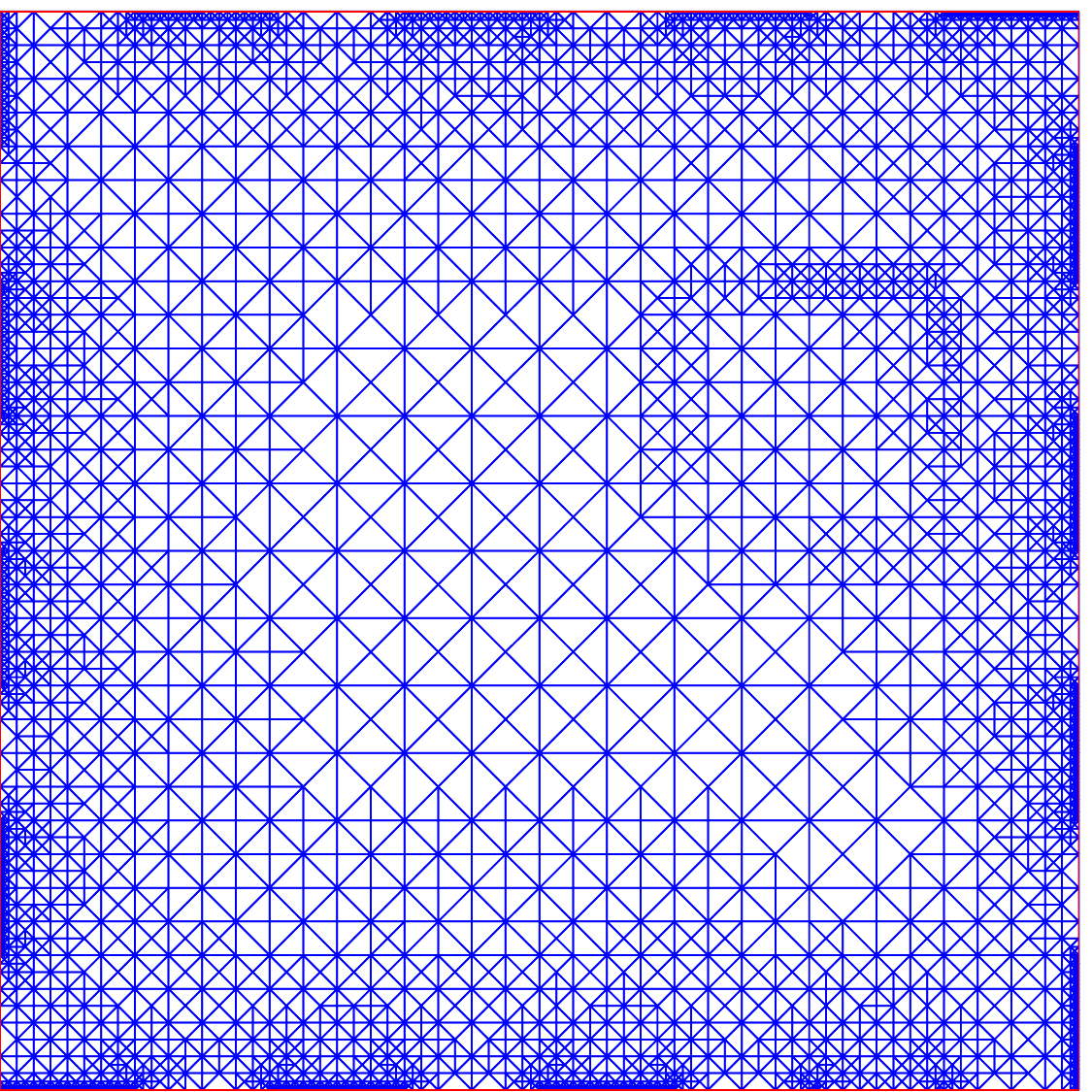}&
 \includegraphics[trim = 0cm 0cm 0cm 0cm, width=.23\textwidth]{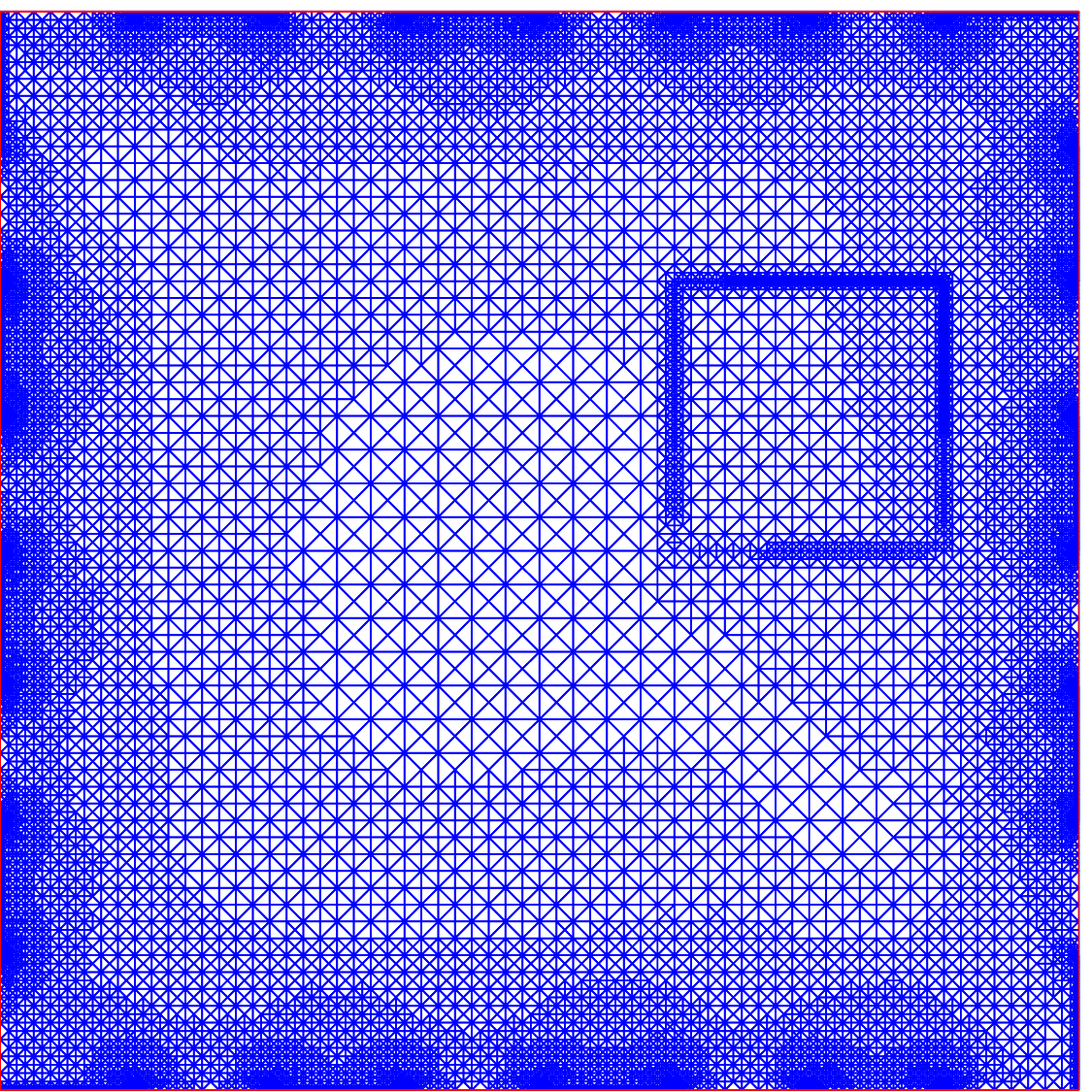}\\
    \includegraphics[trim = 1cm 0.5cm 0.2cm 0cm, width=.24\textwidth]{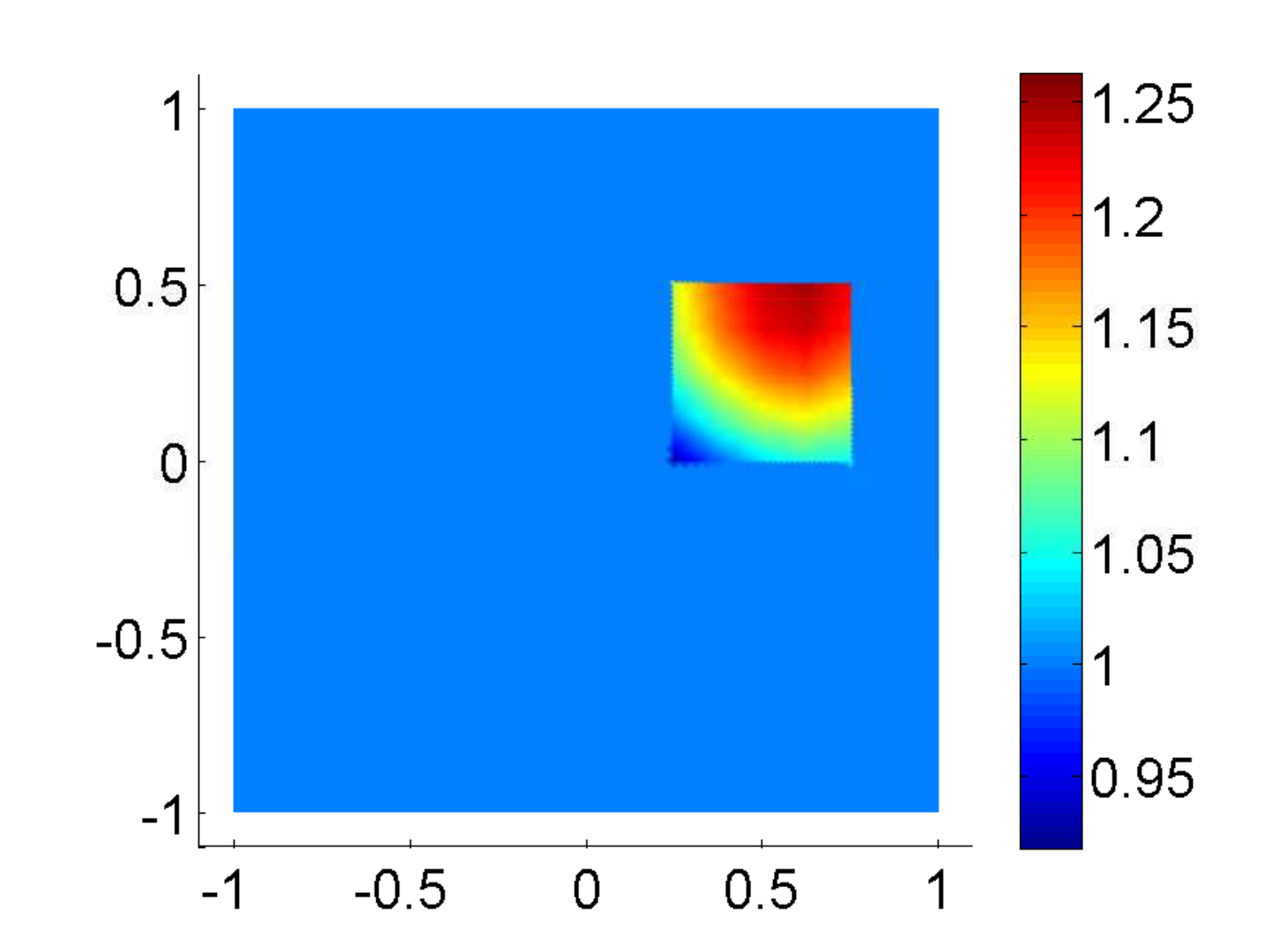} &
    \includegraphics[trim = 1cm 0.5cm 0.2cm 0cm, width=.24\textwidth]{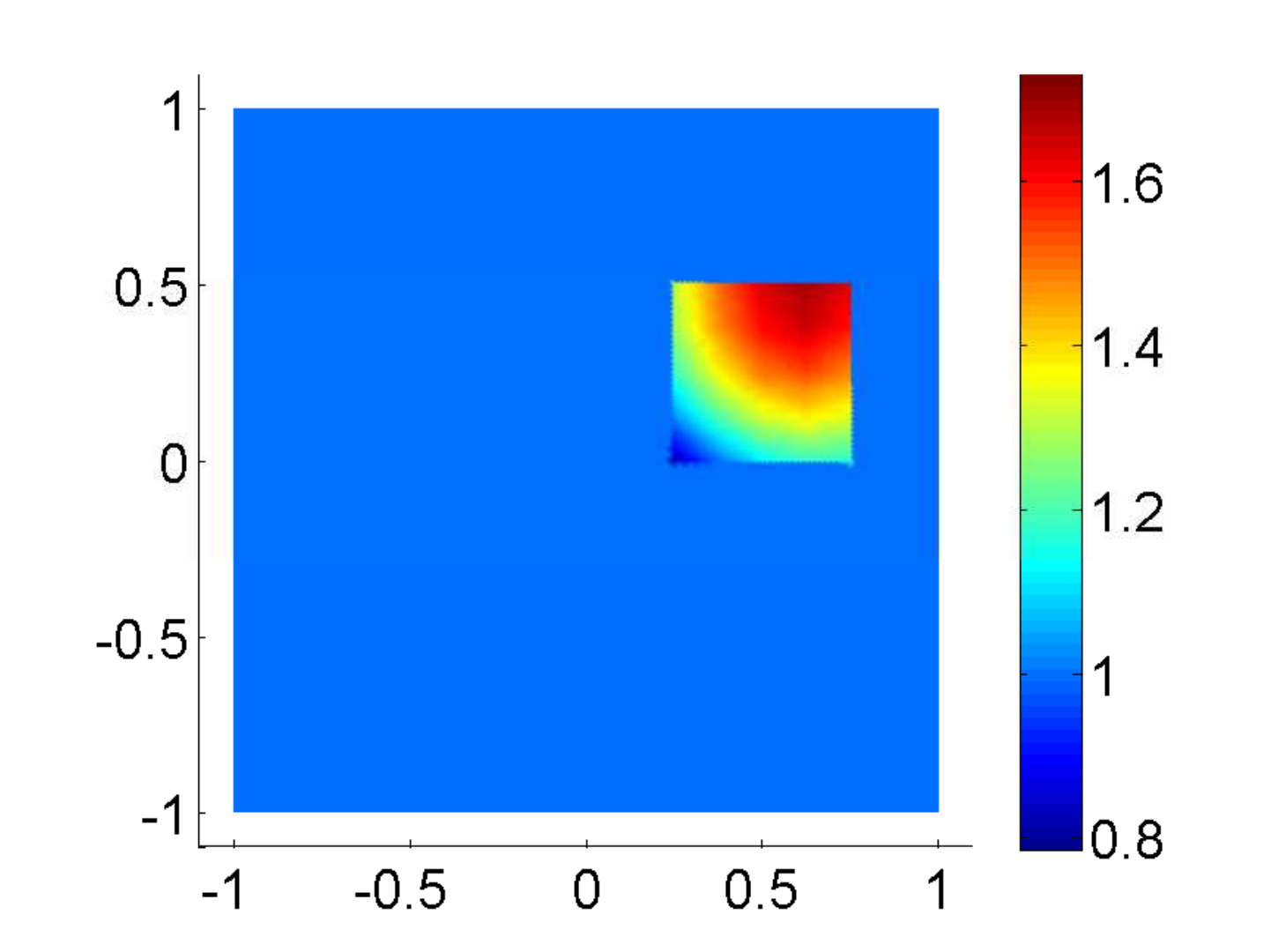} &
    \includegraphics[trim = 1cm 0.5cm 0.2cm 0cm, width=.24\textwidth]{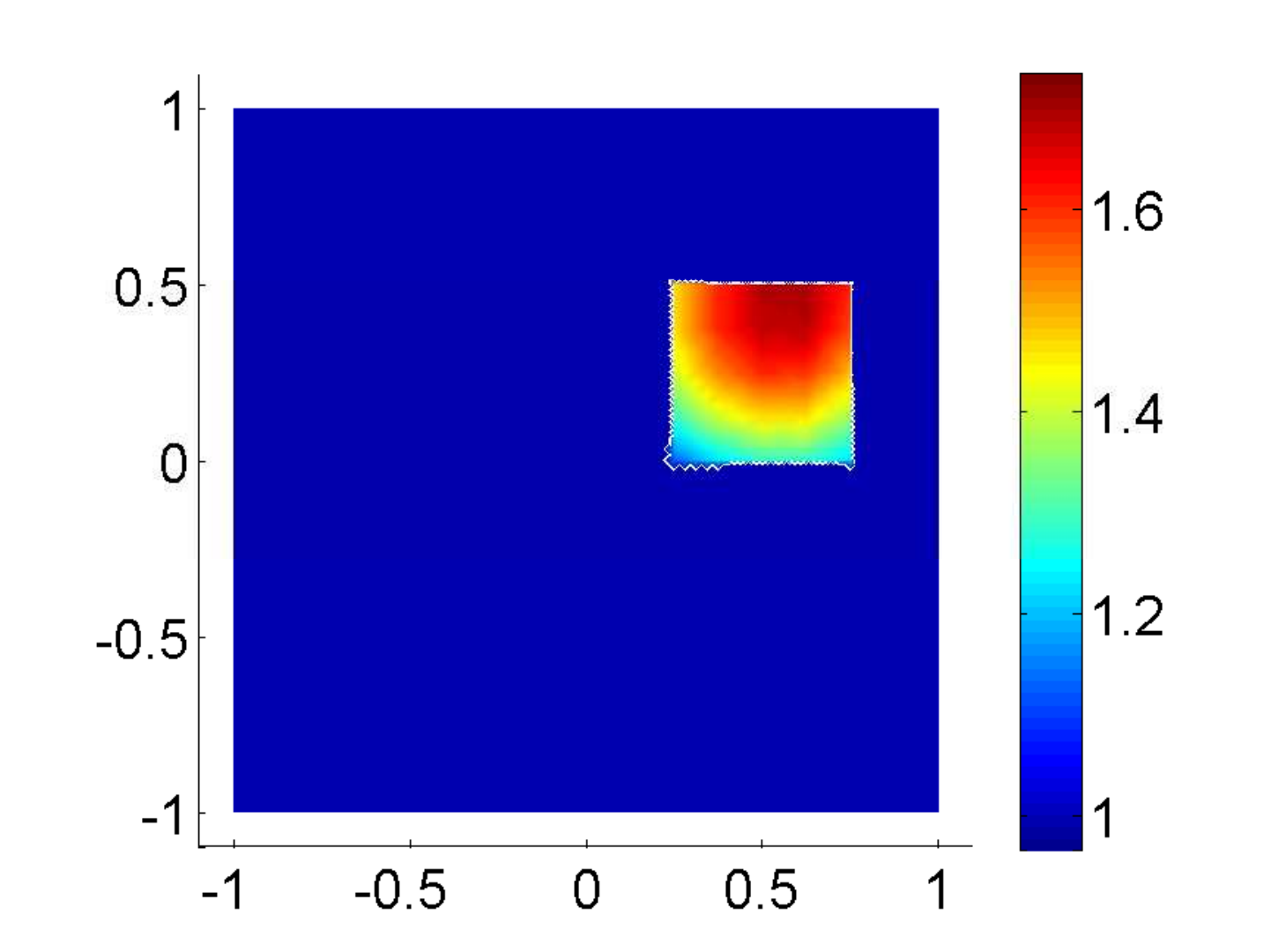} &
    \includegraphics[trim = 1cm 0.5cm 0.2cm 0cm,width=.24\textwidth]{ex4_1e3ad_iter15}\\
    (a) 0th step & (b) 4th step & (c) 9th step & (d) 14th step
 \end{tabular}
 \caption{The recovered conductivity during the adaptive refinement, for Example \ref{exam3} with $\epsilon=0.1\%$ noise.
 The regularization parameter is fixed at $\alpha=3.2\times10^{-3}.$} \label{fig:exam3-recon-iter}
\end{figure}

\begin{figure}[hbt!]
  \centering
  \begin{tabular}{cc}
    \includegraphics[trim = 1cm 0cm 1.5cm 0cm, clip=true,width=.35\textwidth]{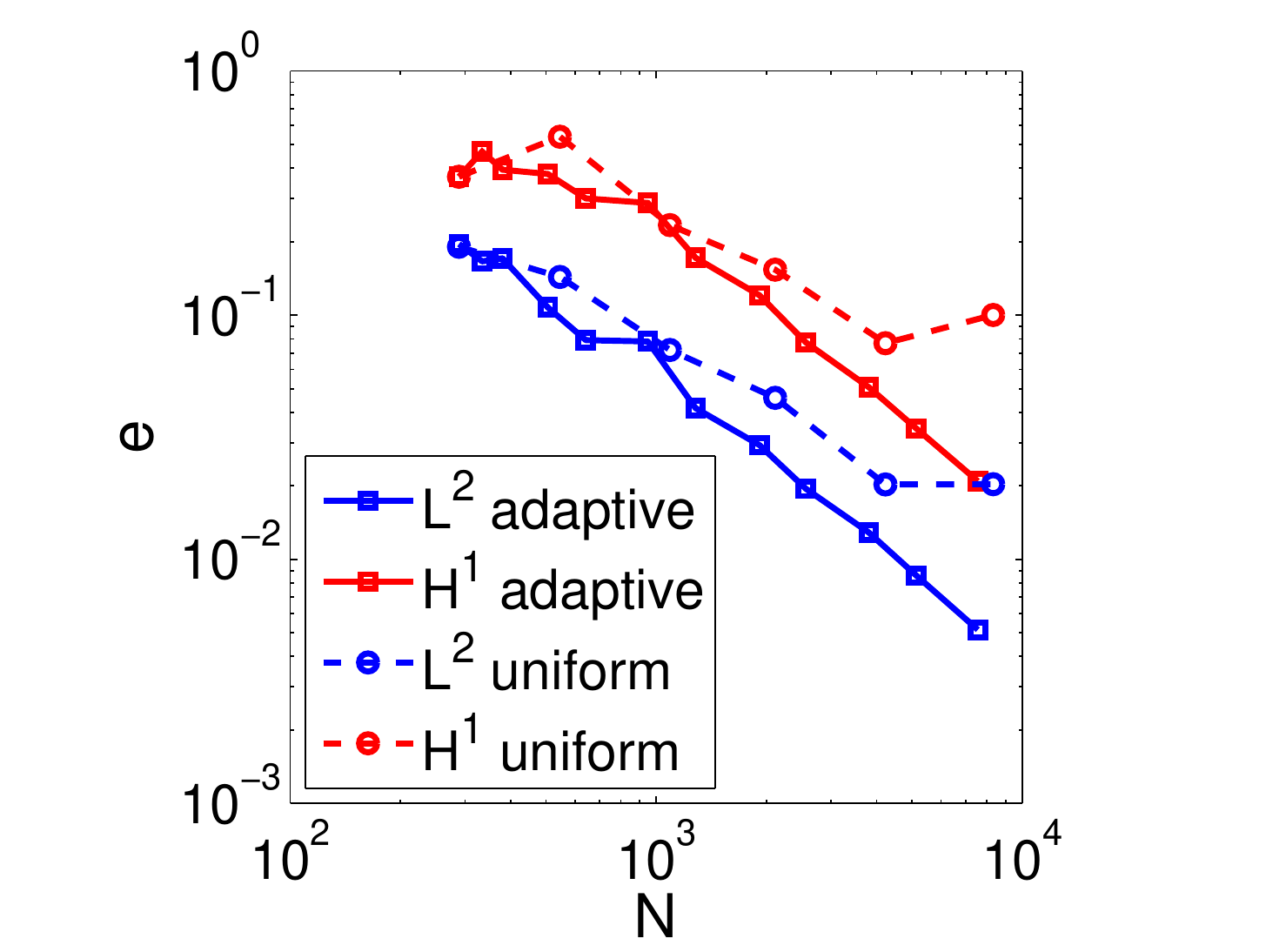}
    & \includegraphics[trim = 1cm 0cm 1.5cm 0cm, clip=true,width=.35\textwidth]{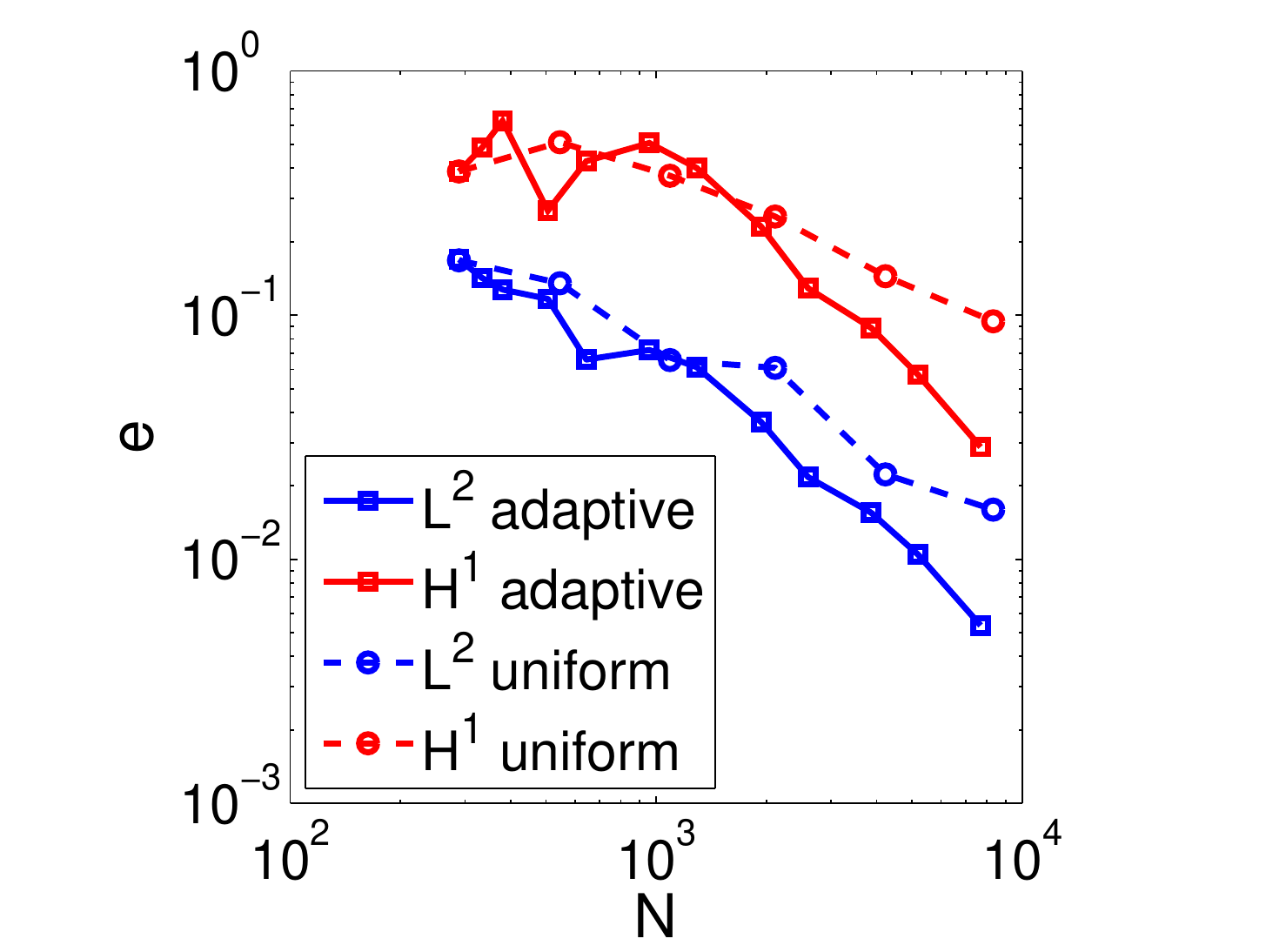}\\
    (a) $\epsilon=1\times10^{-3}$ & (b) $\epsilon=1\times10^{-2}$
  \end{tabular}
  \caption{The $L^2(\Omega^\prime)$ and $H^1(\Omega^\prime)$ errors versus the degree of freedom $N$ of the mesh, for Example \ref{exam3} at two different noise levels, using the adaptive refinement (solid line) and uniform refinement (dashed line).}\label{fig:exam3-efficiency}
\end{figure}

\section{Concluding remarks}
In this work, we have developed a novel adaptive finite element method for the electrical impedance tomography
inverse problem, modeled by the complete
electrode model. It is formulated as an output least-squares problem with a Sobolev smoothness
penalty. The weak solution singularity around the electrode surfaces and low-regularity conductivity motivate
the use of the adaptive refinement techniques. We have derived a residual-type a posteriori error estimator,
which involves the state, adjoint and conductivity estimate, and established the
convergence of the sequence of solutions generated by the adaptive technique that the accumulation point solves
the continuous optimality system. The
efficiency and convergence of the proposed adaptive algorithm is confirmed by a few numerical experiments.

This work represents only a first step towards the rigorous adaptive finite element method for
nonlinear inverse problems associated with PDEs. There are several research
problems deserving further study. First, the proposed
algorithm is only for the smoothness penalty, which is essential in the development and convergence analysis of the algorithm.
It is of much interest to derive and to analyze adaptive algorithms for nonsmooth penalties, e.g., total
variation and sparsity. Second, numerically one
observes that the algorithm can approximate a (local/global) minimizer of the continuous optimization well,
instead of only a solution to the necessary optimality condition. This is
still theoretically to be justified. Third, the reliability and optimality of the adaptive algorithm for
nonlinear inverse problems
are completely open, which seems not fully understood even for linear ones. The optimality
issue in the context of inverse problems should be related to the noise level. The crucial
interplay between the error estimator and noise level is to be elucidated.

\section*{Acknowledgements}

The authors are grateful to the referees for their constructive comments, which have led
to an improved presentation, and to Mr. Chun-Man Yuen for his great help in carrying out the numerical experiments.
The work of B. Jin was supported by UK EPSRC  grant EP/M025160/1, and that
of Y. Xu by National Natural Science Foundation of China (11201307), Ministry of Education of China through
Specialized Research Fund for the Doctoral Program of Higher Education (20123127120001),
E-Institute of Shanghai Universities (E03004) and Innovation Program of Shanghai Municipal Education Commission (13YZ059).
The work of J. Zou was substantially supported by Hong Kong RGC grants
(projects 14306814 and 405513).

\bibliographystyle{abbrv}
\bibliography{eit}

\begin{thebibliography}{10}

\bibitem{AdGaLi:2011}
A.~Adler, R.~Gaburro, and W.~Lionheart.
\newblock Electrical impedance tomography.
\newblock In O.~Scherzer, editor, {\em Handbook of Mathematical Methods in
  Imaging}. Springer-Verlag, 2011.

\bibitem{AinsworthOden:2000}
M.~Ainsworth and J.~T. Oden.
\newblock {\em {A} {P}osteriori {E}rror {E}stimation in {F}inite {E}lement
  {A}nalysis}.
\newblock Wiley-Interscience, New York, 2000.

\bibitem{BeckerMao:2011}
R.~Becker and S.~Mao.
\newblock Quasi-optimality of an adaptive finite element method for an optimal
  control problem.
\newblock {\em Comput. Methods Appl. Math.}, 11(2):107--128, 2011.

\bibitem{BeckerVexler:2004}
R.~Becker and B.~Vexler.
\newblock {A} posteriori error estimation for finite element discretization of
  parameter identification problems.
\newblock {\em Numer. Math.}, 96:435--459, 2004.

\bibitem{BeilinaClason:2006}
L.~Beilina and C.~Clason.
\newblock An adaptive hybrid {FEM}/{FDM} method for an inverse scattering
  problem in scanning acoustic microscopy.
\newblock {\em SIAM J. Sci. Comput.}, 28(1):382--402, 2006.

\bibitem{BeilinaJohnson:2005}
L.~Beilina and C.~Johnson.
\newblock {A} posteriori error estimation in computational inverse scattering.
\newblock {\em Math. Models Methods Appl. Sci.}, 15(1):23--35, 2005.

\bibitem{BeilinaKlibanov:2010a}
L.~Beilina and M.~V. Klibanov.
\newblock {A} posteriori error estimates for the adaptivity technique for the
  tikhonov functional and global convergence for a coefficient inverse problem.
\newblock {\em Inverse Problems}, 26(4):045012, 27pp, 2010.

\bibitem{BeilinaKlibanov:2010b}
L.~Beilina and M.~V. Klibanov.
\newblock {R}econstruction of dielectrics from experimental data via a hybrid
  globally convergent/adaptive algorithm.
\newblock {\em Inverse Problems}, 26(12):125009, 30 pp, 2010.

\bibitem{BeilinaKlibanovKokurin:2010}
L.~Beilina, M.~V. Klibanov, and M.~Y. Kokurin.
\newblock {A}daptivity with relaxation for ill-posed problems and global
  convergence for a coefficient inverse problem.
\newblock {\em J. Math. Sci.}, 167(3):279--325, 2010.

\bibitem{CFPP:2014}
C.~Carstensen, M.~Feischl, M.~Page, and D.~Praetorius.
\newblock {A}xioms of adaptivity.
\newblock {\em Comput. Math. Appl.}, 67(6):1195--1253, 2014.

\bibitem{ChengIsaacsonNewellGisser:1989}
K.-S. Cheng, D.~Isaacson, J.~C. Newell, and D.~G. Gisser.
\newblock {E}lectrode models for electric current computed tomography.
\newblock {\em IEEE Trans. Biomed. Eng.}, 36(9):918--924, 1989.

\bibitem{ChowItoZou:2014}
Y.~T. Chow, K.~Ito, and J.~Zou.
\newblock A direct sampling method for electrical impedance tomography.
\newblock {\em Inverse Problems}, 30(9):095003, 25 pp., 2014.

\bibitem{Ciarlet:2002}
P.~G. Ciarlet.
\newblock {\em {T}he {F}inite {E}lement {M}ethod for {E}lliptic {P}roblems}.
\newblock SIAM, Philadelphia, PA, 2002.

\bibitem{DunlopStuart:2015}
M.~M. Dunlop and A.~M. Stuart.
\newblock The {B}ayesian formulation of {EIT}: analysis and algorithms.
\newblock The {B}ayesian formulation of {EIT}: analysis and algorithms.
  preprint, arXiv:1508.04106, 2015.

\bibitem{Evans:1992a}
L.~C. Evans and R.~F. Gariepy.
\newblock {\em {M}easure {T}heory and {F}ine {P}roperties of {F}unctions}.
\newblock CRC Press, Boca Raton, 1992.

\bibitem{FengYanLiu:2008}
T.~Feng, Y.~Yan, and W.~Liu.
\newblock {A}daptive finite element methods for the identification of
  distributed parameters in elliptic equation.
\newblock {\em Adv. Comput. Math.}, 29(1):27--53, 2008.

\bibitem{GehreJin:2014}
M.~Gehre and B.~Jin.
\newblock Expectation propagation for nonlinear inverse problems with an
  application to electrical impedance tomography.
\newblock {\em J. Comput. Phys.}, 259:513--535, 2014.

\bibitem{GehreJinLu:2014}
M.~Gehre, B.~Jin, and X.~Lu.
\newblock An analysis of finite element approximation of electrical impedance
  tomography.
\newblock {\em Inverse Problems}, 30(4):045013, 24 pp., 2014.

\bibitem{GriesbaumKaltenbacherVexler:2008}
A.~Griesbaum, B.~Kaltenbacher, and B.~Vexler.
\newblock Efficient computation of the {T}ikhonov regularization parameter by
  goal-oriented adaptive discretization.
\newblock {\em Inverse Problems}, 24(2):025025, 20, 2008.

\bibitem{Grisvard:1985}
P.~Grisvard.
\newblock {\em Elliptic {P}roblems in {N}onsmooth {D}omains}.
\newblock Pitman, Boston, MA, 1985.

\bibitem{HarrachUllrich:2013}
B.~Harrach and M.~Ullrich.
\newblock Monotonicity-based shape reconstruction in electrical impedance
  tomography.
\newblock {\em SIAM J. Math. Anal.}, 45(6):3382--3403, 2013.

\bibitem{HinterHoppe:2010}
M.~Hinterm\"{u}ller and R.~H.~W. Hoppe.
\newblock {G}oal-oriented adaptivity in pointwise state constrained optimal
  control of partial differential equations.
\newblock {\em SIAM J. Control Optim.}, 48(8):5468--5487, 2010.

\bibitem{HinterHoppeIliashKieweg:2008}
M.~Hinterm\"{u}ller, R.~H.~W. Hoppe, Y.~Iliash, and M.~Kieweg.
\newblock {A}n a posteriori error analysis of adaptive finite element methods
  for distributed elliptic control problems with control constraints.
\newblock {\em ESAIM, Control Optim. Calc. Var.}, 14(3):540--560, 2008.

\bibitem{ItoJin:2014}
K.~Ito and B.~Jin.
\newblock {\em Inverse {P}roblems: Tikhonov {T}heory and {A}lgorithms},
  volume~22 of {\em Series on Applied Mathematics}.
\newblock World Scientific Publishing Co. Pte. Ltd., Hackensack, NJ, 2015.

\bibitem{ItoKunisch:2008}
K.~Ito and K.~Kunisch.
\newblock {\em Lagrange {M}ultiplier {A}pproach to {V}ariational {P}roblems and
  {A}pplications}.
\newblock SIAM, Philadelphia, PA, 2008.

\bibitem{JinKhanMaass:2012}
B.~Jin, T.~Khan, and P.~Maass.
\newblock A reconstruction algorithm for electrical impedance tomography based
  on sparsity regularization.
\newblock {\em Internat. J. Numer. Methods Engrg.}, 89(3):337--353, 2012.

\bibitem{JinMaass:2010}
B.~Jin and P.~Maass.
\newblock {A}n analysis of electrical impedance tomography with applications to
  {T}ikhonov regularization.
\newblock {\em ESAIM: Control, Optim. Calc. Var.}, 18(4):1027--1048, 2012.

\bibitem{JinMaass:2012}
B.~Jin and P.~Maass.
\newblock Sparsity regularization for parameter identification problems.
\newblock {\em Inverse Problems}, 28(12):123001, 70 pp., 2012.

\bibitem{KaltenbacherKirchnerVeljovic:2014}
B.~Kaltenbacher, A.~Kirchner, and S.~Veljovi{\'c}.
\newblock Goal oriented adaptivity in the {IRGNM} for parameter identification
  in {PDE}s: {I}. reduced formulation.
\newblock {\em Inverse Problems}, 30(4):0450011, 26, 2014.

\bibitem{KnudsenLassasMueller:2009}
K.~Knudsen, M.~Lassas, J.~L. Mueller, and S.~Siltanen.
\newblock Regularized {D}-bar method for the inverse conductivity problem.
\newblock {\em Inverse Probl. Imaging}, 3(4):599--624, 2009.

\bibitem{Koss:1995}
I.~Kossaczk\'{y}.
\newblock {A} recursive approach to local mesh refinement in two and three
  dimensions.
\newblock {\em J. Comput. Appl. Math.}, 55:275--288, 1995.

\bibitem{LechleiterHyvonen:2008}
A.~Lechleiter, N.~Hyv{\"o}nen, and H.~Hakula.
\newblock The factorization method applied to the complete electrode model of
  impedance tomography.
\newblock {\em SIAM J. Appl. Math.}, 68(4):1097--1121, 2008.

\bibitem{LechleiterRieder:2006}
A.~Lechleiter and A.~Rieder.
\newblock Newton regularizations for impedance tomography: a numerical study.
\newblock {\em Inverse Problems}, 22(6):1967--1987, 2006.

\bibitem{LiXieZou:2011}
J.~Li, J.~Xie, and J.~Zou.
\newblock {A}n adaptive finite element reconstruction of distributed fluxes.
\newblock {\em Inverse Problems}, 27(7):075009, 25pp, 2011.

\bibitem{LiLiuMaTang:2002}
R.~Li, W.~Liu, H.~Ma, and T.~Tang.
\newblock {A}daptive finite element approximation for distributed elliptic
  optimal control problems.
\newblock {\em SIAM J. Control Optim.}, 41(5):1321--1349, 2002.

\bibitem{LiuYan:2001}
W.~Liu and N.~Yan.
\newblock {A} posteriori error estimates for distributed convex optimal control
  problems.
\newblock {\em Adv. Comput. Math.}, 15(1-4):285--309, 2001.

\bibitem{Mitchell:1989}
W.~F. Mitchell.
\newblock A comparison of adaptive refinement techniques for elliptic problems.
\newblock {\em ACM Trans. Math. Software}, 15(4):326--347 (1990), 1989.

\bibitem{NSV:2009}
R.~H. Nochetto, K.~G. Siebert, and A.~Veeser.
\newblock {T}heory of adaptive finite element methods: an introduction.
\newblock In R.~A. DeVore and A.~Kunoth, editors, {\em Multiscale, Nonlinear
  and Adaptive Approximation}, pages 409--542. Springer, New York, 2009.

\bibitem{ScottZhang:1990}
L.~R. Scott and S.~Zhang.
\newblock Finite element interpolation of nonsmooth functions satisfying
  boundary conditions.
\newblock {\em Math. Comp.}, 54(190):483--493, 1990.

\bibitem{Siebert:2011}
K.~G. Siebert.
\newblock {A} convergence proof for adaptive finite elements without lower
  bounds.
\newblock {\em IMA J. Num. Anal.}, 31(3):947--970, 2011.

\bibitem{SomersaloCheneyIsaacson:1992}
E.~Somersalo, M.~Cheney, and D.~Isaacson.
\newblock {E}xistence and uniqueness for electrode models for electric current
  computed tomography.
\newblock {\em SIAM J. Appl. Math.}, 52(4):1023--1040, 1992.

\bibitem{Ver:1996}
R.~Verf\"{u}rth.
\newblock {\em {A} {R}eview of {A} {P}osteriori {E}stimation and {A}daptive
  {M}esh-{R}efinement {T}echniques}.
\newblock Wiley-Teubner, Chichester, New York, Stuttgart, 1996.

\bibitem{WinklerRieder:2014}
R.~Winkler and A.~Rieder.
\newblock Resolution-controlled conductivity discretization in electrical
  impedance tomography.
\newblock {\em SIAM J. Imaging Sci.}, 7(4):2048--2077, 2014.

\bibitem{XuZou:2013b}
Y.~Xu and J.~Zou.
\newblock Analysis of an adaptive finite element method for recovering the
  {R}obin coefficient.
\newblock {\em SIAM J. Control Optim.}, 53(2):622--644, 2015.

\bibitem{XuZou:2015a}
Y.~Xu and J.~Zou.
\newblock {C}onvergence of an adaptive finite element method for distributed
  flux reconstruction.
\newblock {\em Math. Comp.}, 84(296):2645--2663, 2015.

\end{thebibliography}

\end{document}